\documentclass[EJP]{mihai_modified_ejp} 

\SHORTTITLE{Multi-layer intermediate disorder directed polymers} 

\TITLE{Intermediate disorder directed polymers and the multi-layer extension of the stochastic heat equation} 

\AUTHORS{%
  Ivan~Corwin\footnote{Columbia University and the Clay Mathematics
Institute
    \EMAIL{ivan.corwin@gmail.com}}
  \and 
  Mihai~Nica\footnote{New York University \EMAIL{nica@cims.nyu.edu}}}


\AMSSUBJ{60F05, 82C05} 

\SUBMITTED{January 18, 2017} 
\ACCEPTED{January 29, 2017} 

\ARXIVID{1603.08168} 

\VOLUME{0}
\YEAR{2012}
\PAPERNUM{0}
\DOI{vVOL-PID}

\ABSTRACT{We consider directed polymer models involving multiple non-intersecting random walks moving through a space-time disordered environment in one spatial dimension. For a single random walk, Alberts, Khanin and Quastel proved that under intermediate disorder scaling (in which time and space are scaled diffusively, and the strength of the environment is scaled to zero in a critical manner) the polymer partition function converges to the solution to the stochastic heat equation with multiplicative white noise. In this paper we prove the analogous result for multiple non-intersecting random walks started and ended grouped together. The limiting object now is the multi-layer extension of the stochastic heat equation introduced by O'Connell and Warren.}

\global\long\def\p{\mathbb{P}}

\global\long\def\Y{\vec{Y}}
\global\long\def\X{\vec{X}}
\global\long\def\iX{{\vec{X}^{\prime}}}

\global\long\def\y{\vec{y}}
\global\long\def\x{\vec{x}}
\global\long\def\e{\mathbb{E}}
\global\long\def\one{\mathbf{1}}
\global\long\def\B{\vec{B}}
\global\long\def\S{\vec{S}}
\global\long\def\D{\vec{D}}

\global\long\def\xf{{x^{\star}}}
\global\long\def\vde{\vec{\delta}_{d}}
\global\long\def\zf{{z^{\star}}}
\global\long\def\tf{{t^{\star}}}
\global\long\def\nf{{n^{\star}}}
\global\long\def\mo{{\scriptscriptstyle -1}}

\global\long\def\norm#1{\left\Vert #1\right\Vert }
\global\long\def\abs#1{\left|#1\right|}
\global\long\def\given#1{\big|#1\big.}

\global\long\def\floor#1{\left\lfloor #1\right\rfloor }

\global\long\def\bN{\mathbb{N}}

\global\long\def\bR{\mathbb{R}}
\global\long\def\bS{\mathbb{S}}
\global\long\def\bT{\mathbb{T}}

\global\long\def\bW{\mathbb{W}}

\global\long\def\bZ{\mathbb{Z}}

\global\long\def\cB{\mathcal{B}}
\global\long\def\cC{\mathcal{C}}

\global\long\def\cE{\mathcal{E}}
\global\long\def\cF{\mathcal{F}}

\global\long\def\cL{\mathcal{L}}

\global\long\def\cP{\mathcal{P}}

\global\long\def\cS{\mathcal{S}}

\global\long\def\cV{\mathcal{V}}

\global\long\def\cZ{\mathcal{Z}}

\global\long\def\sgn{\text{sgn}}
\global\long\def\d{\text{d}}

\global\long\def\To{\Rightarrow}
\global\long\def\half{\frac{1}{2}}
\global\long\def\nogt{\,0}
\global\long\def\dfac{\al}
\global\long\def\oo#1{\frac{1}{#1}}

\global\long\def\al{\alpha}
\global\long\def\be{\beta}
\global\long\def\ga{\gamma}
\global\long\def\Ga{\Gamma}
\global\long\def\de{\delta}
\global\long\def\De{\Delta}
\global\long\def\ep{\epsilon}
\global\long\def\ze{\zeta}
\global\long\def\et{\eta}
\global\long\def\th{\theta}

\global\long\def\la{\lambda}
\global\long\def\La{\Lambda}
\global\long\def\rh{\rho}
\global\long\def\si{\sigma}

\global\long\def\ta{\tau}
\global\long\def\ph{\phi}

\global\long\def\ps{\psi}
\global\long\def\Ps{\Psi}
\global\long\def\om{\omega}
\global\long\def\Om{\Omega}

\global\long\def\dequal{\stackrel{d}{=}}

\global\long\def\ld{\ldots}

\global\long\def\defequal{\stackrel{{\scriptscriptstyle \Delta}}{=}}
\global\long\def\zmtwo{\equiv0\text{(mod }2)}
\global\long\def\vE{\cE}
\global\long\def\vVar{\cV ar}

\begin{document}


\tableofcontents{}
\section{Introduction}
Last passage percolation (LPP) involves finding the maximal sum of iid weights along random walk trajectories. Based on rigorous results for a few solvable choices of weights (e.g. exponential weights \cite{Joh01}), it is widely conjectured that for generic weight distributions the fluctuations of this maximal sum grows like the cube-root of time, and has a limit under this scaling described by the GUE Tracy-Widom distribution. This distribution owes its name to the fact that it also arises as the fluctuations of the largest eigenvalue of an $N\times N$ Gaussian Unitary Ensemble (GUE) matrix, as $N$ goes to infinity \cite{TW94}.
 
The asymptotics for the second, third, and so on eigenvalues also come up in LPP when one considers the maximal sum of weights along two, three, and so on non-intersecting trajectories which start and end clumped together \cite{Borodin00}. Again, this has only been shown for the solvable LPP models. By varying the endpoints of the random walks suitably, one encounters the Airy line ensemble \cite{Spohn02,Joh03,Corwin13}, which also arises from the asymptotics of Dyson's Brownian motion near its edge. Therefore, a strengthened version of the LPP conjecture would be that the Airy line ensemble arises in this manner for arbitrary choices of weights. This conjecture can be considered a form of the conjecture that LPP with general weights is in the Kardar-Parisi-Zhang universality class \cite{CorwinKPZ,QuastelKPZ}
 
In this paper we prove an {\it intermediate disorder} version of this conjecture. We consider the positive temperature version of LPP (called directed polymers) whereby one essentially transforms the $(\max,+)$ algebra to $(+,\times)$. In other words, one considers sums over non-intersecting random walks of the products of weights along their trajectories. These are often called polymer partition functions. At fixed positive temperature it is still conjectured that the Airy line ensemble arises in the exact same manner as for LPP. As before, this is only proved for certain solvable models and even then only in terms of the one-point marginal of the single path partition function – see e.g. \cite{corwin2014,Borodin15,Ortmann15,Corwin15,Barra16}. (See also \cite{update_a} for a two-point marginal formula which has not yet yielded to asymptotics, as well as \cite{update_b,update_c,update_d} for non-rigorous physics work regarding multi-point Airy asymptotics.)

  In this paper we scale time and space diffusively as well as simultaneously weaken the strength of the disorder (this is known as intermediate disorder, or weak noise scaling). Under this scaling we prove convergence to the KPZ line ensemble \cite{CorHam15} for general choices of weight distributions. (In fact our result is stated in terms of the closely related O'Connell and Warren multi-layer extension to the solution of the stochastic heat equation \cite{OConnellWarren2015}.) This main result is stated as Theorem \ref{thm:partition_result}. In the case of a single random walk path this result was proved by Alberts, Khanin and Quastel \cite{AKQ_2014} and the convergence was to the top curve of the O'Connell and Warren multi-layer extension which is simply the solution to stochastic heat equation (whose logarithm is the solution to KPZ equation).
 
It is quite intuitive to see why Theorem \ref{thm:partition_result} should hold. Under diffusive scaling, non-intersecting random walks started and ended grouped together converge to non-intersecting Brownian bridges (sometimes called Brownian watermelons). Under the same space-time scaling, a field of iid random weights converges when their strength is simultaneously scaled to zero (in a critical manner) to space-time Gaussian white noise. Therefore one expects the limits in question should be given by the average with respect to the Brownian watermelon measure of the exponential of the integral of white noise along these trajectories. Such objects require some work to be made sense of, but that is essentially what O'Connell and Warren defined. See Remark \ref{rem:exponential} for more on this.
 
The result of Alberts, Khanin and Quastel for a single random walk polymer partition function relies on writing a discrete chaos series and then proving convergence of each term (with control over the tail of the series) to the corresponding Gaussian chaos series for the stochastic heat equation. Caravenna, Sun and Zygouras \cite{caravenna_sun_zygouras_2015polynomial} provided a more general formulation of the approach of Alberts, Khanin and Quastel. In proving our main result, we appeal to this general formulation. Consequently, the proof of Theorem \ref{thm:partition_result} boils down to proving convergence of the correlation functions for non-intersecting random walks to those of non-intersecting Brownian motions. This is the main technical result of this paper and is presented in Theorem \ref{thm:L2_convergence_psiN_to_psi}. Pointwise versions of this convergence are present previously in the literature \cite{Johansson2005-Hahn}. However, to apply the results of Caravenna, Sun and Zygouras we must show $L^2$ convergence (with respect to space and time). The fact that the starting and ending points are grouped together introduces some challenging technical impediments in proving this fact and requires us to use methods beyond those employed in the pointwise limits. See Remark \ref{rem:ptwise_vs_l2} for more discussion on this.
   
One of our main motivations for the present investigation comes from the desire to better understand the properties of O'Connell and Warren's multi-layer extension to the solution of the stochastic heat equation. In \cite{CorHam15}, Corwin and Hammond considered a semi-discrete directed polymer model and showed that under the same intermediate disorder scaling considered here, the associated line ensemble is tight. They defined any subsequential limit as a KPZ line ensemble and conjectured that there is a unique such limit which can be identified with O'Connell and Warren's multi-layer extension. What was missing, their Conjecture 2.17, was exactly the analog of the convergence result which we provide herein in the case of discrete polymers. We expect that similar methods as developed here can be imported into that semi-discrete setting to prove the conjecture. (Essentially one must control the $L^2$ convergence of non-intersecting Poissonian walks rather than simple symmetric random walks.) 
 
The KPZ line ensembles enjoy a certain resampling invariance called a Brownian Gibbs property. In \cite{CorHam15} this was due to the existence of a corresponding property for the prelimiting semi-discrete directed polymer model which was shown to hold by O'Connell by utilizing a continuous version of the geometric lifting of the Robinson-Schensted correspondence \cite{Oconnell12}. On the discrete polymer side, Sepp\"{a}l\"{a}inen's log-gamma polymer \cite{corwin2014} likewise enjoys a discrete Gibbs resampling property which essentially follows from the work of Corwin, O'Connell, Sepp\"{a}l\"{a}inen and Zygouras \cite{corwin2014} by likewise utilizing the geometric lifting of the Robinson-Schensted-Knuth correspondence. Our results here also apply to the geometric RSK corresponence when the weights are critically scaled, which we present in Theorem \ref{thm:RSK_result}. Corollary \ref{cor:log_gamma_result} shows how this applies exactly to the log-gamma polymer. This requires a bit of an argument simply because the log-gamma distributions do not come with an inverse temperature and instead one must identify an effective inverse temperature parameter.
 
There is also motivation for the study of directed polymers coming from various directions in physics. The general class was introduced to study domain walls of Ising type models with impurities \cite{BBC55,BBC69}, and also applied to study vortices in superconductors \cite{BBC13}, roughness of crack interfaces \cite{BBC52}, Burgers turbulence \cite{BBC47}, and interfaces in competing bacterial colonies \cite{BBC50} (see also the reviews \cite{BBC51} or \cite{BBC45} for more applications). Directed polymers with many non-intersecting paths was studied in \cite{Kardar87} as a model for two-dimensional random interfaces subject to disorder. Recently there has also been interest \cite{DeLuca15} in studying the probability of non-cross for independent polymers in the same environment. This naturally leads to the type of generalized multi-path polymers considered herein.

\subsection{Conventions} Let $\bN=\{1,2,\ld\}$. We use $\defequal$ for definitions. We use the letters $t\in(0,\infty)$, $z\in\bR$ to denote continuous time and space and the letters $n\in\bN$, $x\in\bZ$ to denote discrete time and space. We use the vector symbol $\vec{\cdot}$ to denote vectors and use subscripts for their components, e.g. $\vec{v}=(v_1,\ld,v_j)$. We often use $\vec{w}$ as a variable in $k$-fold space-time integrals:
\[
\intop_{A} f(\vec{w})\d \vec{w} \defequal \displaystyle \idotsint_{\left\{(t_1,z_1),\ld,(t_k,z_k)\right\}\in A} f\big((t_1,z_1),\ld,(t_k,z_k)\big)\d t_1 \ld \d t_k \d z_1 \ld \d z_k
\] 
We use the superscript $\star$ to denote the endpoint of polymers; for example $(\tf,\zf)$ denotes the endpoint of non-intersecting Brownian bridges, and $(\nf,\xf)$ will denote the endpoint of non-intersecting random walk bridges. 

For the rest of the article we always use $d\in\bN$ to denote the number of random walks/Brownian motions in the non-intersecting ensembles
we consider. We think of this as fixed throughout the paper.

\subsection{Main results} The continuum partition function of $d$ non-intersecting Brownian bridges in a space-time white noise environment was first introduced and studied by O'Connell and Warren in \cite{OConnellWarren2015} in connection to the multi-layer extension of the stochastic heat equation. 

\begin{definition}[Continuum partition function]
 \label{def:Z_partition}   Fix
any $\tf>0$ and $\zf\in\bR$. Let $\vec{D}^{(\tf,\zf)}(t)\in\bR^{d},t\in(0,\tf)$
denote the stochastic process of $d$ non-intersecting Brownian bridges
which start at $\vec{D}^{(\tf,\zf)}(0)=(0,0,\ld,0)$ and end at $\vec{D}^{(\tf,\zf)}(\tf)=(\zf,\zf,\ld,\zf)$; see Figure \ref{fig:BrownianBridgeCap} for an example of sample paths of $\vec{D}^{(\tf,\zf)}$ and Definition \ref{def:NIBb} for more details. We define the following Wiener chaos series:
\begin{equation}
\cZ_{d}^{\be}(\tf,\zf) \defequal \rh(\tf,\zf)^{d} \sum_{k=0}^{\infty}\be^{k}\intop_{\De_k(0,t)}\intop_{\bR^k}\ps_{k}^{(\tf,\zf)}\big((t_{1},z_{1}),{\ld},(t_{k},z_{k})\big)\xi(\d t_{1},\d z_{1}){\cdots}\xi(\d t_{k},\d z_{k}),\label{eq:cZ_def}
\end{equation}
where $\xi$ denotes space time white noise (see \cite{janson97,stochPDE,AKQ_2014,caravenna_sun_zygouras_2015polynomial} for the background on $k$-fold white noise integrals) and where
\begin{eqnarray*}
\De_k(s,s^\prime) &\defequal & \left\{\vec{t}\in(0,\infty)^k : s<t_1<\cdots<t_k<s^\prime \right\},\\
\rh(t,z) &\defequal& (2\pi t)^{-\half} \exp\left(-{z^2}/{2t} \right),
\end{eqnarray*}
and the functions $\ps_{k}$ are the $k$-point correlation functions for $\vec{D}^{(\tf,\zf)}$; see Definition \ref{def:NIBb} for the precise definition of $\ps_{k}^{(\tf,\zf)}$. When $d=1$, $\cZ_1^\be(\tf,\zf)$ is the continuum directed polymer studied in \cite{Alberts2013} and is shown to be a solution of the stochastic heat equation with multiplicative noise. For $d>1$, $\cZ_d^\be(\tf,\zf)$ was first studied in \cite{OConnellWarren2015}. (Note that in their original definition, they fixed $\be = 1$) \end{definition}
\begin{figure}[h] \label{fig:BrownianBridge}
\begin{center}
\includegraphics[width=\textwidth]{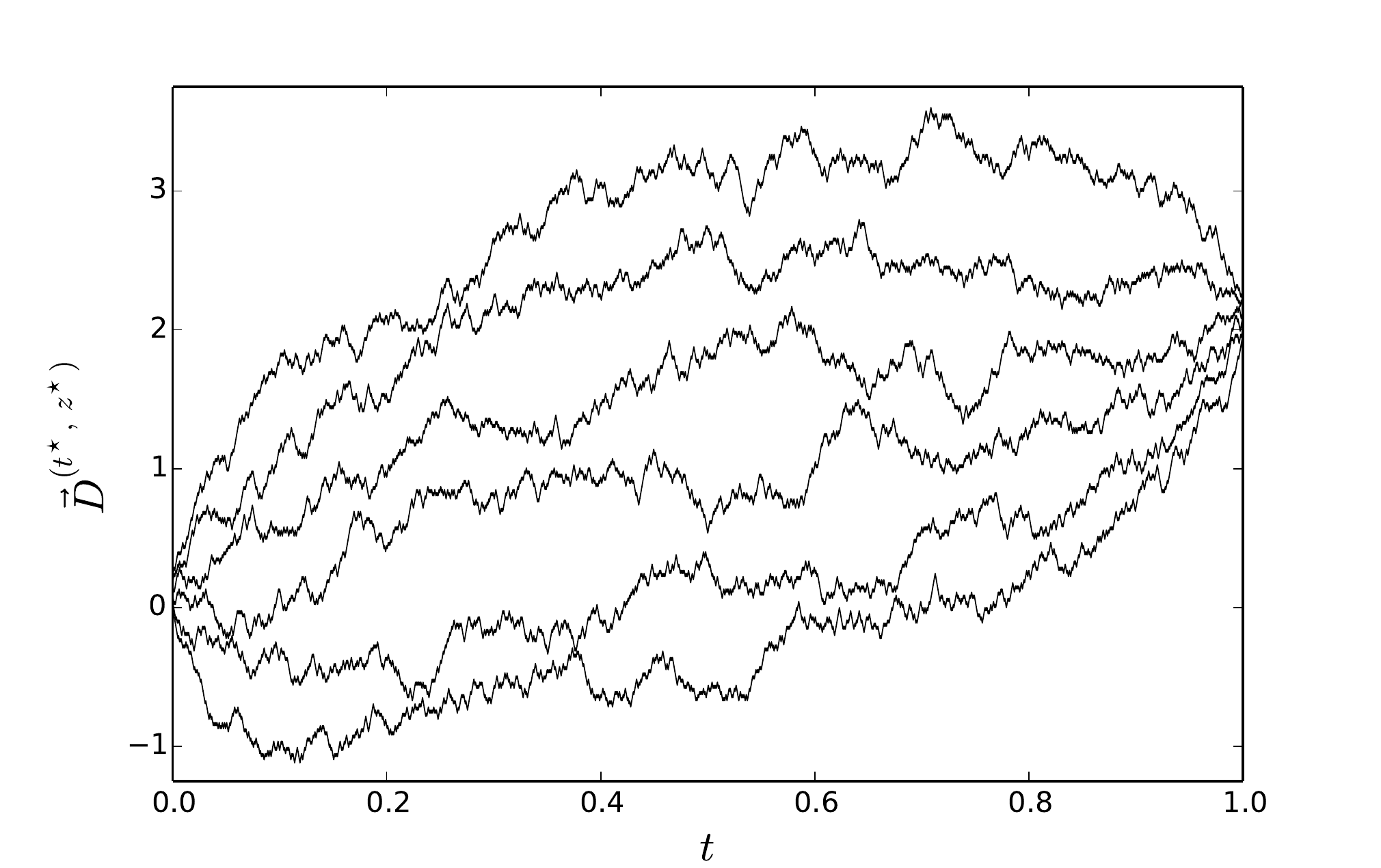}
\end{center}
\caption{\label{fig:BrownianBridgeCap} Sample path of the non-intersecting Brownian bridges $\vec{D}^{(\tf,\zf)}(t)$ with $d=6$, $\tf=1.0$, $\zf=2.0$.}
\end{figure}

\begin{proposition}
(Theorem 1.2. of \cite{OConnellWarren2015}) The series in equation
(\ref{eq:cZ_def}) that defines $\cZ_{d}^{\be}$ is convergent in
$L^{2}(\xi)$.
\end{proposition}
\begin{remark}\label{rem:exponential} 
The chaos series defining $\cZ_d^\be$ has a continuum path integral interpretation
 \begin{equation}\label{eq:cZ_is_exp}
 \cZ_d^\be(\tf,\zf)=\rho(\tf,\zf)^d \e\bigg[:\exp:\Big\{\be \sum_{j=1}^{d}\intop_{0}^{\tf} \xi\big(t,D^{(\tf,\zf)}_j(t)\big)dt\Big\}\bigg],
 \end{equation}
where $:\exp:$ denotes the Wick exponential, see \cite{janson97,stochPDE,caravenna_sun_zygouras_2015polynomial,Alberts2013} for background. There is work necessary to fully understand this since Brownian motion is not smooth enough to make sense of the white noise integral in equation (\ref{eq:cZ_is_exp}). When $d=1$, it is shown in  \cite{Bertini95} that $\cZ_1^\be$ can be recovered by renormalization as a limit of integration against smoothed noise. The case $d>1$ has not yet been treated.
\end{remark}
In this article we prove that $\cZ_{d}^{\be}$ arises
in the intermediate disorder scaling limit for the partition
function for a particular ensemble of $d$ non-intersecting random walks in a disordered
environment. We also show that $\cZ_d^\be$ appears in a scaling limit for the $d$-th row
of the geometric RSK correspondence when applied to suitably rescaled random weights. These are the statements of Theorems
\ref{thm:partition_result} and \ref{thm:RSK_result} which are both
based on the main technical result of this article Theorem \ref{thm:L2_convergence_psiN_to_psi}.
Corollary \ref{cor:log_gamma_result} contains an application of Theorem \ref{thm:RSK_result} in the case of the
exactly solvable log-gamma weights in \cite{seppalainen2012,corwin2014} (with parameters of the log-gamma weights scaled like $\ga\sim\sqrt{N}$.) 
\begin{definition}
\label{def:partition_function}Fix any $\nf\in\bN$ and $\xf\in\bZ$ so that $\nf+\xf\zmtwo$. Let $\vec{X}^{(\nf,\xf)}(n)\in\bZ^{d}$, $n\in[0,\nf]\cap\bN$
denote an ensemble of $d$ non-intersecting simple symmetric random
walks conditioned to start at $\vec{X}^{(\nf,\xf)}(0)=\left(0,2,\ld,2d-2\right)$,
end at $\vec{X}^{(\nf,\xf)}(\nf)=(\xf,\xf+2,\ld,\xf+2d-2)$
and never intersect for all $n\in[0,\nf]\cap\bN$. See Figure \ref{fig:LatticePathsCap} for an example of a sample path and Definition \ref{def:NIWb} for more details. Denote by $\e$ the expectation for this process. 

Let $\om=\left\{ \om(n,x)\right\} _{n\in\bN,x\in\bZ}$ be an iid collection
of random variables which we think of as a disordered environment.
We will denote by $\vE$ the expectation with respect to this disorder.
Define the energy of the ensemble $\vec{X}^{(\nf,\xf)}$ in the environment
$\om$ by:
\[
H^{\om}\left(\vec{X}^{(\nf,\xf)}\right)=\sum_{j=1}^{d}\sum_{n=1}^{\nf-1}\om\left(n,X_{j}^{(\nf,\xf)}(n)\right).
\]
Define the partition function at inverse temperature $\be>0$ for
the environment $\om$ by: 
\[
Z_{d}^{\be}\left(\nf,\xf\right)=\e\left[\exp\left(\be H^{\om}\big(\vec{X}^{(\nf,\xf)}\big)\right)\right].
\]
\end{definition}
We now state our main result.
\begin{theorem}
\label{thm:partition_result} Fix any $\tf>0$ and $\zf\in\bR$. Let \textup{$\om=\left\{ \om(n,x)\right\} _{x\in\bZ,n\in\bN}$}
be any iid environment of mean zero, unit variance random variables
with finite exponential moments $\La(\be)\defequal\log\left(\vE\left(e^{\be\om(0,0)}\right)\right)<\infty$.
For any $\be>0$, let $\be_{N}\defequal N^{-\oo 4}\be$.
We have the following convergence in distribution of the partition
functions in the limit $N\to\infty$.
\begin{equation}
Z_{d}^{\be_{N}}\left(\big(N\tf,\sqrt{N}\zf\big)_{2}\right)\exp\Big(-dN\tf\La\big(\be_{N}\big)\Big) \To \frac{\cZ_{d}^{\sqrt{2}\be}\left(\tf,\zf\right)}{\rh(\tf,\zf)^{d}}.\label{eq:partition_fucntion_convergence}
\end{equation}
where the notation $(\nf,\xf)=(N\tf,\sqrt{N}\zf)_{2}$ denotes the
lattice point nearest $(N\tf,\sqrt{N}\zf)$ which has $\xf+\nf\zmtwo$
(see Definition \ref{def:T_N_and_rounding} for more details on this notation).\end{theorem}
\begin{remark}
When $d=1$, equation (\ref{eq:partition_fucntion_convergence}) is
the convergence of the partition function for a simple symmetric random
walk to the Wiener chaos solution of the stochastic heat equation
first proven in \cite{AKQ_2014}. 
\end{remark}
\begin{figure}[h] \label{fig:LatticePaths}
\begin{center}
\includegraphics[width=\textwidth]{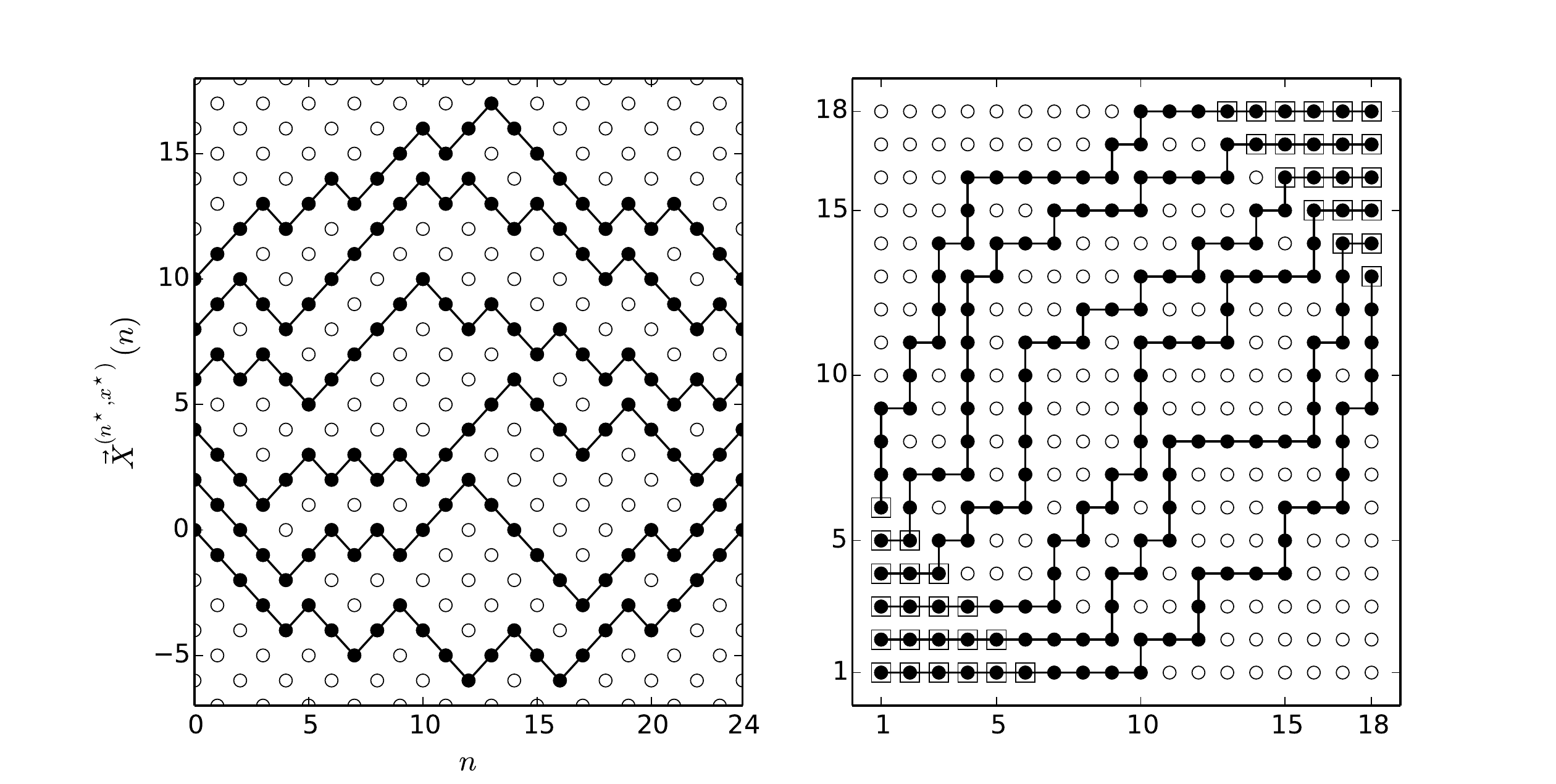}
\end{center}
\caption{\label{fig:LatticePathsCap} (Left) A sample path of the non-intersecting random walks $\vec{X}^{(\nf,\xf)}$ with $d=6$, $\nf=24$, $\xf=0$. Notice that only points with $x+n \zmtwo$ are accessible. (Right) An example of the non-intersecting lattice paths from $\Pi_{n,m}^d$ when $d=6$, $n=18$, $m=18$. The points that have a square around them are forced; they appear in \emph{all} the paths in $\Pi_{n,m}^d$. The left and right plots are related by rotation by $45^{\circ}$.}
\end{figure}

A related construction is the geometric RSK correspondence applied to a matrix of positive weights. We also have a limit theorem for this object.

\begin{definition}
\label{def:RSK}(From Section 2.2 in \cite{corwin2014}) Fix any $\left(n,m\right)\in\bN^{2}$. Let $\Pi_{n,m}^{d}$ denote the
set of $d$-tuples $\pi=\left(\pi_{1},\ld,\pi_{d}\right)$ of non-intersecting
lattice paths in $\left\{ 1,\ld,n\right\} \times\left\{ 1,\ld,m\right\} $
such that for $1\leq r\leq d$, $\pi_{r}$ is lattice path from $(1,r)$
to $(n,m+r-d)$. A ``lattice path'' is a path that takes unit steps
in the coordinate directions (i.e. up or right) between nearest-neighbor
lattice points of $\bN^{2}$; non-intersecting means the paths never
touch. See Figure \ref{fig:LatticePathsCap} for an example.  

Let $g=\left\{ g_{ij}\right\} _{\left(i,j\right)\in\bN^{2}}$ be
an infinite matrix of positive real weights. The weight of a tuple
$\pi=\left(\pi_{1},\ld,\pi_{d}\right)$ of such paths in the environment
$g$ is defined by:
\[
wt(\pi) \defequal \prod_{r=1}^{d}\prod_{(i,j)\in\pi_{r}}g_{ij}.
\]
For $1\leq d\leq\min(n,m)$ define:
\[
\ta_{m,d}(n) \defequal \sum_{\pi\in\Pi_{n,m}^{d}}wt(\pi).
\]
These are used to define the elements of the geometric RSK array $\left\{ z_{m,j}\right\} _{m\in\bN,j\in\bN}$
by the prescription that $z_{m,1}(n)\cdot z_{m,2}(n)\cdots z_{m,d}(n)=\ta_{m,d}(n).$

\end{definition}

\begin{theorem}
\label{thm:RSK_result}Fix some $\be>0$. For each $N\in\bN$,
suppose we are given $g^{(N)}=\left\{ g_{ij}^{(N)}\right\} _{(i,j)\in\bN^{2}}$
an iid collection of positive real weights with finite second moment.
We use $\vE$ and $\vVar$ to denote expectation and variance of these
weights. Assume that:
\[
\lim_{N\to\infty}N^{\half}\frac{\vVar\left(g_{1,1}^{(N)}\right)}{\vE\left[g_{1,1}^{(N)}\right]^{2}}=\be.
\]
Assume also that the sequence of random variables $\left\{ \big( g^{(N)}_{1,1} - \cE( g^{(N)}_{1,1} ) \big)^2 \right\}$ is uniformly integrable. For $\left(n,m\right)\in\bN^{2}$, let $\ta_{m,d}^{(N)}(n)$ be defined
as in Definition \ref{def:RSK} using the weights $g^{(N)}$ as input.
Then we have convergence in distribution:
\begin{equation}
\frac{\left(\vE\left[g_{1,1}^{(N)}\right]\right)^{-d(2N+d)}}{2^{d(2N+d)}N^{-\half d^{2}}\big({\textstyle \prod_{j=0}^{d-1}j!}\big)}\ta_{N+d,d}^{(N)}(N+d)\To\cZ_{d}^{\sqrt{2}\be}(2,0).\label{eq:RSK_eq}
\end{equation}
\end{theorem}

\begin{remark}
The proof of Theorem \ref{thm:RSK_result} can be extended to give limit theorems for $\ta_{m(N),d}(n(N))$ when $m(N) \to \infty$, $n(N)\to \infty$ and $\frac{m(N)}{n(N)} = 1+O(N^{-\half})$ as $N \to \infty$. 
\end{remark}
\begin{definition}
For $\th>0$. A random variable $g$ has the inverse-Gamma distribution
with parameter $\th>0$ if it is supported on the positive reals where
it has the density:
\[
\p\left(g\in\d x\right)=\frac{1}{\Ga\left(\th\right)}x^{-\th-1}\exp\left(-\frac{1}{x}\right)\d x.
\]
We denote this by $g\sim\Ga^{-1}(\th)$. An elementary calculation
shows that for $\th>1$, $\vE(g)=\left(\th-1\right)^{-1}$ and for
$\th>2$, $\vVar(g)=\left(\th-1\right)^{-2}\left(\th-2\right)^{-1}$\end{definition}
\begin{corollary}
\label{cor:log_gamma_result}If $g^{(N)}\sim\Ga^{-1}\left(\be^{-1}N^{\half}\right)$
then Theorem \ref{thm:RSK_result} applies and we have:
\[
\frac{ \be^{-d(2N+d)}} { 2^{d(2N+d)}N^{dN}\big({\textstyle \prod_{j=0}^{d-1}j!}\big)}   \ta_{N+d,d}^{(N)}(N+d)\To\cZ_{d}^{\sqrt{2}\be}(2,0).
\]

\end{corollary}
\begin{remark}
The log-gamma weights are special in that they lead to an exactly solvable polymer model related to Whittaker measures which have been studied in \cite{seppalainen2012,borodin2013,corwin2014}. Corollary \ref{cor:log_gamma_result} can be interpreted as a limit law for these Whittaker measures.
\end{remark}
The main technical result needed for Theorems
\ref{thm:partition_result} and \ref{thm:RSK_result} is $L^2$ convergence of the correlation functions of the non-intersecting random
walk bridges $\X^{(\nf,\xf)}$ to those of the non-intersecting Brownian bridges $\vec{D}^{(\tf,\zf)}$ under the diffusive scaling $\left(\nf,\xf\right)\approx\big(N\tf,\sqrt{N}\zf\big)$.
\begin{theorem}
\label{thm:L2_convergence_psiN_to_psi} Fix any $\tf>0$ and $\zf\in\bR$. For all $k\in\bN$, let $\ps_{k}^{(N),(\tf,\zf)}:\big((0,\tf)\times\bR\big)^{k}\to\bR$
be the $k$-point correlation functions for $d$ non-intersecting
random walk bridges which are rescaled as in Definition \ref{def:scaled-NIWb}.
Let $\ps_{k}^{(\tf,\zf)}:\big((0,\tf)\times\bR\big)^{k}\to\bR$
be the $k$-point correlation function for $d$ non-intersecting Brownian
bridges given in Definition \ref{def:NIBb}. Then we have the following
convergence as $N\to\infty$:
\begin{equation}
\lim_{N\to\infty}\norm{\ps_{k}^{(N),(\tf,\zf)}-\ps_{k}^{(\tf,\zf)}}_{L^{2}\left(\big((0,\tf)\times\bR\big)^{k}\right)}=0.\label{eq:L2_thm}
\end{equation}
\end{theorem}
\begin{remark} \label{rem:ptwise_vs_l2}
The \emph{pointwise} convergence of these $k$-point correlation
functions has been observed in the special case that $\zf=0$ in Section
3.2 of \cite{Johansson2005-Hahn}. This is proven by explicitly writing
$\ps_{k}^{(N),(\tf,\zf)}$ as a $k\times k$ determinant of a kernel $K^{(N),(\tf,\zf)}$ and showing that this converges pointwise
to a corresponding kernel for $\ps_{k}^{(\tf,\zf)}$ (This is explained
in detail in Section \ref{sec:Orthogonal-Polynomials}.) The $L^{2}$ convergence is much harder because the conditioning
on non-intersection is singular near the endpoints
$t=0$ and $t=\tf$ in a way that gives rise to singularities in the determinantal kernel which are not square integrable for $d\geq2$. The fact that $\psi_k$ is indeed in square integrable seems to arise due to cancellations in the determinant which cancel away these singularities.

 The bulk of this article (Sections \ref{sec:overlap_times} and \ref{sec:L2_bounds})
is devoted to developing alternative techniques, which do not rely on the determinantal structure, to control any possible
singularity of $\ps^{(N),(\tf,\zf)}$ at $t=0$ and $t=\tf$. This
is very similar in spirit and inspired by the analysis used to justify
the convergence of the series defining $\cZ_{d}^{\be}$ which was
carried out in \cite{OConnellWarren2015}. The additional complexity
in our analysis can be attributed to the fact that the generator for
the non-intersecting random walk bridges $\vec{X}^{(\nf,\xf)}$ is more
complicated than the generator for the non-intersecting Brownian bridges
$\vec{D}^{(\tf,\zf)}$. (See Remark \ref{rem:X_is_more_complicated_than_D}
for more discussion about this.)
\end{remark}

%
%
\subsection{Outline}

Subsections \ref{sub:def_NIBb} and \ref{sub:NIWb} contain the precise
definitions relating to the non-intersecting Brownian bridges
and the non-intersecting random walks respectively. Subsection \ref{sub:poly_chaos}
contains the proof of Theorems \ref{thm:partition_result}
and \ref{thm:RSK_result} by showing how these fit into the general
framework of polynomial chaos series, and then applying a result from
\cite{caravenna_sun_zygouras_2015polynomial} about the convergence
of polynomial chaos series to Wiener chaos series. In Subsection
\ref{sub:L2}, the main technical result, Theorem \ref{thm:L2_convergence_psiN_to_psi},
is proven with the important estimates, Propositions \ref{prop:D1},
\ref{prop:D2}, \ref{prop:D3} and \ref{prop:D4}, deferred to later
sections.

Section \ref{sec:Orthogonal-Polynomials} uses the theory of determinantal
kernels and orthogonal polynomials to prove the pointwise convergence
in Proposition \ref{prop:D1} and the bound in Proposition \ref{prop:D2}.
Section \ref{sec:overlap_times} introduces ``overlap times'' which
are connected to the $L^{2}$ norm of $\ps_{k}^{(N),(\tf,\zf)}$.
The main result of this section is Proposition \ref{prop:ON_exp_mom}
which gives a particular type of control on the moments of the overlap
time. Section \ref{sec:L2_bounds} uses Proposition \ref{prop:ON_exp_mom}
as a tool to prove the estimates in Propositions \ref{prop:D3}, \ref{prop:D4}
and \ref{prop:uniform_exp_moms}.

\section{Definitions and Proof of Main Results}

We now define the main objects of study in detail and give the proofs of the
Theorems \ref{thm:partition_result}, \ref{thm:RSK_result} and \ref{thm:L2_convergence_psiN_to_psi}, deferring technical estimate to Sections \ref{sec:Orthogonal-Polynomials},
\ref{sec:overlap_times} and \ref{sec:L2_bounds}.

\subsection{\label{sub:def_NIBb}Non-intersecting Brownian motions and bridges} \begin{definition}[Non-intersecting Brownian motions]\label{def:NIB}
Let $\bW^d=\left\{ \vec{z}\in\bR^{d}:\ z_{1}\leq \ld \leq z_{d}\right\} $
be the $d$-dimensional Weyl chamber. We denote by $\vec{D}(t)\in\bW^d$, $t\in(0,\infty)$
an ensemble of $d$ Brownian motions conditioned not to intersect for all time, and use $\e_{\vec{z}^{\nogt}}\left[\cdot\right]$ to denote the expectation started from $\vec{D}(0)=\vec{z}^{\nogt}$. This is also
known as Dyson Brownian motion. $\vec{D}(t)$ is the Markov process which is obtained from
$d$ iid standard Brownian motions via a Doob $h$-transform by the Vandermonde determinant 
\[
h_{d}(\vec{z})\defequal\prod_{1\leq i<j\leq d}(z_{j}-z_{i}).
\]
(See Section 3 of \cite{Warren_Dyson_Brownian_Motions} for more details.) More precisely this is the prescription that, for $\vec{z}^{\nogt}\in\bW^d$ with $h_d\big(\vec{z}^{\nogt}\big)>0$ and for any continuous function $f:\bW^d\to\bR$, we have:
\begin{eqnarray*}
\e_{\vec{z}^{\nogt}}\left[f\big(\vec{D}(t)\big)\right] & \defequal & \frac{1}{h_{d}\left(\vec{z}^{\nogt}\right)}\e\left[f\big(\vec{B}(t)+\vec{z}^{\nogt}\big)h_{d}\big(\vec{B}(t)+\vec{z}^{\nogt}\big)\one\left\{ \ta_{\vec{z}^{\nogt}}>t\right\} \right]\\
\ta_{\vec{z}^{\nogt}} & \defequal & \inf\left\{ t>0:\vec{B}(t)+\vec{z}^{\nogt}\notin\bW^d\right\},
\end{eqnarray*}
where $\vec{B}(t)$ are $d$ iid standard Brownian motions started from $\vec{B}(0)=\vec{0}$. This definition requires some interpretation for starting points $\vec{z}^{\nogt}$
on the boundary of $\bW^d$ (i.e. where $z_{i}^{0}=z_{j}^{0}$ for some
$i\neq j$) because $h_{d}(\vec{z}^{\nogt})$ vanishes here. When the
starting point is $\vec{0}\defequal\left(0,0,\ld,0\right)\in\bW^d$,
this is carried out in Section 4 of \cite{O'ConellYor_Representation_for_NonColliding}
and the law of the GUE eigenvalues is obtained:
\begin{equation}
\e_{\vec{0}}\left[f(\vec{D}(t))\right]=\frac{1}{(2\pi)^{\half d}\prod_{j=1}^{d-1}j!}\intop_{\vec{z}\in\bW^d}f(z)\cdot t^{-\half d^{2}}\exp\left(-\sum_{j=1}^{d}\frac{z_{j}^{2}}{2t}\right)h_{d}(\vec{z})^{2}\d\vec{z}.\label{eq:law_of_D_from_0}
\end{equation}

\end{definition}

\begin{definition}[Non-intersecting Brownian bridges]
\label{def:NIBb} Fix $\tf>0$ and $\zf\in\bR$. As in Definition \ref{def:Z_partition}, we will denote by $\vec{D}^{(\tf,\zf)}(t)\in\bW^d$, $t\in[0,\tf]$ an ensemble of $d$
non-intersecting Brownian bridges that start at $\D^{(\tf,\zf)}(0)=(0,0\ld,0)$
and end at the point $\D^{(\tf,\zf)}(\tf)=(\zf,\zf,\ld\zf)$. This is sometimes
called a Brownian watermelon; see Figure \ref{fig:BrownianBridgeCap}. $\D^{(\tf,\zf)}$ is defined by starting with the process $\vec{D}(t)\in\bW^d$ from Definition \ref{def:NIB} and applying the general Markovian construction
of a bridge, see Proposition 1 of \cite{MarkovianBridges}. This construction
is carried out in detail in Section 2 of \cite{OConnellWarren2015}, see in particular Lemma 2.1.

Given any $k\in\bN$, indices $\vec{j}=\left(j_{1},\ld,j_{k}\right)\in\left\{ 1,\ld,d\right\} ^{k}$
and space-time coordinates $\big((t_{1},z_{1}),\ld,(t_{k},z_{k})\big)\in\big((0,\tf)\times\bR\big)^{k}$,
let $\rho_{j_{1},j_{2},\ld,j_{k}}\big((t_{1},z_{1}),\ld,(t_{k},z_{k})\big)$
denote the probability density of the random vector $\Big(D_{j_{1}}^{(\tf,\zf)}(t_{1}),\ld,D_{j_{k}}^{(\tf,\zf)}(t_{k})\Big)$
with respect to Lebesgue measure on $\bR^{k}$ evaluated at $(z_{1},\ld,z_{k})$.
The $k$-point correlation functions $\ps_{k}^{(\tf,\zf)}$
are defined for $k$-tuples $\big((t_{1},z_{1}),\ld,(t_{k},z_{k})\big)\in\big((0,\tf)\times\bR\big)^{k}$
where all the entries $(t_{i},z_{i})$ are distinct by
\[
\ps_{k}^{(\tf,\zf)}\big((t_{1},z_{1}),\ld,(t_{k},z_{k})\big)\defequal\sum_{\vec{j}\in\left\{ 1,\ld,d\right\} ^{k}}\rho_{j_{1},\ld,j_{k}}\big((t_{1},z_{1}),\ld,(t_{k},z_{k})\big),
\] 
and declaring that $\ps_{k}^{(N),(\tf,\zf)}\defequal0$ if any of
the space time coordinates are duplicated $(t_{i},z_{i})=(t_{j},z_{j})$
for $i\neq j$. Informally, this is the probability density of finding the position $z_{i}$ occupied at time $t_{i}$ by \emph{any} of the $d$ Brownian
bridges from the ensemble $\vec{D}^{(\tf,\zf)}$ for every $1\leq i\leq k$
with no regard as to which curve of the ensemble is in which location. 
This definition agrees with the definition of $R_{k}^{(d)}$ from
\cite{OConnellWarren2015}.
\end{definition}

\begin{proposition}
\label{prop:psi_k_L2}(Lemma 4.1 and Proposition 4.2. of \cite{OConnellWarren2015})
Fix any $\zf\in\bR,\tf>0$. For each $k\in\bN$, we have that $\ps_{k}^{(\tf,\zf)}\in L^{2}\Big(\big((0,\tf)\times\bR \big)^{k}\Big)$
and moreover for any $\be>0$, the following series is absolutely
convergent 
\[
1+\sum_{k=1}^{\infty}\frac{\be^{k}}{k!}\norm{\ps_{k}^{(\tf,\zf)}}_{L^{2}\left(\big((0,\tf)\times\bR\big)^{k}\right)}^{2}<\infty
\]

\end{proposition}

\begin{remark}
This result is used to justify the convergence of the Wiener chaos
series which defines $\cZ_{d}^{\be}(\tf,\zf)$ in equation (\ref{eq:cZ_def}).
\end{remark}

\subsection{\label{sub:NIWb}Non-intersecting random walks and non-intersecting
random walk bridges}
\begin{definition}[Non-intersecting random walks]
\label{def:NIW} Define the discrete period-2 Weyl chamber $\bW^d_{2}=\bW^d\cap\left\{ x\in\bZ^{d}:x_{i}+x_{j}\zmtwo \, \forall \, 1\leq i,j\leq d\right\} $.  We denote by $\X(n)\in\bW^d_{2}$, $n\in\bN$
an ensemble of $d$ non-intersecting simple symmetric random walks and use $\e_{\vec{x}^{0}}\left[\cdot\right]$
to denote the expectation from the initial
condition $\vec{X}(0)=\vec{x}^{0}$. More precisely, this is the Markov process one gets by conditioning iid simple symmetric random walks not to intersect
for all time via a Doob $h$-transform with the Vandermonde determinant $h_{d}$. (see \cite{OConnell_Roch_Konig_NonCollidingRandomWalks} for more details) The transition probabilities are
\[
\p\left(\X(n^{\prime})=\x^{\prime}\given{\X(n)=\x}\right)\defequal q_{n^{\prime}-n}\left(\x,\x^{\prime}\right)\frac{h_{d}\left(\x^{\prime}\right)}{h_{d}\left(\x\right)},
\]
where $q_{n}(\x,\y)$ is defined to be the probability for $d$ iid simple
symmetric random walks to go from $\x$ to $\y$ in time $n$ without
intersections. By the Karlin-MacGregor/Lindstrom-Gessel-Viennot theorem (see e.g. \cite{kmcg,LGV}),
this is given by
\[
q_{n}(\x,\y)=\det\left[\frac{1}{2^{n}}\binom{n}{\half\left(n+x_{i}-y_{j}\right)}\right]_{i,j=1}^{d},
\]
where we use the convention that $\binom{n}{x}$ is zero unless $x\in\bZ$
and $0\leq x\leq n$.

\end{definition}

\begin{definition}[Non-intersecting random walk bridges]
\label{def:NIWb} Fix any $\nf\in\bN$ and $\xf\in\bZ$ so that $\xf+\nf\zmtwo$.
For $x\in\bZ$, define $\vec{\de}_{d}(x)\defequal(x,x+2,\ld,x+2(d-1))\in\bW^d_{2}$.
As in Definition \ref{def:partition_function}, we will denote by $\vec{X}^{(\nf,\xf)}(n)\in\bW^d_{2},n\in[0,\nf]\cap\bN$
the ensemble of $d$ non-intersecting simple symmetric random walks
bridges that start at $\X^{(\nf,\xf)}(0)=\vec{\de}_{d}(0)$ and end
at $\X^{(\nf,\xf)}(\nf)=\vec{\de}_{d}(\xf)$ after $\nf$ steps. We
take the uniform measure on $\vec{X}^{(\nf,\xf)}$ over the finite
set of trajectories that satisfy these properties; equivalently one
could run $d$ i.i.d. simple symmetric walks started from $\vec{\de}_{d}(0)$
and then conditioning on the positive-probability event that at time
$\nf$ the walks are exactly at the final position $\vec{\de}_{d}(\xf)$
and that there have been no collisions between the walks at any
intermediate time $n\in[0,\nf]\cap\bN$. See Figure \ref{fig:LatticePathsCap} for an example of a sample path $\X^{(\nf,\xf)}$. By the Karlin-MacGregor/Lindstrom-Gessel-Viennot
theorem, one can explicitly write the transition probabilities for
this Markov process as

\[
\p\left(\X^{(\nf,\xf)}(n^{\prime})=\x^{\prime}\given{\X^{(\nf,\xf)}(n)=\x}\right)=\frac{q_{n^{\prime}-n}\big(\x,\x^{\prime}\big)q_{\nf-n^{\prime}}\big(\x^{\prime},\vec{\de}(\xf)\big)}{q_{\nf-n}\big(\x,\vec{\de}(\xf)\big)}\ \, \forall \, n<n^{\prime}.
\]
By comparing with Definition \ref{def:NIW}, we see that this
process is absolutely continuous with respect to the non-intersecting random
walks $\X(n)$ started from $\vec{X}(0)=\vec{\de}(0)$ on the interval
$n\in[0,\nf]\cap\bN$ with Radon-Nikodym derivative given by
\begin{equation}
\frac{\p\big(\X^{(\nf,\xf)}(n)=\x\big)}{\p\big(\X(n)=\x\big)}=\frac{q_{\nf-n}\big(\x,\vec{\de}_{d}(\xf)\big)}{q_{\nf}\big(\vec{\de}_{d}(0),\vec{\de}_{d}(\xf)\big)}\frac{h_{d}\big(\vec{\de}_{d}(0)\big)}{h_{d}\big(\x\big)}. \label{eq:radon_nik_deriv}
\end{equation}
We may also think of $\X^{(\nf,\xf)}(n)$
as an unordered random subset of $\bZ$ with $d$ points by $\X^{(\nf,\xf)}(n)=\left\{ X_{1}^{(\nf,\xf)}(n),\ld,X_{d}^{(\nf,\xf)}(n)\right\} $.
We will abuse notation in this way and write
 $\left\{ x\in\X^{(\nf,\xf)}(n)\right\}$ to mean the event that $X_j^{(\nf,\xf)}(n)=x$ for some $1\leq j\leq d$.
\end{definition}

\begin{definition}
\label{def:tesselation} (Following definitions in Section 2.3 of
\cite{caravenna_sun_zygouras_2015polynomial}) For some set $S\subset\bR^{2}$,
let $\cB(S)$ denote the Borel subsets of $S$ and let $\cL eb\left(A\right)\in\bR^{+}$
denote the Lebesgue measure for sets $A\in\cB(S)$. Given a locally
finite set $\bT\subset S$ we call $\cC:\bT\to\cB(S)$ a tessellation
of $S$ indexed by $\bT$ if $\left\{\cC(y)\right\}_{y\in\bT}$ form
a disjoint union of $S$ and such that $y\in\cC(y)$ for each $y\in\bT$.
We call $\cC(y)$ the cell associated to $y\in\bT$.

Once a tessellation $\cC$ is fixed, any function $f:\bT^{k}\to\bR$
can be extended to a function $f:S^{k}\to\bR$ by declaring that $f$
is constant on cells of the form $\cC(y_{1})\times\ld\times\cC(y_{k})$
for every $\left(y_{1},\ld,y_{k}\right)\in\bT^{k}$. Note that for
such extensions we have: 
\begin{equation}
\norm f_{L^{2}(S^{k})}^{2}=\sum_{\vec{y}\in\bT^{k}}f(y_1,\ld,y_k)^{2} \prod_{j=1}^{k} \cL eb\left(\cC(y_j)\right).\label{eq:integral_is_sum}
\end{equation}

\end{definition}

\begin{definition}
\label{def:T_N_and_rounding}Define 
\begin{eqnarray*}
\bT^{(N)} & \defequal & \left\{ (t,z)\ :\ t \in\frac{\bN}{N},z \in\frac{\bZ}{\sqrt{N}},\text{ and }z\sqrt{N}+tN\zmtwo\right\} .
\end{eqnarray*}
This discrete set is the set of space-time points that are accessible
to a diffusively rescaled simple symmetric random walk which takes steps of
spatial size $\frac{1}{\sqrt{N}}$ and temporal size $\frac{1}{N}$.
The map $\cC^{(N)}:\bT^{(N)}\to\cB\big((0,\tf)\times\bR\big)$ given by $\cC^{(N)}\left(\frac{n}{N},\frac{x}{\sqrt{N}}\right)=\left[\frac{n}{N},\frac{n}{N}+\frac{1}{N}\right)\times\left[\frac{x}{\sqrt{N}},\frac{x}{\sqrt{N}}+\frac{2}{\sqrt{N}}\right)$
is a tesselation of $(0,\tf)\times\bR$ indexed by $\bT^{(N)}$ in
the sense of Definition \ref{def:tesselation}. For a space-time point
$\left(t,z\right)\in\bR\times(0,\infty)$, we will sometimes want
to access the closest lattice point in $\bT^{(N)}$ to the point $(t,z)$.
Because of the periodicity condition $x+n\zmtwo$, using simply $\floor z,\floor t$
will not give exactly what we want. We will instead use the following
notation that takes into account this periodicity issue:
\[
\left(t,z\right)_{2}\defequal\Big(\floor t, \max\big\{x\in\bZ\ :\ x\leq z,\ x+\floor t\zmtwo\big\}\Big).
\]

\end{definition}

\begin{definition}
\label{def:scaled-NIWb} Fix $\tf>0$ and $\zf\in\bR$. Let $X^{(N),(\tf,\zf)}(t)\in\bW^d,t\in(0,\tf)$
be the rescaled version of the non-intersecting random walk bridges $\X^{(\nf,\xf)}$,
which is scaled by $N^{-\half}$ in space and by $N^{-1}$ in
time:
\[
\vec{X}^{(N),(\tf,\zf)}(t)\defequal\frac{1}{\sqrt{N}}\X^{\left(N\tf,\sqrt{N}\zf\right)_{2}}\left(\floor{Nt}\right).
\]
Define the rescaled $k$-point correlation function $\ps_{k}^{(N),(\tf,\zf)}:\left(\bT^{(N)}\right)^{k}\to\bR$
by defining for $k$-tuples $\big((t_{1},z_{1}),\ld,(t_{k},z_{k})\big)\in\left(\bT^{(N)}\right)^{k}$
where all the entries $(t_{i},z_{i})$ are distinct: 
\begin{eqnarray}
\ps_{k}^{(N),(\tf,\zf)}\big((t_{1},z_{1}),\ld,(t_{k},z_{k})\big) & \defequal & \left(\frac{\sqrt{N}}{2}\right)^{k}\sum_{\vec{j}\in\left\{ 1,\ld,d\right\} ^{k}}\p\left(\bigcap_{i=1}^{k}\left\{ X_{j_{i}}^{(N),(\tf,\zf)}(t_{i}) = z_i \right\} \right)\nonumber \\
 & = & \left(\frac{\sqrt{N}}{2}\right)^{k}\p\left(\bigcap_{i=1}^{k}\left\{ z_{i}\in\X^{(N),(\tf,\zf)}(t_{i})\right\} \right), \label{eq:def_ps_k_N}
\end{eqnarray}
and declaring that $\ps_{k}^{(N),(\tf,\zf)}\defequal0$ if any of
the space time coordinates are duplicated $(t_{i},z_{i})=(t_{j},z_{j})$
for $i\neq j$. We extend the domain of $\ps_{k}^{(N),(\tf,\zf)}$
to all of $\big((0,\tf)\times\bR\big)^{k}$ as in Definition \ref{def:tesselation}
by declaring it to be constant on the cells $\cC^{(N)}\left(t_{1},z_{1}\right)\times\cdots\times\cC^{(N)}\left(t_{k},z_{k}\right)$.
Notice that because $\ps_{k}^{(N),(\tf,\zf)}$ is constant on these
cells, whenever we compute an integral of $\ps_{k}^{(N),(\tf,\zf)}$,
which we will frequently do, such an integral is actually a sum over
discrete cells as in equation (\ref{eq:integral_is_sum}).
\end{definition}

\begin{remark}
The scaling by $N^{-\half}$ in space and $N^{-1}$ in time is
to be expected as this is the scaling needed for convergence of a random walk
to a Brownian motion.
The factor of $\left(\sqrt{N}/2\right)^{k}$ in the definition of
the $k$-point correlation function is the
same factor that appears in the local central limit theorem for simple
symmetric random walks (which have periodicity 2) to Brownian
motion. 
\end{remark}

\subsection{\label{sub:poly_chaos}Polynomial chaos expansions{\:\textendash\:}proof of Theorems
\ref{thm:partition_result} and \ref{thm:RSK_result}}

The proof of Theorems \ref{thm:partition_result} and \ref{thm:RSK_result}
follow by writing the discrete partition functions as polynomial chaos expansions and applying convergence results from \cite{caravenna_sun_zygouras_2015polynomial} which show the convergence of polynomials
chaos expansions to Wiener chaos expansions. Theorem \ref{thm:L2_convergence_psiN_to_psi} provides the key technical input in this application, and its proof is deferred to Section \ref{sub:L2}. We begin by reviewing some definitions about polynomial chaos expansions from Section 2 of \cite{caravenna_sun_zygouras_2015polynomial}.
\begin{definition}
\label{def:poly_chaos}(From Section 2.2 of \cite{caravenna_sun_zygouras_2015polynomial})
Let $\bT$ be a finite or countable set. Define:
\[
\cP^{fin}\left(\bT\right)\defequal\left\{ I\subset\bT:\abs I<\infty\right\} ,
\]
be the collection of finite subsets of $\bT$. Any function $\ps:\cP^{fin}(\bT)\to\bR$
determines a (formal if $\abs{\bT}=\infty)$ multilinear polynomial
$\Ps$ by:
\[
\Ps(x)=\sum_{I\in\cP^{fin}(\bT)}\ps(I)x^{I}\;\text{ where }\;x^{I}\defequal\prod_{i\in I}x_{i}
\]
Let $\ze=\left\{ \ze_{i}\right\} _{i\in\bT}$ be a family of independent
random variables. When $\abs{\bT}<\infty$, we say that a random variable
$X$ admits a polynomial chaos expansion with respect to $\ze$ if
it can be expressed as $X=\Ps(\ze)=\Ps\left(\left\{ \ze_{i}\right\} {}_{i\in\bT}\right)$.
When $\abs{\bT}=\infty$, we say that $X$ admits a polynomial chaos expansion with respect to $\ze$ 
if for any increasing sequence of subsets $A_{M}\subset\bT$
with $\abs{A_{M}}<\infty$ and $\lim_{M\to\infty}A_{M}=\bT$, we have
that
\[
X=\lim_{M\to\infty}\sum_{I\subset A_{M}}\ps(I)\ze^{I}\;\text{ in probability}
\]
When this is the case, the function $\ps$ is called a polynomial
chaos kernel function with respect to $\ze$. \end{definition}
\begin{proposition}
\label{prop:poly_chaos_theorem}(Theorem 2.3 from \cite{caravenna_sun_zygouras_2015polynomial})
Fix a set $S\subset\bR^{2}$. Assume that for each $N\in\bN$ the
following four ingredients are given:
\begin{enumerate}
\item $\bT^{(N)}\subset S$ a locally finite subset of $S$,
\item $\cC^{(N)}=\left\{ \cC^{(N)}(x)\right\} _{x\in\bT^{(N)}}$ a tessellation
of $S$ indexed by $\bT^{(N)}$ as in Definition \ref{def:tesselation}
and so that every cell $\cC^{(N)}(x)$ has the same volume $v^{(N)}\defequal\cL eb(\cC^{(N)}(x))$,
\item $\ze^{(N)}=\left\{ \ze_{x}^{(N)}\right\} _{x\in\bT^{(N)}}$ an independent
family of mean-zero random variables, i.e. $\cE\left[\ze_{x}^{(N)}\right]=0$,
all with the same variance $\left(\si^{(N)}\right)^{2}=\vVar\left(\ze_{x}^{(N)}\right)$
and such that the family  $\left\{ \frac{1}{v^{(N)}}\left(\ze_{x}^{(N)}\right)^{2}\right\}$, $N\in\bN,x\in\bT^{(N)}$
is uniformly integrable,
\item $\ps^{(N)}:\cP^{fin}\left(\bT^{(N)}\right)\to\bR$ a polynomial chaos
kernel function.
\end{enumerate}
Let $\Ps^{(N)}(z)$ be a formal multi-linear polynomial defined by
the kernel $\ps^{(N)}:\cP^{fin}(\bT^{(N)})\to\bR$ as in Definition
\ref{def:poly_chaos}. Assume that $v^{(N)}\to0$ as $N\to\infty$
and that the following conditions are satisfied:

\noindent i) There exists $\si\in(0,\infty)$
such that
\[
\lim_{N\to\infty}\frac{\left(\si^{(N)}\right)^{2}}{v^{(N)}}=\si^{2}.
\]

\noindent ii) There exists $\ps:\cP^{fin}(S)\to\bR$
so that for every $k\in\bN$, the restriction of $\ps$
to $k$ elements subsets, $\ps:S^{k}\to\bR$ has
$\norm{\ps}_{L^{2}\left(S^{k}\right)}<\infty$ and
so that we have the convergence
\[
\lim_{N\to\infty}\norm{\ps^{(N)}-\ps}_{L^{2}\left(S^{k}\right)}=0.
\]

\noindent iii) The following limit holds:
\[
\lim_{\ell\to\infty}\limsup_{N\to\infty}\sum_{I\in\cP^{fin}(\bT^{(N)}),\abs I>\ell}\left(\si^{(N)}\right)^{2\abs I}\ps^{(N)}\left(I\right)^{2}=0.
\]
Then the polynomial chaos expansion $\Ps^{(N)}\left(\ze^{(N)}\right)$
is well defined and converges in distribution as $N\to\infty$ to
a random variable $\Ps$ with explicit Wiener chaos
expansion given in terms of a white noise $\xi$ on $S$:
\[
\Ps^{(N)}\left(\ze^{(N)}\right)\To\Ps, \;\; \Ps\defequal\sum_{k=0}^{\infty}\frac{\si^k}{k!}\intop_{S^{k}}\ps(y_{1,\ld,}y_{k})\prod_{i=1}^{k}\xi(\d y_{i}).
\]
\end{proposition}
\begin{remark}
In general, the white noise integral over the set $S^k$ requires interpretation due to issues that arise on the set $E = \left\{\vec{y} : y_i = y_j, \text{ for some } i\neq j \right\}$; see Section 2.1 in \cite{caravenna_sun_zygouras_2015polynomial} or the final Remark in Section 3.2 in \cite{AKQ_2014}. When $S = [s,s^\prime]\times A$, $A\subset\bR$ and $\ps$ is symmetric with $\ps(\vec{y})=0$ for $y\in E$, an equivalent way to write this integral which avoids this issue is:
\[
\intop_{([s,s^\prime] \times A)^k} \ps(y_1,\ld,y_k) \prod_{j=1}^k \xi(\d y_{j}) = k! \intop_{\De_k[s,s^\prime]} \intop_{A^k} \ps\big( (t_1,z_1), \ld, (t_k,z_k) \big) \prod_{j=1}^k \xi(\d z_{j},\d t_{j})
\]
\end{remark}
Both Theorems \ref{thm:partition_result} and \ref{thm:RSK_result}
are proven using Proposition \ref{prop:poly_chaos_theorem}. The ingredients
for the two theorems are very similar: only the collection of variables
$\ze^{(N)}$ differs. The other ingredients are
defined in Definition \ref{def:ingredients} below. Lemma \ref{lem:Z_is_poly_chaos}
will justify our choice for the variables $\ze^{(N)}$ . With these
ingredients chosen, the verification of condition i) from Proposition
\ref{prop:poly_chaos_theorem} will be a straightforward calculation,
ii) will follow by the the main technical Theorem \ref{thm:L2_convergence_psiN_to_psi}
and iii) will follow by Proposition \ref{prop:uniform_exp_moms}.
\begin{definition}
\label{def:ingredients} Fix $\tf>0$ and $\zf\in\bR$. Set $S=(0,\tf)\times\bR$.
For $N\in\bN$, let $\bT^{(N)}\subset(0,\tf)\times\bR$ and the tessellation
$\cC^{(N)}$ be as in Definition \ref{def:T_N_and_rounding}. Notice
that $v^{(N)}=\frac{2}{N\sqrt{N}}$. Let $\ps^{(N)}:\cP^{fin}\left(\bT^{(N)}\right)\to\bR$
be the polynomial chaos kernel which is defined by setting its action
on sets of size $\abs I=k$ to be the $k$-point correlation function
for the non-intersecting random walks $\ps_{k}^{(N),(\tf,\zf)}$
\[
\ps^{(N)}\big(\left\{ (t_{1},z_{1}),\ld,(t_{k},z_{k})\right\} \big)\defequal\ps_{k}^{(N),(\tf,\zf)}\big((t_{1},z_{1}),\ld,(t_{k},z_{k})\big).
\]
Define the target limiting kernel $\ps$
in a similar way by defining its action on sets of size $k$ to be
the $k$-point correlation functions for the non-intersecting Brownian
bridges $\ps_{k}^{\left(\tf,\zf\right)}$
\[
\ps\left(\big\{ (t_{1},z_{1}),\ld,(t_{k},z_{k})\right\} \big)=\ps_{k}^{(\tf,\zf)}\big((t_{1},z_{1}),\ld,(t_{k},z_{k})\big).
\]
With this choice of $\ps$, by comparing Definition \ref{def:Z_partition} and  $\Ps$ from Proposition \ref{prop:poly_chaos_theorem},
we observe
\[
\Ps=\frac{\cZ_{d}^{\si}(\tf,\zf)}{\rh(\tf,\zf)^{d}}.
\]
\end{definition}
\begin{lemma}
\label{lem:Z_is_poly_chaos} Given an iid collection
of random variables $\left\{ \om\left(n,x\right)\right\} _{x\in\bZ,n\in\bN}$,
define for every $\xf\in\bZ$, $\nf\in\bN$ with $\xf+\nf\zmtwo$
the quantity
\[
Z_{d}^{\om}\left(\nf,\xf\right)\defequal\e\left[\prod_{j=1}^{d}\prod_{n=1}^{\nf-1}\om\left(n,X_{j}^{(\nf,\xf)}(n)\right)\right].
\]
Let $\bT^{(N)}$ be as in Definition \ref{def:T_N_and_rounding}.
For each $N\in\bN$ define a field of random variables $\ze^{(N)}=\left\{ \ze^{(N)}\left(t,z\right)\right\} _{(t,z)\in\bT^{(N)}}$
by
\[
\ze^{(N)}(t,z)\defequal\frac{2}{\sqrt{N}}\bigg(\om\Big(Nt,\sqrt{N}z\Big)-1\bigg).
\]
Then, $Z_{d}^{\om}\left(\big(N\tf,\sqrt{N}\zf\big)_{2}\right)$
is a polynomial chaos expansion in the variables $\ze^{(N)}$ with
the polynomial chaos kernel $\ps^{(N)}$ as in Definition \ref{def:ingredients}:
\begin{equation}
Z_{d}^{\om}\left(\big(N\tf,\sqrt{N}\zf\big)_{2}\right)=\Ps^{(N)}\left(\ze^{(N)}\right)=\sum_{I\subset\bT^{(N)}}\ps^{(N)}(I)\left(\ze^{(N)}\right)^{I}.\label{eq:poly_chaos}
\end{equation}
\end{lemma}
\begin{remark}
Even though $\bT^{(N)}$ is an infinite space, there is no issue
with the interpretation of the RHS of equation (\ref{eq:poly_chaos})
since for each $N$, $\ps^{(N)}$ is nonzero on only finitely many
cells of $\bT^{(N)}$.\end{remark}
\begin{proof}
Using $X^{(N),(\tf,\zf)}(t)$ from Definition \ref{def:scaled-NIWb},
the definition of $\ze^{(N)}$, and the definition of $Z_{d}^{\om}\left(\nf,\xf\right)$
we have 
\begin{eqnarray*}
Z_{d}^{\om}\left(\big(N\tf,\sqrt{N}\zf\big)_{2}\right) & = & \e\Bigg[\prod_{j=1}^{d}\prod_{t\in(0,\tf)\cap\frac{\bN}{N}}\om\left(\big(Nt,\sqrt{N}X_{j}^{(N),(\tf,\zf)}(t)\big)_{2}\right)\Bigg]\\
 & = & \e\Bigg[\prod_{j=1}^{d}\prod_{t\in(0,\tf)\cap\frac{\bN}{N}}\left(1+\frac{\sqrt{N}}{2}\cdot\ze^{(N)}\left(t,X_{j}^{(N),(\tf,\zf)}(t)\right)\right)\Bigg].
\end{eqnarray*}
From here, we expand the product completely into a finite sum of monomials
and then bring the expectation into each monomial individually. Focus
attention for the moment only on those monomials which are a product
of exactly $k\in\bN$ random variables. Since the walks $\X^{(N),(\tf,\zf)}$
is independent of the disorder, the expectation of each monomial is
a weighted sum over all the possible positions $\vec{X}^{(N),(\tf,\zf)}$
could take. For the monomial corresponding to indices $\vec{j}=\left(j_{1},\ld,j_{k}\right)\in\left\{ 1,\ld,d\right\} ^{k}$
and times $\vec{t}=\left(t_{1},\ld,t_{k}\right)\in\left((0,\tf)\cap\frac{\bN}{N}\right)^{k}$
this is:
\begin{align*}
&\e\left[\prod_{\ell=1}^{k}\frac{N^\half}{2}\ze^{(N)}\left(t_{{\ell}},X_{j_{\ell}}^{(N),(\tf,\zf)}(t_{{\ell}})\right)\right] \\
=& { \frac{N^\frac{k}{2}}{2^k}} \sum_{\vec{z}\in\frac{\bZ^{k}}{\sqrt{N}}}\p\left( \bigcap_{\ell=1}^{k} {\left\{  X_{j_{\ell}}^{(N),(\tf,\zf)}(t_{{\ell}})=z_{\ell}\right\}} \right) \prod_{\ell=1}^{k}\ze^{(N)}\left(t_{\ell},z_{\ell}\right).
\end{align*}
We now recognize by comparing with the definition of $\ps_{k}^{(N),(\tf,\zf)}$
from equation (\ref{eq:def_ps_k_N}) that when summing over all indices
$\vec{j}$ and all times $\vec{t}$ this yields exactly the contribution
from sets of size $\abs I=k$ in the sum from the RHS of equation
(\ref{eq:poly_chaos}). Doing this for every $k\in\bN$ gives exactly
the desired result.\end{proof}
\begin{proposition}
\label{prop:uniform_exp_moms} For any $\ga>0$, we have that
\begin{equation}
\lim_{\ell\to\infty}\sup_{N\in\bN}\sum_{k=\ell}^{\infty}\frac{\ga^{k}}{k!}\intop_{\big((0,\tf)\times\bR\big)^{k}}\abs{\ps_{k}^{(N),(\tf,\zf)}(\vec{w})}^{2}\d\vec{w}=0.\label{eq:uniform_exp_moms}
\end{equation}
\end{proposition}
The proof of Proposition \ref{prop:uniform_exp_moms} is deferred
until Section \ref{sec:L2_bounds}, where it is proven using tools
which are developed on the way to the proof of Theorem \ref{thm:L2_convergence_psiN_to_psi}.

\begin{proof}
(Of Theorem \ref{thm:partition_result}) The proof will be an application
of Proposition \ref{prop:poly_chaos_theorem}. Take the ingredients
$S=(0,\tf)\times\bR$, $\bT^{(N)}$, $\cC^{(N)}$, $\ps^{(N)}$ and
${\ps}$ as in Definition \ref{def:ingredients}.
We define the variables $\ze^{(N)}=\left(\ze^{(N)}(t,z)\right)_{(t,z)\in\bT^{(N)}}$
by:
\[
\ze^{(N)}(t,z)\defequal\frac{2}{\sqrt{N}}\left(\frac{\exp\left(\be_{N}\om\big(Nt,\sqrt{N}z\big)_{2}\right)}{\exp\left(\La(\be_{N})\right)}-1\right).
\]
The collection $\Big\{ \frac{1}{v^{(N)}}\big(\ze_{y}^{(N)}\big)^{2}\Big\} _{N\in\bN,y\in\bT^{(N)}}$
can be verified to be uniformly integrable by finding a uniform bound
on the second moment: this computation is carried out in equation
(6.7) in \cite{caravenna_sun_zygouras_2015polynomial}. Comparing
Definition \ref{def:partition_function} to the result of Lemma \ref{lem:Z_is_poly_chaos},
we see that $Z_{d}^{\be_{N}}$ can be expressed as a polynomial chaos
series in these variables:
\[
Z_{d}^{\be_{N}}\left(\big(N\tf,\sqrt{N}\zf\big)_{2}\right)\cdot\exp\big(-d\floor{N\tf}\La(\be_{N})\big)=\Ps^{(N)}\big(\ze^{(N)}\big),
\]
and so the desired convergence to $\cZ_{d}^{\sqrt{2}\be}$ would be the
conclusion of Proposition \ref{prop:poly_chaos_theorem} with $\si=\sqrt{2}\be$,
provided we verify conditions i), ii) and iii) from Proposition
\ref{prop:poly_chaos_theorem}. This is carried out below:

\noindent i) Recall the definition $\be_{N}=N^{-\oo 4}{\be}$
. By the definition of $\left(\si^{(N)}\right)^{2}=\vVar\left(\ze^{(N)}(t,z)\right)$
we have as $N\to\infty$:
\begin{eqnarray*}
\frac{\left(\si^{(N)}\right)^{2}}{v^{(N)}} & = & \left(\frac{N\sqrt{N}}{2}\right)\left(\frac{2}{\sqrt{N}}\right)^{2}\left(\frac{\vVar\left[\exp\left(\be_{N}\om\right)\right]}{\vE\left[\exp\left(\be_{N}\om\right)\right]^{2}}\right)\\
 & = & 2\sqrt{N}\left(\frac{\exp\left(\La\left(2\be_{N}\right)\right)}{\exp\left(2\La(\be_{N})\right)}-1\right)\\
 & = & 2\sqrt{N}\Big(\be_{N}^{2}+o\big(N^{-\half}\big)\Big)=2{\be}^{2}+o(1).
\end{eqnarray*}
The above limit follows from the Taylor series expansion $\exp(\La(\ga))=\cE[\exp(\ga \om)]=1+\half\ga^2+o(\ga^2)$ as $\ga \to 0$ since the weights $\om$ is assumed to be mean zero and unit variance.

\noindent ii) By the definition $\ps^{(N)}$ and $\ps$ from
Definition \ref{def:ingredients}, the condition to be verified is
that for each $k\in\bN$, we have
\[
\lim_{N\to\infty}\norm{\ps_{k}^{(N),(\tf,\zf)}-\ps_{k}^{(\tf,\zf)}}_{L^{2}\left(\big((0,\tf)\times\bR\big)^{k}\right)}=0.
\]
This is exactly the conclusion of Theorem \ref{thm:L2_convergence_psiN_to_psi}.

\noindent iii) Since all the terms of the sum are non-negative, we
can rearrange the sum into sets of size $\abs I=k$. This gives:
\begin{eqnarray}
 &  & \lim_{\ell\to\infty}\limsup_{N\to\infty}\sum_{k=\ell+1}^{\infty}\sum_{ \stackrel{I \subset \bT^{(N)}}{\abs I=k}}\left(\si^{(N)}\right)^{2k}\abs{\ps_{k}^{(N),(\tf,\zf)}\left(I\right)}^{2}\nonumber \\
 & = & \lim_{\ell\to\infty}\limsup_{N\to\infty}\sum_{k=\ell+1}^{\infty}\left(\frac{\left(\si^{(N)}\right)^{2}}{v^{(N)}}\right)^{k}\left(v^{(N)}\right)^{k}\sum_{ \stackrel{I \subset \bT^{(N)}}{\abs I=k}}\abs{\ps_{k}^{(N),(\tf,\zf)}\left(I\right)}^{2}\nonumber \\
 & = & \lim_{\ell\to\infty}\sup_{N\to\infty}\sum_{k=\ell+1}^{\infty}\left(\frac{\left(\si^{(N)}\right)^{2}}{v^{(N)}}\right)^{k}\frac{1}{k!}\intop_{\big((0,\tf)\times\bR\big)^{k}}\abs{\ps_{k}^{(N),(\tf,\zf)}\left(\vec{w}\right)}^{2}\d\vec{w},\label{eq:step_iii}
\end{eqnarray}
where we have recognized the sum over cells times the volume of the
cell as the integral of $\ps_{k}^{(N),\left(\tf,\zf\right)}$ as in
equation (\ref{eq:integral_is_sum}). Finally, since $\lim_{N\to\infty}\frac{\left(\si^{(N)}\right)^{2}}{v^{(N)}}=2{\be}^{2}$,
we have for $N$ sufficiently large, this ratio is less than $2{\be}^{2}+1$.
With this bound in place, we see that the limit on the RHS of equation
(\ref{eq:step_iii}) is $0$ by application of Proposition \ref{prop:uniform_exp_moms}.
\end{proof}

\begin{proof}
(Of Theorem \ref{thm:RSK_result}) We first notice that any tuple
of paths $\pi=\left(\pi_{1},\ld,\pi_{d}\right)$ in $\Pi_{N+d,N+d}^{d}$
is forced to pass through a triangle of $\half d(d+1)$ points
near $(1,1)$ and a triangle of $\half d(d+1)$ points
near $(N+d,N+d)$. These points are indicated in Figure \ref{fig:LatticePathsCap} with a square around them. Specifically, for index $1\leq\ell\leq d$
the lattice path $\pi_{\ell}$ from $(1,\ell)$ to $(N+d,N+\ell)$
is forced to pass through the points
$\left\{ (i,\ell)\right\} _{i=1}^{d-\ell+1}$ and the points $\left\{ (N+i,N+\ell)\right\} _{i=d-\ell+1}^{d}$.  Let us denote by $A$ this
set of $d(d+1)$ lattice points and by $\tilde{\Pi}_{N+d,N+d}^{d}$ the
``free'' part of the lattice paths, that is $d$ non-intersecting
lattice paths from the diagonal $\left\{ (d+1-\ell,\ell)\right\} _{\ell=1}^{d}$
to the opposite diagonal $\left\{ \left(N+d+1-\ell,N+\ell\right)\right\} _{\ell=1}^{d}$.
We can thus factor $\ta_{N+d,d}^{(N)}(N+d)$ as:
\begin{equation}
\vE\left[g_{1,1}^{(N)}\right]^{-d\left(2N+d\right)}\ta_{N+d,d}^{(N)}(N+d)=\left(\prod_{(i,j)\in A}\frac{g_{ij}^{(N)}}{\vE\left[g_{1,1}^{(N)}\right]}\right)\left(\sum_{\pi\in\tilde{\Pi}_{N+d,N+d}^{d}}\prod_{\ell=1}^{d}\prod_{(i,j)\in\pi_{\ell}}\frac{g_{ij}^{(N)}}{\vE\left[g_{1,1}^{(N)}\right]}\right).\label{eq:tau}
\end{equation}
The factor $\left(\prod_{(i,j)\in A}g_{ij}^{(N)}\cE\left[g_{1,1}^{(N)}\right]^{-1}\right)\to1$
as $N\to\infty$ since the $g_{ij}^{(N)}$ are independent and since
we are given that $\vVar\left(g_{ij}^{(N)}\right)/\vE\left[g_{ij}^{(N)}\right]^{2}=O\left(N^{-\half}\right)$.
Thus this prefactor coming from weights in $A$ is asymptotically
negligible. Notice now that rotating by $45^{\circ}$ sends every
tuple of lattice path in $\tilde{\Pi}_{N+d,N+d}^{d}$ to an ensemble of non-intersecting random walks $\vec{X}^{(2N,0)}$ as defined in Definition \ref{def:partition_function}. This is illustrated in Figure \ref{fig:LatticePathsCap} by comparing the relationship between the left and right subfigures. More precisely the map $\la:\left\{ 1,\ld,N+d\right\} ^{2}\backslash A\to\bZ\times[0,2N]$
by $\la\left(n,m\right)=\left(n+m-d-1,n-m+d-1\right)$ gives a bijection
from the set of lattice paths $\tilde{\Pi}_{N+d,N+d}^{d}$ to the
set of non-intersecting random walks $\vec{X}^{(2N,0)}$. Writing
$\Om^{(2N,0)}$ as the set of all such paths, we have:
\begin{align}
\text{LHS (\ref{eq:tau})} = & \left(1+o(1)\right)\left(\sum_{\vec{X}^{(2N,0)}\in\Om^{(2N,0)}}\prod_{\ell=1}^{d}\prod_{n=1}^{2N-1}\frac{g^{(N)}\left(\la^{-1}\left(n,\vec{X}^{(2N,0)}(n)\right)\right)}{\cE\left[g_{1,1}^{(N)}\right]}\right)\nonumber \\
  = & (1+o(1))\abs{\Om^{(2N,0)}}\e\left[\prod_{\ell=1}^{d}\prod_{n=1}^{2N-1}\frac{g^{(N)}\left(\la^{-1}\left(n,\vec{X}^{(2N,0)}(n)\right)\right)}{\cE\left[g_{1,1}^{(N)}\right]}\right].\label{eq:rewrite_as_E}
\end{align}
where $\e$ is the uniform measure on $\Om^{(2N,0)}$, as in Definition
\ref{def:NIWb}. There is an exact formula, MacMahon's formula, for
$\abs{\Om^{(2N,0)}}$. We use the form of MacMahon's formula given in equation (3.34) of  \cite{Johansson2005-Hahn} to obtain
\begin{eqnarray*}
\abs{\Om^{(2N,0)}} & = & \prod_{i=0}^{d-1}\binom{2N+2i}{N+i}\binom{2N+2i}{i}^{-1}.
\end{eqnarray*}
By Stirling's formula, we have
that 
\[
\abs{\Om^{(2N,0)}}=2^{d(2N+d-1)}N^{-\half d^2}\pi^{-\half d}\big({\textstyle \prod_{j=0}^{d-1}j!}\big)(1+o(1)),
\]
as $N\to\infty$. Along with $\rh(2,0)^{d}=\left(4\pi\right)^{-d/2}$,
this accounts for the prefactor that appears on the LHS of equation (\ref{eq:RSK_eq}), and it remains
only to show that:
\[
\e\left[\prod_{\ell=1}^{d}\prod_{n=1}^{2N-1}\frac{g^{(N)}\left(\la^{-1}\left(n,\vec{X}^{(2N,0)}(n)\right)\right)}{\cE\left[g_{1,1}^{(N)}\right]}\right]\To\frac{\cZ_{d}^{\sqrt{2}{\be}}\left(2,0\right)}{\rho(2,0)^{d}}.
\]
This is very similar to Theorem \ref{thm:partition_result}. We take the ingredients $S=(0,\tf)\times\bR$, $\bT^{(N)}$,
$\cC^{(N)}$, $\ps^{(N)}$ and $\ps$
as in Definition \ref{def:ingredients}. Set the variables $\ze^{(N)}=\left(\ze^{(N)}(t,z)\right)_{(t,z)\in\bT^{(N)}}$
by: 
\[
\ze^{(N)}(t,z)\defequal\frac{2}{\sqrt{N}}\left(\frac{g^{(N)}\left(\la^{-1}\left(Nt,\sqrt{N}z\right)\right)}{\cE\left[g_{1,1}^{(N)}\right]}-1\right).
\]
With this definition, we recognize by Lemma \ref{lem:Z_is_poly_chaos}
the expectation on the RHS of equation (\ref{eq:rewrite_as_E}) as
a polynomial chaos series in $\ze^{(N)}$, so the desired result follows
by application of Proposition \ref{prop:poly_chaos_theorem} with $\si=\sqrt{2}\be$.
The justification of condition ii) from Proposition \ref{prop:poly_chaos_theorem}
proceed exactly as in proof of Theorem \ref{thm:partition_result}
with no changes. Condition iii) follows by the uniform integrability assumption and the same variance calculation as in Theorem \ref{thm:partition_result}. Condition i) follows since:
\begin{eqnarray*}
\frac{\left(\si^{(N)}\right)^{2}}{v^{(N)}} & = & \left(\frac{N\sqrt{N}}{2}\right)\left(\frac{2}{\sqrt{N}}\right)^{2}\left(\frac{\cV ar\left[g_{1,1}^{(N)}\right]}{\cE[g_{1,1}^{(N)}]^{2}}\right)\to 2{\be},
\end{eqnarray*}
by the hypothesis on the weights $g^{(N)}$.
\end{proof}

\subsection{\label{sub:L2}$L^{2}$ Convergence{\:\textendash\:}proof of Theorem \ref{thm:L2_convergence_psiN_to_psi}}

The main technical result of this article is the $L^{2}$ convergence
in Theorem \ref{thm:L2_convergence_psiN_to_psi}. The proof of Theorem
\ref{thm:L2_convergence_psiN_to_psi} goes by dividing the space $\big((0,\tf)\times\bR\big)^{k}$
into four parts and analyzing contribution to the integral on each
one. 
\begin{definition}
\label{def:space-time} Instead of working with $\big((0,\tf)\times\bR\big)^{k}$,
it will be more convenient to work with $k$-tuples where the times
are ordered so that $t_{1}<t_{2}<\ld<t_{k}$. To this end, define
$\cS_{k}(0,\tf)\subset\big((0,\tf)\times\bR\big)^{k}$ by:
\[
\cS_{k}(0,\tf)\defequal\left\{ \big((t_{1},z_{1});\ld,(t_{k},z_{k})\big):\ z_{i}\in\bR,\ 0<t_{1}<\ld<t_{k}<\tf\right\} .
\]
Since the set of points omitted is a set of measure $0$ and since
there are $k!$ ways to permute the points in $\cS_{k}(0,\tf)$, we
have for any integrable $f$ which is symmetric with respect to permutations
of its entries that:
\begin{equation}
\intop_{\cS_{k}(0,\tf)}f(\vec{w})\d\vec{w}=\frac{1}{k!}\intop_{\big((0,\tf)\times\bR\big)^{k}}f(\vec{w})\d\vec{w}. \label{eq:S_k_int}
\end{equation}
For any parameters $\de,\et,M>0$, define the sets:
\begin{eqnarray*}
D_{1}(\de,\et,M) & \defequal & \cS_{k}(0,\tf)\cap\bigcap_{i=1}^{k}\left\{ \abs{z_{i}}\leq M,\ t_{i}\in\left[\de,\tf-\de\right] \right\} \cap\bigcap_{i=1}^{k-1}\left\{ t_{i+1}-t_{i}>\et\right\} ,\\
D_{2}(\de,\et,M) & \defequal & \cS_{k}(0,\tf)\cap\bigcap_{i=1}^{k}\left\{ \abs{z_{i}} \leq M,\ t_{i}\in\left[\de,\tf-\de\right]\right\} \cap\bigcup_{i=1}^{k-1}\left\{ t_{i+1}-t_{i}\leq\et\right\} ,\\
D_{3}(\de) & \defequal & \cS_{k}(0,\tf)\cap\bigcup_{i=1}^{k}\left\{ t_{i}\in(0,\de)\cup(\tf-\de,\tf)\right\} ,\\
D_{4}(M) & \defequal & \cS_{k}(0,\tf)\cap\bigcup_{i=1}^{k}\left\{ \abs{z_{i}}>M\right\} .
\end{eqnarray*}
These sets, of course, depend on $\tf$ and $k$ as well but we suppress
this from our notation for simplicity. Notice that for any choice
of parameters these four sets subdivide $\cS_{k}(0,\tf)$, 
\[
\cS_{k}(0,\tf)=D_{1}\left(\de,\et,M\right)\cup D_{2}(\de,\et,M)\cup D_{3}(\de)\cup D_{4}(M).
\]
The set $D_{1}(\de,\et,M)$, for small $\de,\et>0$ and large $M>0$
can be thought of as covering the typical part of the space $\cS_{k}(0,\tf)$
and the sets $D_{2}(\de,\et,M)$, $D_{3}(\de)$ and $D_{4}(M)$ can
be thought of as exceptional sets. This subdivision is chosen to make
$D_{1}(\de,\et,M)$ a bounded set on which the function $\ps_{k}^{(N),(\tf,\zf)}$
develops no singularities as $N\to\infty$; this makes the $L^{2}$ convergence in equation
(\ref{eq:L2_thm}) more straightforward on $D_{1}(\de,\et,M)$. All
of the singularities/non-boundedness issues occur on the exceptional
sets $D_{2}(\de,\et,M)$, $D_{3}(\de)$ and $D_{4}(M)$ where we will
separately argue that they have a negligible contribution to the integral
in equation (\ref{eq:L2_thm}). With this strategy, the substance
of the proof of Theorem \ref{thm:L2_convergence_psiN_to_psi} is divided
into Propositions \ref{prop:D1}, \ref{prop:D2}, \ref{prop:D3} and
\ref{prop:D4}, each handling one of these four sets. \end{definition}
\begin{proposition}
\label{prop:D1} Fix $\tf>0$ and $\zf\in\bR$. For any $\de,\et,M>0$,
we have pointwise convergence uniformly over all $\vec{w}\in D_{1}\left(\de,\et,M\right)$:
\[
\lim_{N\to\infty}\ps_{k}^{(N),(\tf,\zf)}\left(\vec{w}\right)=\ps_{k}^{(\tf,\zf)}\left(\vec{w}\right).
\]
In addition, there is a constant $C_{D_{1}}=C_{D_{1}}(\de,\et,M)$
so that that for all $\vec{w}\in D_{1}\left(\de,\et,M\right)$ we
have the bounds:
\[
\sup_{N}\ps_{k}^{(N),(\tf,\zf)}(\vec{w})\leq C_{D_{1}},\ \ \ps_{k}^{(\tf,\zf)}(\vec{w})\leq C_{D_{1}}.
\]

\end{proposition}

\begin{proposition}
\label{prop:D2} Fix $\tf>0$ and $\zf\in\bR$. For any $\de,M>0$,
and any given $\ep>0$, there exists $\et>0$ small enough so that:
\[
\limsup_{N\to\infty}\intop_{D_{2}(\de,\et,M)}\abs{\ps_{k}^{(N),(\tf,\zf)}(\vec{w})}^{2}\d\vec{w}<\ep.
\]

\end{proposition}

\begin{proposition}
\label{prop:D3} Fix $\tf>0$ and $\zf\in\bR$. For any $\ep>0$,
there exists $\de>0$ small enough so that:
\begin{eqnarray*}
\limsup_{N\to\infty}\intop_{D_{3}(\de)}\abs{\ps^{(N),(\tf,\zf)}(\vec{w})}^{2}\d\vec{w} & < & \ep.
\end{eqnarray*}

\end{proposition}

\begin{proposition}
\label{prop:D4} Fix $\tf>0$ and $\zf\in\bR$. For any $\ep>0$,
there exists $M>0$ large enough so that:
\begin{equation}
\limsup_{N\to\infty}\intop_{D_{4}(M)}\abs{\ps_{k}^{(N),(\tf,\zf)}(\vec{w})}^{2}\d\vec{w}<\ep. \label{eq:D4_target}
\end{equation}

\end{proposition}
We defer the proof of these estimates to later sections. Proposition
\ref{prop:D1} and Proposition \ref{prop:D2} are proven using tools
from determinantal point processes and orthogonal polynomials in Section
\ref{sec:Orthogonal-Polynomials}. Proposition \ref{prop:D3} and
Proposition \ref{prop:D4} are proven in Section \ref{sec:L2_bounds}
by a connection to the overlap time of two independent
non-intersecting random walk bridges which is analyzed in Section \ref{sec:overlap_times}. 
\begin{proof}
(Of Theorem \ref{thm:L2_convergence_psiN_to_psi}) By the relationship
in equation (\ref{eq:S_k_int}), it suffices to show $L^{2}$ convergence
on $\cS_{k}(0,\tf)$. Fix any $\ep>0$. The strategy will be to first
choose $\de,\et,M>0$ so that the contribution on the exceptional
sets $D_{2}(\de,\et,M)$, $D_{3}(\de)$ and $D_{4}(M)$ are less than
$\ep$, and once $\de,\et,M$ are fixed, we will argue that on the
typical set $D_{1}(\de,\et,M)$ we have $L^{2}$ convergence.

By Proposition \ref{prop:D3} and Proposition \ref{prop:D4} and a
union bound, we can find $\de>0$ small enough and $M>0$ large enough
so that: 
\[
\limsup_{N\to\infty}\intop_{D_{3}(\de)\bigcup D_{4}(M)}\abs{\ps^{(N),(\tf,\zf)}(\vec{w})}^{2}\d\vec{w}<\ep.
\]
With this $\de,M$ chosen in this way, by Proposition \ref{prop:D2}
we can now find $\et>0$ so small so that 
\[
\limsup_{N\to\infty}\intop_{D_{2}(\de,\et,M)}\abs{\ps_{k}^{(N),(\tf,\zf)}(\vec{w})}^{2}\d\vec{w}<\ep.
\]

Now, since $\intop_{\cS_{k}(0,\tf)}\abs{\ps_{k}^{(\tf,\zf)}}^{2}<\infty$
by Proposition \ref{prop:psi_k_L2}, and since $\one\{D_{1}(\de,\et,M)\}\to one\{\cS_{k}(0,\tf)\}$ as ${\de\to0,\et\to0,M\to\infty}$,
we know by the dominated convergence theorem that it is possible to
further shrink $\de,\et$ and enlarge $M$ so that on $D_{2}(\de,\et,M)\cup D_{3}(\de)\cup D_{4}(M)=\cS_k(0,\tf) \setminus D_{1}\left(\de,\et,M\right)$ we have
\[
\intop_{D_{2}(\de,\et,M)\cup D_{3}(\de)\cup D_{4}(M)}\abs{\ps_{k}^{(\tf,\zf)}(\vec{w})}^{2}\d\vec{w}<\ep.
\]
On the remaining set $D_{1}(\de,\et,M)$, by Proposition \ref{prop:D1}, we have pointwise convergence $\abs{\ps_{k}^{(N),(\tf,\zf)}(\vec{w})-\ps_{k}^{(\tf,\zf)}(\vec{w})}^{2} \to 0$ as $N\to\infty$. Since $\abs{\ps_{k}^{(N),(\tf,\zf)}(\vec{w})-\ps_{k}^{(\tf,\zf)}(\vec{w})}^{2}\leq4C_{D_{1}}^{2}$ is
bounded by Proposition \ref{prop:D1}, and since $D_{1}(\de,\et,M)$
is a compact set hence by an application of the bounded convergence theorem
we have
\[
\lim_{N\to\infty}\intop_{D_{1}(\de,\et,M)}\abs{\ps_{k}^{(N),(\tf,\zf)}(\vec{w})-\ps_{k}^{(\tf,\zf)}(\vec{w})}^{2}=0.
\]
Finally, we use $\abs{\ps_{k}^{(N),(\tf,\zf)}(\vec{w})-\ps_{k}^{(\tf,\zf)}(\vec{w})}^{2}\leq\abs{\ps_{k}^{(N),(\tf,\zf)}(\vec{w})}^{2}+\abs{\ps_{k}^{(\tf,\zf)}(\vec{w})}^{2}$
(both are non-negative) and a union bound to arrive at:
\begin{align}
\intop_{\cS_{k}(0,\tf)}\abs{\ps_{k}^{(N),(\tf,\zf)}(\vec{w})-\ps_{k}^{(\tf,\zf)}(\vec{w})}^{2}\d\vec{w} \leq & \intop_{D_{1}(\de,\et,M)}\abs{\ps_{k}^{(N),(\tf,\zf)}(\vec{w})-\ps_{k}^{(\tf,\zf)}(\vec{w})}^{2}\d\vec{w} \nonumber \\
   & +\intop_{D_{2}(\de,\et,M)\cup D_{3}(\de)\cup D_{4}(M)}\abs{\ps_{k}^{(N),(\tf,\zf)}(\vec{w})}^{2}\d\vec{w} \nonumber \\
   & +\intop_{D_{2}(\de,\et,M)\cup D_{3}(\de)\cup D_{4}(M)}\abs{\ps_{k}^{(\tf,\zf)}(\vec{w})}^{2}\d\vec{w} \label{eq:cut_up_space}
\end{align}
Taking the limit $N \to \infty$ of equation (\ref{eq:cut_up_space}), we see by choice of $\de, \et, M$ that the RHS is less than $3 \ep$. Since $\ep$ arbitrary, the limit $N\to\infty$ of the LHS of equation (\ref{eq:cut_up_space}) is $0$, as desired.
\end{proof}

\section{Determinantal Kernels and Orthogonal Polynomials\label{sec:Orthogonal-Polynomials}}

In this section, we will prove the pointwise convergence in Proposition
\ref{prop:D1} and the bound in Proposition \ref{prop:D2}. To do
this, we will exploit the fact the fact that $k$-point correlation
functions $\ps_{k}^{(\tf,\zf)}$, $\ps_{k}^{(N),(\tf,\zf)}$ can be
written as $k\times k$ determinants of kernel
functions $K^{(\tf,\zf)},K^{(N),(\tf,\zf)}$, namely:
\begin{eqnarray*}
\ps_{k}^{(\tf,\zf)}\big(\left(t_{1},z_{1}\right),\ld,\left(t_{k},z_{k}\right)\big) & = & \det\left[K^{(\tf,\zf)}\big((t_{i},z_{i});(t_{j},z_{j})\big)\right]_{i,j=1}^{k}\\
\ps_{k}^{(N),(\tf,\zf)}\big(\left(t_{1},z_{1}\right),\ld,\left(t_{k},z_{k}\right)\big) & = & \det\left[K^{(N),(\tf,\zf)}\big((t_{i},z_{i});(t_{j},z_{j})\big)\right]_{i,j=1}^{k}.
\end{eqnarray*}
Our analysis will proceed by writing the determinantal kernels $K^{(N),(\tf,\zf)}$
and $K^{(\tf,\zf)}$ explicitly in terms of orthogonal polynomials.
\begin{remark}
\label{rem:equiv_kernels} Given any determinantal kernel $K\big((t,z);(t^{\prime},z^{\prime})\big)$,
one can construct an equivalent kernel by choosing any function $g(t,z)\neq0$
and then setting $\tilde{K}\big((t,z);(t^{\prime},z^{\prime})\big)\defequal\frac{g(t^{\prime},z^{\prime})}{g(t,z)}K\big((t,z);(t^{\prime},z^{\prime})\big)$.
The resulting kernel is equivalent in the sense that the $k\times k$
determinants are the same, $\det\left[K\big((t_{i},z_{i});(t_{j},z_{j})\big)\right]_{i,j=1}^{k}=\det\left[\tilde{K}\big((t_{i},z_{i});(t_{j},z_{j})\big)\right]_{i,j=1}^{k}$.
This is because the factors of $g(t,z)$ and $g(t,z)^{-1}$
can be factored out of the rows and columns to cancel each other.
For this reason, the choice of kernel is not unique and it
is often helpful to choose a function $g(t,z)$ in order to get a kernel that is more convenient to work with.
\end{remark}

\subsection{Determinantal kernel for non-intersecting Brownian bridges}
\begin{definition}
\label{def:K} Fix any $\tf>0$. For $t\in(0,\tf)$, define the shorthand
$\dfac_{t}\defequal\sqrt{\frac{\tf}{2t(\tf-t)}}$. For $z,z^{\prime}\in\bR$
and $t,t^{\prime}\in(0,\tf)$, define the kernel $K^{(0,\tf)}\big((t,z);(t^{\prime},z^{\prime})\big)$
by:

\begin{align}
   & K^{(\tf,0)}\big((t,z);(t^{\prime},z^{\prime})\big)\label{eq:NIBb_kernel}\\
  \defequal & -\frac{1}{\sqrt{2\pi\left(t^{\prime}-t\right)}}\exp\left(-\frac{(z^{\prime}-z)^{2}}{2\left(t^{\prime}-t\right)}\right)\one\left\{ t<t^{\prime}\right\} \nonumber \\
   & +\left(\frac{\tf}{2t^{\prime}(\tf-t)}\right)^{\half}\;\sum_{j=0}^{d-1}\left(\frac{t(\tf-t^{\prime})}{(\tf-t)t^{\prime}}\right)^{j/2}p_{j}(z\dfac_{t})\exp\left(-\frac{z^{2}}{2(\tf-t)}\right)p_{j}(z^{\prime}\dfac_{t^{\prime}})\exp\left(-\frac{z^{\prime2}}{2t^{\prime}}\right).\nonumber 
\end{align}
where $p_{j}(y)$, $j\in\bN$, $y\in\bR$ are the normalized Hermite polynomials:
\begin{eqnarray}
p_{j}(y) & \defequal & \frac{H{}_{j}(y)}{\sqrt{\sqrt{\pi}\cdot j!\cdot2^{j}}},\label{eq:Hermite_normal}\\
H{}_{j}(y) & \defequal & (-1)^{j}e^{y^{2}}\frac{\d^{j}}{\d y^{j}}e^{-y^{2}}.\nonumber 
\end{eqnarray}
Finally, for any $\zf\in\bR$, we will define:
\begin{equation}
K^{(\tf,\zf)}\big((t,z);(t^{\prime},z^{\prime})\big)\defequal K^{(\tf,0)}\left(\Big(t,z-\zf\frac{t}{\tf}\Big);\Big(t^{\prime},z^{\prime}-\zf\frac{t^{\prime}}{\tf}\Big)\right)\frac{\exp\Big(\frac{\zf}{\tf}\big(z^{\prime}-\zf\frac{t^{\prime}}{\tf}\big)\Big)}{\exp\Big(\frac{\zf}{\tf}\big(z-\zf\frac{t}{\tf}\big)\Big)}.\label{eq:K_zf_tf}
\end{equation}
\end{definition}
\begin{lemma}
\label{lem:psi_is_det_K}Fix any $\tf>0$ and $\zf\in\bR$. Recall
from Definition \ref{def:NIBb} the $k$-point correlation functions
$\ps_{k}^{(\tf,\zf)}$ for non-intersecting Brownian bridges $\vec{D}^{(\tf,\zf)}$.
We have that $\ps_{k}^{(\tf,\zf)}$ is determinantal with kernel $K^{(\tf,\zf)}$:
\begin{equation}
\ps_{k}^{(\tf,\zf)}\big((t_{1},z_{1});\ld,;(t_{k},z_{k})\big)=\det\left[K^{(\tf,\zf)}\big((t_{i},z_{i});(t_{j},z_{j})\big)\right]_{i,j=1}^{k}.\label{eq:ps_is_detK}
\end{equation}
\end{lemma}

\begin{proof}
When $\zf=0$, the fact that $K^{(\tf,0)}$ is the determinantal kernel
for $\vec{D}^{(\tf,0)}$ is exactly equation (2.21) in \cite{Johansson2005-Hahn}.
For $\zf\neq0$, we notice that the shift of coordinates $\big(t,z\big)\to\big(t,z-\zf\frac{t}{\tf}\big)$
maps the non-intersecting Brownian bridges $\vec{D}^{(\tf,\zf)}$
to the bridges $\vec{D}^{(\tf,0)}$ in a measure preserving way. This shows that $K^{(\tf,0)}\left(\big(t,z-\zf\frac{t}{\tf}\big);\big(t^\prime,z^{\prime}-\zf\frac{t^{\prime}}{\tf}\big)\right)$
is a determinantal kernel for $\vec{D}^{(\tf,\zf)}$. The multiplication by the factor $\exp\Big(\frac{\zf}{\tf}\big(z^{\prime}-\zf\frac{t^{\prime}}{\tf}\big)\Big)\exp\Big(\frac{\zf}{\tf}\big(z-\zf\frac{t}{\tf}\big)\Big)^{-1}$
yields an equivalent kernel, as explained in Remark \ref{rem:equiv_kernels}. \end{proof}
\begin{lemma}
\label{lem:bounds_on_K}Fix any $\tf>0$ and $\zf\in\bR$. For any
$\de,M>0$, there exist constants $C_{K}^{<}=C_{K}^{<}(\de,M),\ C_{K}^{\geq}=C_{K}^{\geq}(\de,M)$
so that for all pairs $\big((t,z);(t^{\prime},z^{\prime})\big)$ that satisfy
$t,t^{\prime}\in(\de,\tf-\de)$ and $z,z^{\prime}\in(-M,M)$ we have:
\[
\begin{array}{lclr}
\abs{K^{(\tf,\zf)}\big((t,z);(t^{\prime},z^{\prime})\big)} & \leq & C_{K}^{<} \left(t^{\prime}-t\right)^{-1/2} & \text{if }t<t^{\prime},\\
\abs{K^{(\tf,\zf)}\big((t,z);(t^{\prime},z^{\prime})\big)} & \leq & C_{K}^{\geq} & \text{ if }t\geq t^{\prime}.
\end{array}
\]
\end{lemma}
\begin{proof}
First notice that the multiplicative factor that appears in equation
(\ref{eq:K_zf_tf}) is bounded here by a constant we denote by $C_{\times}=C_{\times}(M)$:
\[
\frac{\exp\left(\frac{\zf}{\tf}\left(z^{\prime}-\zf\frac{t^{\prime}}{\tf}\right)\right)}{\exp\left(\frac{\zf}{\tf}\left(z-\zf\frac{t}{\tf}\right)\right)}\leq\exp\left(2\frac{\zf}{\tf}\left(M+\abs{\zf}\right)\right)\defequal C_{\times}.
\]
Now, if $t\geq t^{\prime}$, then inspecting Definition \ref{def:K}
reveals that $K^{(\tf,\zf)}$ consists of a sum of $d$ terms. By
the hypothesis on $t,t^{\prime}$ we have $\dfac_{t}\leq\sqrt{\frac{\tf}{2\de^{2}}}$
and $\dfac_{t^{\prime}}\leq\sqrt{\frac{\tf}{2\de^{2}}}$. From this we
notice that the domain $\left\{ \left(z-\zf\frac{t}{\tf}\right)\dfac_{t}:\ (t,z)\in\left(-M,M\right)\times(\de,\tf-\de)\right\} \subset\bR$
is a bounded set, so we have for each $0\leq j\leq d-1$ a constant $C_{j}=C_{j}(\de,M)<\infty$
defined by:
\[
C_{j}\defequal\sup_{(t,z)\in(\de,\tf-\de)\times(-M,M)}\abs{p_{j}\left(\left(z-\zf\frac{t}{\tf}\right)\dfac_{t}\right)}.
\]
Thus, using the bounds on $t,t^{\prime}$ from the hypothesis and
the definition $K^{(\tf,\zf)}$ we have for $t\geq t^{\prime}$:
\[
\abs{K^{(\tf,\zf)}\big((t,z);(t^{\prime},z^{\prime})\big)}\leq C_{\times}\cdot\left(\frac{\tf}{2\de^{2}}\right)^{\half}\sum_{j=0}^{d-1}\left(\frac{\tf}{\de}\right)^{j}C_{j}^{2}.
\]
This is some constant depending on $\de,M$ as desired. When $t<t^{\prime}$,
there is an additional term in $K^{(\tf,\zf)}$. Using the bound $\sqrt{t^{\prime}-t}<\sqrt{\tf}$
and the triangle inequality now gives
\[
\abs{K^{(\tf,\zf)}\big((t,z);(t^{\prime},z^{\prime})\big)}\leq\frac{C_{\times}}{\sqrt{t^{\prime}-t}}\left(\frac{1}{\sqrt{2\pi}}+\sqrt{\tf}\left(\frac{\tf}{2\de^{2}}\right)^{\half}\sum_{j=0}^{d-1}\left(\frac{\tf}{\de}\right)^{j}C_{j}^{2}\right).
\]
This is some other constant that depends on $\de,M$ as desired.\end{proof}
\begin{corollary}
\label{cor:bounds_on_K_D1}Fix any $\tf>0$ and $\zf\in\bR$. For
any $\de,\et,M>0$ there exists a constant $C_{D_{1},K}=C_{D_{1},K}(\de,\et,M)$
such that $\abs{K^{(\tf,\zf)}\big((t,z);(t^{\prime},z^{\prime})\big)}\leq C_{D_{1},K}$
for all pairs $\big((t,z);(t^{\prime},z^{\prime})\big)$ that satisfy $t,t^{\prime}\in(\de,\tf-\de)$,
$\abs{t^{\prime}-t}>\et$ and $z,z^{\prime}\in(-M,M)$.\end{corollary}
\begin{proof}
For such $(t,z);(t^{\prime},z^{\prime})$, the bound $\sqrt{t^{\prime}-t}>\sqrt{\et}$
and the result from Lemma \ref{lem:bounds_on_K} show together that
the constant $C_{D_{1},K}=\max\left(C_{K}^{\geq},\frac{1}{\sqrt{\et}}C_{K}^{<}\right)$
works for the stated inequality.
\end{proof}
\subsection{Determinantal kernel for non-intersecting random walk bridges}
\begin{definition}
\label{def:K_N}The Hahn polynomials are a family of orthogonal polynomials
depending on three parameters $\al,\be,N$ and given explicitly
in terms of the hypergeometric function $_{3}F_{2}$ by:
\begin{equation}
Q_{j}(x,\al,\be,M)={_{3}F_{2}}\binom{-j,-j+\al+\be+1,-x}{\al+1,-M}(1).\label{eq:hyp_geom}
\end{equation}
See \cite{koekoek1998askey_scheme} for extensive details of the Hahn polynomials.
Fix $\nf\in\bN$ and $\xf\in\bZ$ with $\nf+\xf\zmtwo$. For any $x\in\bZ,n\in\bN$
with $x+n\zmtwo$ define now $P_{j}^{(\nf,\xf)}(n,x)$ and $\tilde{P}_{j}^{(\nf,\xf)}(n,x)$ in terms of Hahn polynomials with parameters depending
on $n,x,\xf,\nf$:
\begin{align*}
P_{j}^{(\nf,\xf)}(n,x) \defequal & Q_{j}\left(\frac{n+x}{2},-\frac{\nf+\xf}{2}-d,-\frac{\nf-\xf}{2}-d,\nf+d-1\right),  \\
\tilde{P}_{j}^{(\nf,\xf)}(n,x) \defequal & Q_{j}\left(\frac{\nf-n+x-\xf}{2},-\frac{\nf-\xf}{2}-d,-\frac{\nf+\xf}{2}-d,\nf-n+d-1\right).
\end{align*}
Define for each $0\leq j\leq d-1$:
\begin{equation}
F_{j}^{(\nf,\xf)}\defequal\frac{\left(\nf+2d-2j-1\right)(\nf+2d-j)_{(j-1)}}{j!}\binom{\nf+2d-2}{\half\nf+\half\xf+d-1}^{-1}, \nonumber
\end{equation}
where we use the notation for the Pochhammer symbol $(x)_{m}=x(x+1)\cdots(x+m-1)$.
For $x,x^{\prime}\in\bZ$, $n,n^{\prime}\in\bN$, that have $x+n\zmtwo$,
$x^{\prime}+n^{\prime}\zmtwo$ define the kernel:
\begin{align}
  & K_{RW}^{(\nf,\xf)}\left(\left(n,x\right);(n^{\prime},x^{\prime})\right) \label{eq:NIWb_kernel} \\
  \defequal & 2^{n-n^{\prime}}\binom{n^{\prime}-n}{\half\left(n^{\prime}-n\right)+\half\left(x^{\prime}-x\right)}\one\left\{ n<n^{\prime}\right\} \nonumber \\
  & +2^{n-n^{\prime}}\sum_{j=0}^{d-1}F_{j}^{(\nf,\xf)} P_{j}^{(\nf,\xf)}(n^{\prime},x^{\prime})\binom{n^{\prime}+d-1}{\half n^{\prime}+\half x^{\prime}}\tilde{P}_{j}^{(\nf,\xf)}(n,x)\binom{\nf-n+d-1}{\half(\nf-n)+\half(x-\xf)}.\nonumber 
\end{align}
Finally, for any $N\in\bN$, $\zf\in\bR$, $\tf>0$, and any pair of coordinates
$\big((t,z);(t^{\prime},z^{\prime})\big)\in\big((0,\tf)\times\bR\big)^{2}$
define the rescaled kernel $K^{(N),(\tf,\zf)}\big(\left(t,z\right);\left(t^{\prime},z^{\prime}\right)\big)$
by (recall the notation $(t,z)_{2}$ from Definition \ref{def:T_N_and_rounding})
\[
K^{(N),(\tf,\zf)}\Big(\big(t,z\big);\big(t^{\prime},z^{\prime}\big)\Big)\defequal\frac{\sqrt{N}}{2}K_{RW}^{(N\tf,\sqrt{N}\zf)_{2}}\left(\big(Nt,\sqrt{N}z\big)_{2};\big(Nt^\prime,\sqrt{N}z^\prime\big)_{2}\right).
\]
\end{definition}
\begin{lemma}\label{lem:psiN_is_det_KN}
Fix any $\zf\in\bR$,$\tf>0$ and $k \in \bN$. We have that $\ps_k^{(N),(\tf,\zf)}$ is determinantal with kernel $K^{(N),(\tf,\zf)}$, namely:
\begin{equation}
\ps_k^{(N),(\tf,\zf)}\big((t_{1},z_{1}),\ld,(t_{k},z_{k})\big) = \det\left[K^{(N),(\tf,\zf)}\big((t_i,z_i);(t_j,z_j)\big)\right]_{i,j=1}^{k}.\label{eq:psN_is_det_KN}
\end{equation}
\end{lemma}
\begin{proof}
It suffices to show that for $\xf\in\bZ,\nf\in\bN$ that $K_{RW}^{(\xf,\nf)}\big(\left(n,x\right);(n^{\prime},x^{\prime})\big)$
is the determinantal kernel for $d$ non-intersecting simple symmetric
random walk bridges that start at $\X(0)=\vec{\de}(0)$ and end
at $\vec{X}(\nf)=\vec{\de}(\xf)$; the result then follows
by the scaling in equation (\ref{eq:def_ps_k_N}). $K_{RW}^{(\nf,\xf)}$ is a rewriting
of the determinantal kernel for non-intersecting simple symmetric
random walks that appears in \cite{Johansson2005-Hahn} with the identification
of the three main parameters $a,b,c$ from \cite{Johansson2005-Hahn}
to the parameters $d,\xf,\tf$ in our setting by $a=d$, $b=\half\nf-\half\xf,c=\half\nf+\half\xf.$ 

Specifically, comparing the definition of $F_{j}^{(\nf,\xf)}$  and the definitions of $P_{j}^{(\nf,\xf)}(n,x)$, $\tilde{P}_{j}^{(\nf,\xf)}(n,x)$
from Definition \ref{def:K_N}, to equation (3.21), equation (3.22)
and equation (3.29) in \cite{Johansson2005-Hahn}, we have the following
identification to the quantities $C_{j}(a,b,c)$, $\ph_{0,n}\left(j,x\right)$
and $\ph_{n,b+c}\left(x,j\right)$ that appear in that paper:
\begin{eqnarray*}
\frac{1}{d!}P_{j}^{(\nf,\xf)}(n,x)\cdot\binom{n+d-1}{\half n+\half x} & = & \ph_{0,n}\left(j,x\right),\\
\frac{1}{d!}\tilde{P}_{j}^{(\nf,\xf)}(n,x)\cdot\binom{\nf-n+d-1}{\half\left(\nf-n\right)+\half\left(x-\xf\right)} & = & \ph_{n,b+c}\left(x,j\right),\\
\left(d!\right)^{2}F_{j}^{(\nf,\xf)} & = & C_{j}(a,b,c).
\end{eqnarray*}
With this identification, the fact that $2^{n^{\prime}-n}K_{RW}^{(\nf,\xf)}\big((n,x);(n^{\prime},x^{\prime})\big)$
is the determinantal kernel for simple symmetric random walks is equation
(3.24) in \cite{Johansson2005-Hahn}. Finally, we have multiplied
by the factor $2^{-\left(n^{\prime}-n\right)}$ which yields an equivalent
kernel, see Remark \ref{rem:equiv_kernels}.
\end{proof}
\begin{remark}
Another description of the kernel for $\ps^{(N),(\tf,\zf)}$ also appears in \cite{Gorin-Hahn}. The parameters chosen in the Hahn polynomials $P_j^{(\nf,\xf)}$ and $\tilde{P}_j^{(\nf,\xf)}$ are motivated by the choice of parameters from this paper. By using hypergeometric identities, it is possible to show that the kernel there is equivalent to the kernel $K_{RW}^{(N),(\nf,\xf)}$.
\end{remark}

\subsection{Pointwise convergence on  \texorpdfstring{$D_{1}(\de,\et,M)$}{D1(de,et,M)}{\:\textendash\:}proof of Proposition
\ref{prop:D1} }
\begin{lemma}
\label{lem:hahn_to_hermite}Fix $0<p<1$, $c\in\bR$, $\ga\in\bR\backslash\{0\}$
such that $1+\ga^{-1}>0$, and a compact set $E\subset\bR.$ Let $p_{M}=p+c{M}^{-\half}$
and let $\tilde{y}_{M},\al_{M},\be_{M}$ be such that $\tilde{y}_{M}=p_{M}M+y\sqrt{2p(1-p)M(1+\ga^{-1})}+O(1), $
$\al_{M}=\ga p_{M}M+O(1)$ and $\be_{M}=\ga(1-p_{M})M+O(1)$ as $M\to\infty$.
Define the polynomials $G_{j}^{(M)}(y)$ in terms of the Hahn polynomials
$Q_{j}(x,\al,\be,M)$ by: 
\begin{eqnarray*}
G_{j}^{(M)}(y) & \defequal & (-1)^{j}\sqrt{\binom{M}{j}2^{j}j!\left(\frac{p}{1-p}\right)^{j}\left(\frac{\ga}{1+\ga}\right)^{j}}Q_{j}\left(\tilde{y}_{M},\al_{M},\be_{M},M\right).
\end{eqnarray*}
Then we have pointwise convergence of these polynomials to the Hermite
polynomials $H_{j}$ as $M\to\infty$, uniformly over all $y\in E$. In fact:
\[
G_{j}^{(M)}(y)=H{}_{j}(y)+O\big(M^{-\half}\big).
\]
\end{lemma}
\begin{proof}
This is a very slight extension of Theorem A.1. from \cite{Joh-Nord-GUE-Minors}
where we allow the $O(1)$ corrections in the parameters $\tilde{y}_{M},\al_{M},\be_{M}$
and the $O\big(M^{-\half}\big)$ correction on the factor $p_{M}=p+cM^{-\half}$.
An inspection of the proof there shows that these corrections give
the same asymptotic three term recurrence for the polynomials $G_{j}^{(M)}(y)$,
namely:
\[
G_{j+1}^{(M)}(y)=\left(2y+O\big(M^{-\half}\big)\right)G_{j}^{(M)}(y)+\left(2j+O\big(M^{-\half}\big)\right)G_{j-1}^{(M)}(y).
\]
With this observation, the convergence follows exactly by the same
induction argument as in Theorem A.1 in \cite{Joh-Nord-GUE-Minors}
.
\end{proof}
\begin{corollary}
\label{cor:Pj_to_Hj} Fix any $\tf>0$ and $\zf\in\bR$. Recall the definition of the polynomials $P_{j}$,  $\tilde{P}_{j}$ from Definition \ref{def:K_N} and the factor $\dfac_t$ from Definition \ref{def:K}. For any $\de,M>0$,
we have the following limit as $N\to\infty$, uniformly over the set $(t,z) \in (\de,\tf-\de) \times (-M,M)$:
\begin{align*}
\sqrt{N}^{j}P_{j}^{(N\tf,\sqrt{N}\zf)_{2}}\Big(\big(Nt,\sqrt{N}z\big)_{2}\Big) &= \left(\frac{-1}{\sqrt{2}}\right)^{j}\left(\frac{\tf-t}{\tf t}\right)^{j/2}H{}_{j}\left(\Big(z-\zf\frac{t}{\tf}\Big)\dfac_{t}\right)+O\big({N}^{-\half}\big),\\
\sqrt{N}^{j}\tilde{P}_{j}^{(N\tf,\sqrt{N}\zf)_{2}}\left(\big(Nt,\sqrt{N}z\big)_{2}\right) &= \left(\frac{-1}{\sqrt{2}}\right)^{j}\left(\frac{\tf t}{\tf-t}\right)^{j/2}H{}_{j}\left(\Big(z-\zf\frac{t}{\tf}\Big)\dfac_{t}\right)+O\big({N}^{-\half}\big).
\end{align*}
\end{corollary}
\begin{proof}
This follows by the definition $P_{j}$ and $\tilde{P}_{j}$ in terms
of Hahn polynomials from Definition \ref{def:K_N} and the asymptotics
from Lemma \ref{lem:hahn_to_hermite}. For $P_{j}$ the parameters
from Lemma \ref{lem:hahn_to_hermite} are fixed as $p=\half$, $c=\half\sqrt{t}\frac{\zf}{\tf}$,
$M=\floor{tN}+d-1$, $\ga=-\frac{\tf}{t}$ and $y=\left(z-\zf\frac{t}{\tf}\right)\dfac_{t}$.
For $\tilde{P}_{j}$ the parameters are fixed as $p=\half,c=-\half\sqrt{\tf-t}\frac{\zf}{\tf},M=\floor{(\tf-t)N}+d-1$, $\ga=-\frac{\tf}{\tf-t}$
and $y=\left(z-\zf\frac{t}{\tf}\right)\dfac_{t}.$
\end{proof}

\begin{lemma}
\label{lem:KN_to_K}Fix $\tf>0$ and $\zf\in\bR$. For any
$\de,\et,M>0$, we have the following pointwise convergence uniformly
over all pairs $\left(t,z\right);\left(t^{\prime},z^{\prime}\right)$
that satisfy $z,z^{\prime}\in\left(-M,M\right)$, $t,t^{\prime}\in(\de,\tf-\de)$
and $\abs{t-t^{\prime}}>\et$:
\[
\lim_{N\to\infty}K^{(N),(\tf,\zf)}\left(\left(t,z\right);\left(t^{\prime},z^{\prime}\right)\right)=K^{(\tf,\zf)}\big((t,z);(t^{\prime},z^{\prime})\big).
\]
\end{lemma}
\begin{proof}
Define for convenience the variables (which depend on $N$), $n,n^{\prime},\nf\in\bN$  and $x,x^{\prime},\xf\in\bZ$ by $\big(n^{\prime},x^{\prime}\big)\defequal\big(Nt^{\prime},\sqrt{N}z^{\prime}\big)_{2}$,
$\big(n,x\big)\defequal\big(Nt,\sqrt{N}z\big)_{2}$, $\big(\nf,\xf\big)\defequal\big(N\tf,\sqrt{N}\zf\big)_{2}$
(recall the notation $(t,z)_{2}$ from Definition \ref{def:T_N_and_rounding})).
By inspecting equation (\ref{eq:NIBb_kernel}) and equation (\ref{eq:NIWb_kernel}),
we see that both kernels consist of a sum of $d+1$ terms. We will show
convergence of each term separately. In the convergence of each term,
we will use the local central limit theorem for binomial coefficients
(see e.g. Theorem 3.5.2 in \cite{durrett2010probability}) that: 
\[
\lim_{M\to\infty}\sup_{\ell\in\bZ}\abs{\sqrt{M}2^{-M}\binom{M}{\ell}-\sqrt{\frac{2}{\pi}}\exp\left(-\frac{\left(2\ell-M\right)^{2}}{2M}\right)}=0.
\]

The convergence of the first term in equation (\ref{eq:NIBb_kernel})
and equation (\ref{eq:NIWb_kernel}) is a direct application of this
result. Notice that uniformly over all $t^{\prime},t$ with $\abs{t^{\prime}-t}>\et$
that we have $n^{\prime}-n>N\et$. By application of the local central
limit theorem we have uniformly over all such $t^{\prime},t$ and
any choice of $z,z^{\prime}$ that
\[
\lim_{N\to\infty}\frac{\sqrt{N}}{2}2^{n-n^{\prime}}\binom{n^{\prime}-n}{\half\left(n^{\prime}-n\right)+\half\left(x^{\prime}-x\right)} =\frac{1}{\sqrt{2\pi\left(t^{\prime}-t\right)}}\exp\left(-\frac{(z^{\prime}-z)^{2}}{2\left(t^{\prime}-t\right)}\right),
\]
and it is clear that $\one\left\{ n<n^{\prime}\right\}=\one\left\{ t<t^{\prime}\right\}$. It remains to see convergence of the remaining $d$ terms in $K^{(\tf,\zf)}$
and $K^{(N),(\tf,\zf)}$. Focusing attention on the $j$-th term of
the sum in the definition of $K^{(\tf,\zf)}$ equation (\ref{eq:NIWb_kernel}),
we observe that since $t>\de$ we have $n^{\prime}>\de N$ and since
$\tf-t^{\prime}>\de$ we have $\nf-n^{\prime}>\de N$. By application
of the local central limit theorem, we have uniformly over all such
$t^{\prime},t$ and any choice of $z,z^{\prime}\in\bR$ that:
\begin{eqnarray*}
\lim_{N\to\infty}\frac{\sqrt{N}}{2}2^{-\left(n^{\prime}+d-1\right)}\binom{n^{\prime}+d-1}{\half n^{\prime}+\half x^{\prime}} & = & \frac{1}{\sqrt{2\pi t^{\prime}}}\exp\left(-\frac{z^{\prime2}}{2t^{\prime}}\right),\\
\lim_{N\to\infty}\frac{\sqrt{N}}{2}2^{-\left(\nf-n+d-1\right)}\binom{\nf-n+d-1}{\half(\nf-n)+\half(x-\xf)} & = & \frac{1}{\sqrt{2\pi\left(\tf-t\right)}}\exp\left(-\frac{z^{2}}{2\left(\tf-t\right)}\right).
\end{eqnarray*}
Since $\nf>\tf N\to\infty$ we also have
\[
\lim_{N\to\infty}\frac{\sqrt{N}}{2}2^{-(\nf+2d-2)}\binom{\nf+2d-2}{\half\nf+\half\xf+d-1}=\frac{1}{\sqrt{2\pi\tf}}\exp\left(-\frac{\zf^{2}}{2\tf}\right).
\]
By the definition of $F_{j}^{(\nf,\xf)}$ from Definition \ref{def:K_N}
then
\[
\lim_{N\to\infty}\frac{2}{\sqrt{N}}\frac{2^{(\nf+2d-2)}F_{j}^{(\nf,\xf)}}{N^{j}}=\frac{1}{j!}\sqrt{2\pi}(\tf)^{j+\half}\exp\left(\frac{\zf^{2}}{2\tf}\right).
\]
Combining these asymptotics with the asymptotics for $P_{j}$ and
$\tilde{P}_{j}$ from Corollary \ref{cor:Pj_to_Hj} we have the following
limit for the $j$-th term in the sum from equation (\ref{eq:NIWb_kernel}):
\begin{align*}
  & \lim_{\mathclap{N\to\infty}}\frac{\sqrt{N}}{2}2^{n-n^{\prime}}F_{j}^{(\nf,\xf)} P_{j}^{(\nf,\xf)}(n^{\prime},x^{\prime})\binom{n^{\prime}+d-1}{\half n^{\prime}+\half x^{\prime}}\tilde{P}_{j}^{(\nf,\xf)}(n,x)\binom{\nf-n+d-1}{\half(\nf-n)+\half(x-\xf)}\\
 = & \lim_{\mathclap{N\to\infty}}\left(\frac{2}{\sqrt{N}}2^{\nf+2d-2}N^{-j}F_{j}^{(\nf,\xf)}\right)\left(\sqrt{N}^{j}P_{j}^{(\nf,\xf)}(n^{\prime},x^{\prime})\right)\left(\frac{\sqrt{N}}{2}2^{-(n^{\prime}+d+1)}\binom{n^{\prime}+d-1}{\half n^{\prime}+\half x^{\prime}}\right)\\
  & \times\left(\sqrt{N}^{j}\tilde{P}_{j}^{(\nf,\xf)}\right)\left(\frac{\sqrt{N}}{2}2^{-(\nf-n+d+1)}\binom{\nf-n+d-1}{\half(\nf-n)+\half(x-\xf)}\right)\\
 = & \frac{1}{\sqrt{2\pi}}\frac{1}{j!}\left(\frac{\tf}{t^{\prime}(\tf-t)}\right)^{\half}\left(\frac{t(\tf-t^{\prime})}{(\tf-t)t^{\prime}}\right)^{j/2}\left(-\frac{1}{\sqrt{2}}\right)^{j}H_{j}\left(\left(z^{\prime}-\zf\frac{t^{\prime}}{\tf}\right)\dfac_{t^{\prime}}\right)\exp\left(-\frac{z^{\prime2}}{2t^{\prime}}\right)\\
  & \times\left(-\frac{1}{\sqrt{2}}\right)^{j}H_{j}\left(\left(z-\zf\frac{t}{\tf}\right)\dfac_{t}\right)\exp\left(-\frac{\left(z-\zf\right)^{2}}{2(\tf-t)}\right).
\end{align*}
Keeping in mind the normalization of the Hermite polynomials from
equation (\ref{eq:Hermite_normal}), we see that this is exactly the
corresponding $j$-th term in equation (\ref{eq:NIBb_kernel}), as
desired.\end{proof}
\begin{remark} When $\zf=0$, Lemma \ref{lem:KN_to_K} confirms
equation (3.36) from \cite{Johansson2005-Hahn}.\end{remark}
\begin{corollary}
\label{cor:bounds_on_KN}Fix $\tf>0$ and $\zf\in\bR$. For any choice
of parameters $\de,M>0$, there exist constants $C_{K}^{<}=C_{K}^{<}(\de,M), $
and $C_{K}^{\ge}=C_{K}^{\geq}(\de,M)$ so that for pairs $(t,z);(t^{\prime},z^{\prime})$
with $t,t^{\prime}\in(\de,T-\de)$ and
$z,z^{\prime}\in(-M,M)$ we have
\[
\begin{array}{lclr}
\sup_{N}\abs{K^{(N),(\tf,\zf)}\big((t,z);(t^{\prime},z^{\prime})\big)} & \leq & {C_{K}^{<}}\left({t^{\prime}-t}\right)^{-\half} & \text{if }t<t^{\prime},\\
\sup_{N}\abs{K^{(N),(\tf,\zf)}\big((t,z);(t^{\prime},z^{\prime})\big)} & \leq & C_{K}^{\geq} & \text{ if }t\geq t^{\prime}.
\end{array}
\]
\end{corollary}
\begin{proof}
When $t\geq t^{\prime}$, the first term in the definition of $K^{(N),(\tf,\zf)}$
and $K^{(\tf,\zf)}$ vanishes, and the proof of Lemma \ref{lem:KN_to_K}
shows that that regardless of $\et, $ $K^{(N),(\tf,\zf)}$ converges
uniformly to $K^{(\tf,\zf)}$ on the set $t,t^{\prime}\in(\de,\tf-\de)$
and $z,z^{\prime}\in(-M,M)$. Thus when $t\geq t^{\prime}$, since
$K^{(\tf,\zf)}$ is bounded by $C_{K}^{\geq}$ here by Lemma \ref{lem:bounds_on_K},
and since the convergence in Lemma \ref{lem:KN_to_K} is uniform,
hence $K^{(N),\left(\tf,\zf\right)}$ is also bounded with a possibly
larger constant. By enlarging if necessary, we denote by $C_{K}^{\geq}$
a constant large enough to bound both of them. 

To see the bound when $t<t^{\prime}$, it remains to study the effect
of the first term. To do this we use the bound on binomial coefficients
which follows from Stirling's formula: 
\[
\sup_{M\in\bN}\sup_{\ell}\sqrt{M}2^{-M}\binom{M}{\ell}\leq\sup_{M\in\bN}\sqrt{M}2^{-M}\binom{M}{\floor{\half M}}\leq C_{Binom},
\]
 Applying this bound to the first term in $K^{(N),(\tf,\zf)}$, along
with the triangle inequality and the bound $\sqrt{t^{\prime}-t}<\sqrt{\tf}$
we conclude that:
\[
\sup_{N}\sqrt{t^{\prime}-t}\abs{K^{(N),(\tf,\zf)}\big((t,z);(t^{\prime},z^{\prime})\big)}\leq\left(C_{Binom}+\sqrt{\tf}C_{K}^{\geq}\right).
\]
This is some constant depending on $\de,M$ as desired. By enlarging
the already defined constant $C_{K}^{<}$ if necessary, we may denote by
$C_{K}^{<}$ a constant large enough to bound both.\end{proof}
\begin{corollary}
\label{cor:bounds_on_KN_D1}Fix $\tf>0$ and $\zf\in\bR$. For any
choice of $\de,\et,M>0$, there exists a constant $C_{D_{1},K}=C_{D_{1},K}(\de,\et,M,\tf,\zf)$
such that $\abs{K^{(N),(\tf,\zf)}\big((t,z);(t^{\prime},z^{\prime})\big)}\leq C_{D_{1},K}$
for all pairs $(t,z);(t^{\prime},z^{\prime})$ that satisfy $t,t^{\prime}\in(\de,\tf-\de)$,
$\abs{t^{\prime}-t}>\et$ and $z,z^{\prime}\in(-M,M)$.\end{corollary}
\begin{proof}
Use the bound $\sqrt{t^{\prime}-t}>\sqrt{\et}$ and the result from
Corollary \ref{cor:bounds_on_KN} to see that the constant $C_{D_{1},K}=\max\left(C_{K}^{\geq},\frac{1}{\sqrt{\et}}C_{K}^{<}\right)$
will do.
\end{proof}

\begin{proof}
(Of Proposition \ref{prop:D1}) By Lemma \ref{lem:psi_is_det_K} and
Lemma \ref{lem:psiN_is_det_KN}, $\ps_{k}^{(\tf,\zf)}$ and $\ps_{k}^{(N),(\tf,\zf)}$
are given by $k\times k$ determinants of the kernels $K^{(\tf,\zf)}$
and $K^{(N)(\tf,\zf)}$ respectively. The stated bounds by $C_{D_{1}}(\de,\et)$
follows by the bound for $\abs{K^{(\tf,\zf)}(\cdot)}\leq C_{D_{1},K}$
in Corollary \ref{cor:bounds_on_K_D1} and the bound for $\abs{K^{(N),(\tf,\zf)}(\cdot)}<C_{D_{1},K}$
in Corollary \ref{cor:bounds_on_KN_D1}. Finally, by Lemma \ref{lem:KN_to_K},
we have uniform convergence $K^{(N),(\tf,\zf)}\big((t_{i},z_{i});(t_{j},z_{j})\big)\to K^{(\tf,\zf)}\big((t_{i},z_{i});(t_{j},z_{j})\big)$
for any pairs $\left(t_{i},z_{i}\right)$ and $\left(t_{j},z_{j}\right)$
chosen from the $k$-tuple $\big((t_{1},z_{1}),\ld,(t_{k},z_{k})\big)\in D_{1}(\de,\et,M)$.
Since determinants are polynomials of the entries, and the entries
are always bounded by $C_{D_{1},K}$, the uniform convergence of the
entries implies uniform convergence of the whole determinant, yielding
the desired result. 
\end{proof}

\subsection{Bounds on  \texorpdfstring{$D_{2}(\de, \et, M)$}{D2(de,et,M)}{\:\textendash\:}proof of Proposition \ref{prop:D2}}
\begin{lemma}
\label{lem:detK_bound} Fix $\tf>0.$ Also fix subsets $I_{z}\subset\bR$
and $I_{t}\subset(0,\tf)$. Suppose that $K\big((t,z);(t^{\prime},z^{\prime})\big)$,
$z,z^{\prime}\in I_{z}, $ $t,t^{\prime}\in I_{t}$ is any determinantal
kernel for which there are constants $C_{1}$ and $C_{2}$ so that
we have the following bounds
\[
\begin{array}{lclr}
\abs{K\big((t,z);(t^{\prime},z^{\prime})\big)} & \leq & {C_{1}}\left({t^{\prime}-t}\right)^{-\half} & \text{if }t<t^{\prime},\\
\abs{K\big((t,z);(t^{\prime},z^{\prime})\big)} & \leq & C_{2} & \text{if }t\geq t^{\prime}.
\end{array}
\]
Then, for any $k\in\bN$, there exists a constant $C_{sq}=C_{sq}(k,C_{1},C_{2})$
so that for all $\big((t_{1},z_{1});\ld;(t_{k},z_{k})\big)\in\left(I_{t}\times I_{z}\right)^{k}$ with $0<t_1<\ldots<t_k<\tf$ we have the bound
\[
\det\Big[K\big((t_{i},z_{i});(t_{j},z_{j})\big)\Big]_{i,j=1}^{k}\leq\frac{C_{sq}}{\sqrt{t_{2}-t_{1}}\sqrt{t_{3}-t_{2}}\cdots\sqrt{t_{k}-t_{k-1}}}.
\]

\end{lemma}

\begin{proof}
Consider $\sqrt{t_{2}-t_{1}}\sqrt{t_{3}-t_{2}}\cdots\sqrt{t_{k}-t_{k-1}}\det\Big[K\big((t_{i},z_{i});(t_{j},z_{j})\big)\Big]_{i,j=1}^{k}$.
By pulling the factors into the rows of the determinant, this is:
\begin{align*}
 &\sqrt{t_{2}-t_{1}}\sqrt{t_{3}-t_{2}}\cdots\sqrt{t_{k}-t_{k-1}}\det\Big[K\big((t_{i},z_{i});(t_{j},z_{j})\big)\Big]_{i,j=1}^k\\[5pt]
 &=\det\left[\begin{array}{ccc}
\sqrt{t_{2}-t_{1}}K\big(\left(t_{1},z_{1}\right);\left(z_{1},t_{1}\right)\big) & .. & \sqrt{t_{2}-t_{1}}K\big(\left(t_{1},z_{1}\right);\left(t_{k},z_{k}\right)\big)\\
\sqrt{t_{3}-t_{2}}K\big(\left(t_{2},z_{2}\right);\left(z_{1},t_{1}\right)\big) & .. & \sqrt{t_{3}-t_{2}}K\big(\left(t_{2},z_{2}\right);\left(t_{k},z_{k}\right)\big)\\
\vdots & .. & \vdots\\
\sqrt{t_{k}-t_{k\mo}}K\big(\left(t_{k\mo},z_{k\mo}\right);\left(t_{1},z_{1}\right)\big) & .. & \sqrt{t_{k}-t_{k\mo}}K\big(\left(t_{k\mo},z_{k\mo}\right);\left(t_{k},z_{k}\right)\big)\\
K\big(\left(t_{k},z_{k}\right);\left(t_{1},z_{1}\right)\big) & .. & K\big(\left(t_{k},z_{k}\right);\left(t_{k},z_{k}\right)\big)
\end{array}\right].
\end{align*}

We now consider the entries above the diagonal and the entries on-or-below
the diagonal of this matrix separately. Above the diagonal, the $\left(i,j\right)$-th
entry with $j>i$ is bounded by:
\[
\sqrt{t_{i+1}-t_{i}}\abs{K\big((t_{i},z_{i});(t_{j},z_{j})\big)}\leq\sqrt{t_{j}-t_{i}}\abs{K\big((t_{i},z_{i});(t_{j},z_{j})\big)}\leq C_{1},
\]
by hypothesis. On-or-below the diagonal, we use the bound $\sqrt{t_{i+1}-t_{i}}\leq\sqrt{\tf}$
and $\abs{K\big((t_{i},z_{i});(t_{j},z_{j})\big)}$ $\leq C_{2}$
to conclude that these entries are bounded by $\sqrt{\tf}C_{2}.$
Thus all the entries of the matrix are bounded in absolute value by
$\max\left\{ C_{2},\sqrt{\tf}C_{2}\right\} $. Since all the entries
are bounded in this way, and determinants are polynomials of the entries,
hence the determinant is bounded by some constant which depends on
$k,C_{1},C_{2}$ and $\tf$ as desired.
\end{proof}

\begin{corollary}
\label{cor:psiN_bound} Fix $\tf>0$ and $\zf\in\bR$. For any $\de,M>0$,
there exists a constant $C_{D_{2}}=C_{D_{2}}(\de,M)$ such that for
all $\big((t_{1},z_{1});\ld;(t_{k},z_{k})\big)\in D_{2}\left(\de,\et,M\right)$
we have: 
\begin{equation}
\sup_{N}\ps_{k}^{(N),(\tf,\zf)}\big((t_{1},z_{1});\ld;(t_{k},z_{k})\big) \leq \frac{C_{D_{2}}}{\sqrt{t_{2}-t_{1}}\sqrt{t_{3}-t_{2}}\cdots\sqrt{t_{k}-t_{k-1}}}.
\end{equation}
\end{corollary}
\begin{proof}
This follows by applying Lemma \ref{lem:detK_bound} to the bounds
on $K^{(N),\left(\tf,\zf\right)}$ from Corollary \ref{cor:bounds_on_KN}
and then finally using the fact that $K^{(N),(\tf,\zf)}$ is the determinantal
kernel for $\ps_{k}^{(N),(\tf,\zf)}$ from Lemma \ref{lem:psiN_is_det_KN}.
\end{proof}

\begin{proof}
(Of Proposition \ref{prop:D2}) Recall from Definition \ref{def:scaled-NIWb}
that $\ps^{(N),(\tf,\zf)}$ is constant on the cells of $\bT^{(N)}$.
Thus, as in equation (\ref{eq:integral_is_sum}), we may rewrite the
integral as a sum over the discrete set of points in $\left(\bT^{(N)}\right)^{k}\cap D_{2}(\de,\et)$.
For convenience, we will define the set
\[
E(\de,\et)\defequal\left\{ \vec{t}\in\left(\bR^{+}\right)^{k}:\de<t_{1}<\ld<t_{k}<\tf-\de,t_{i+1}-t_{i}<\et\text{ for some }1\leq i\leq k\right\} .
\]
With this notation in hand, we now apply the bound from Corollary
\ref{cor:psiN_bound} on $\ps_{k}^{(N),(\tf,\zf)}(\vec{w})$ to get:
\begin{eqnarray}
 &  & \intop_{D_{2}(\de,\et)}\abs{\ps_{k}^{(N),(\tf,\zf)}(\vec{w})}^{2}\d\vec{w}\label{eq:D2_1}  \\
 & = & \left(\frac{2}{N\sqrt{N}}\right)^{k}\sum_{\vec{w}\in\left(\bT^{(N)}\right)^{k}\cap D_{2}(\de,\et)}\abs{\ps_{k}^{(N),(\tf,\zf)}(\vec{w})}^{2}\nonumber \\
 & \leq & \frac{1}{N^k}\sum_{\vec{t}\in E(\de,\et)\cap\frac{\bN^{k}}{N}}\frac{C_{D_{2}}}{\sqrt{t_{2}-t_{1}}\cdots\sqrt{t_{k}-t_{k-1}}}\sum_{\vec{z}\in\frac{\bZ^{k}}{\sqrt{N}}}\frac{2^k}{N^\frac{k}{2}}\ps_{k}^{(N),(\tf,\zf)}\big((t_{1},z_{1});\ld;(t_{k},z_{k})\big).\nonumber
\end{eqnarray}
We notice now from Definition \ref{def:scaled-NIWb} that $\frac{2^k}{N^{\frac{k}{2}}}\ps_{k}^{(N),(\tf,\zf)}(\y)=\p\left(\bigcap_{j=1}^{k}\left\{ \scriptstyle z_{j}\in\X^{(N),(\tf,\zf)}(t_{j})\right\} \right)$,
namely the probability of finding a particle occupying each position
$z_{1},\ld,z_{k}$ at the times $t_{1},\ld,t_{k}$ respectively. With
$\vec{t}$ fixed, summing these probabilities simply counts the $d$
particles: 
\begin{eqnarray*}
\sum_{\vec{z}\in\frac{\bZ^{k}}{\sqrt{N}}}\p\left(\bigcap_{j=1}^{k}\left\{ z_{j}\in\X^{(N),(\tf,\zf)}(t_{j})\right\} \right) & = & \e\left[ \prod_{j=1}^{k} \left(\sum_{z_{j}\in\frac{\bZ}{\sqrt{N}}}\one\left\{ z_{j}\in\X^{(N),(\tf,\zf)}(t_{j})\right\} \right)\right]\\
 & = & \e\left[d^{k}\right]=d^{k}.
\end{eqnarray*}
 Thus the sum over $\vec{z}\in\bZ^{k}/\sqrt{N}$ in equation (\ref{eq:D2_1})
gives $d^{k}$. We finally recognize the remaining piece as a Riemann
sum approximation to an integral, thus concluding that:
\begin{eqnarray}
\intop_{D_{2}(\de,\et)}\abs{\ps_{k}^{(N),(\tf,\zf)}(\vec{w})}^{2}\d\vec{w} & \leq & d^{k}C_{D_{2}}\frac{1}{N^{k}}\sum_{\vec{t}\in E(\de,\et)\cap\frac{\bN^{k}}{N}}\frac{1}{\sqrt{t_{2}-t_{1}}\sqrt{t_{3}-t_{2}}\cdots\sqrt{t_{k}-t_{k-1}}}\nonumber \\
 & \leq & d^{k}C_{D_{2}}\intop_{\vec{t}\in E(\de,\et)}\frac{\d t_{1}\d t_{2}\ld\d t_{k}}{\sqrt{t_{2}-t_{1}}\sqrt{t_{3}-t_{2}}\cdots\sqrt{t_{k}-t_{k-1}}}.\label{eq:D2_2}
\end{eqnarray}
Since $\left(t_{i+1}-t_{i}\right)^{-\half}$ is integrable around the singularity
at $t_{i+1}-t_{i}=0$, the integrand in equation (\ref{eq:D2_2})
is integrable over the range of times $\left\{ \vec{t}\in\bR^{k}:\de<t_{1}<\ld<t_{k}<\tf-\de\right\} $
with finite total integral. Since $\lim_{\et\to0}\one\left\{ E(\de,\et)\right\} =0\ a.s$,
we have by the dominated convergence theorem that the RHS of equation
(\ref{eq:D2_2}) tends to $0$ as $\et\to0$. Hence, given any $\ep>0$,
we can find $\et$ so small so that of equation (\ref{eq:D2_2}) is
less than $\ep$, as desired.
\end{proof}

\section{\label{sec:overlap_times}Overlap Times and Exponential Moment Control}
The main object of study in this section are the ``overlap times'' introduced in Definition \ref{def:overlap_times_NIW}.
These can be thought of as discrete version of the local times studied
in Section 4 of \cite{OConnellWarren2015}. The moments of this object
are naturally related to the $L^{2}$ norm of $\ps_{k}^{(N),(\tf,\zf)}$
(see Corollary \ref{cor:ON_moments}). The main result of this section
is Proposition \ref{prop:ON_exp_mom} which gives a particular type
of control, which we call exponential moment control, on the overlap time. This result is the key ingredient in Section
\ref{sec:L2_bounds} to prove Proposition \ref{prop:D3}, Proposition
\ref{prop:D4} and Proposition \ref{prop:uniform_exp_moms}.

\begin{definition}
\label{def:overlap_times_NIW}Recall from Definition \ref{def:NIW}
that $\X(n)$, $n\in\bN$ denotes $d$ non-intersecting random walks
started from $\X(0)=\vec{\de}_{d}(0)$. Let $\iX(n)$, $n\in\bN$
be an independent copy of the same ensemble. For indices $1\leq k,\ell\leq d$
and times $a,b\in\bN$ with $a<b$, define the overlap time on $[a,b]$
between the $k$-th walk of $\vec{X}$ and the $\ell$-th walk
of $\iX$ by:
\begin{equation}
O_{k,\ell}[a,b]\defequal\sum_{n=a}^{b}\one\left\{ X_{k}(n)=X^\prime_{\ell}(n)\right\} .\label{eq:O_kl}
\end{equation}
For times $a,b\in\bN$ with $a<b$, define the total overlap time\textbf{
}on the interval $[a,b]$ of these processes by: 
\[
O[a,b]\defequal\sum_{1\leq k,\ell\leq d}O_{k,\ell}[a,b]=\sum_{n=a}^{b}\abs{\left\{ \vec{X}(n)\cap{\iX}(n)\right\} },
\]
where we think of $\X(n)$ and $\iX(n)$ as sets
and $\abs{\left\{ \X(n)\cap\iX(n)\right\} }$ is the number
of elements in their intersection. 
\end{definition}

\begin{definition}
\label{def:overlap_times_NIWB} Fix any $\xf\in\bZ$ and $\nf\in\bN$
with $\xf+\nf\zmtwo$. Recall from Definition \ref{def:NIWb} that
we have denoted by $\X^{(\nf,\xf)}(n)$, $n\in[0,\nf]\cap\bN$ the
ensemble of $d$ non-intersecting random walk bridges started from
$\X^{(\nf,\xf)}(0)=\vec{\de}_{d}(0)$ and ended at $\X^{(\nf,\xf)}(\nf)=\vec{\de}_{d}(\xf)$.
Let $\X^{\prime(\nf,\xf)}(n)$, $n\in[0,\nf]\cap\bN$ be an independent
copy of the same ensemble. For times $a,b\in\bN$ with $a<b$, define
the total overlap time on the interval $[a,b]\subset[0,\nf]$ of these
processes by:
\[
O^{(\nf,\xf)}[a,b]\defequal\sum_{n=a}^{b}\abs{\left\{ \X^{(\nf,\xf)}(n)\cap\X^{\prime(\nf,\xf)}(n)\right\} }.
\]
For any fixed $\tf>0$ and $\zf\in\bR$, and any $0 < s <s^\prime < \tf$ define the rescaled version
of this by:
\[
O^{(N),(\tf,\zf)}[s,s^{\prime}]\defequal\frac{1}{\sqrt{N}}O^{\left(N\tf,\sqrt{N}\zf\right)_{2}}\left[\floor{Ns},\floor{Ns^{\prime}}\right].
\]

\end{definition}

\subsection{Exponential moment control{\:\textendash\:}definition and properties}
\begin{definition}
\label{def:exp_mom_control} We say that a collection of non-negative
valued processes 
\[
\left\{ Z^{(N)}(t)\ :\ t\in\left[0,\tf\right]\right\} _{N\in\bN},
\]
is ``exponential moment controlled as $t\to0$'' if the following
conditions are all met:

\noindent i) For any fixed $t\in[0,\tf]$, $\ga>0$: 
\[
\sup_{N\in\bN}\e\left[\exp\left(\ga Z^{(N)}(t)\right)\right]<\infty.
\]

\noindent ii) For any fixed $\ga>0$:
\[
\lim_{t\to0}\sup_{N\in\bN}\e\left[\exp\left(\ga Z^{(N)}(t)\right)\right]=1.
\]

\noindent iii) For any fixed $t\in[0,\tf]$ and $\ga>0$:
\[
\lim_{\ell\to\infty}\sup_{N\in\bN}\e\left[\sum_{k=\ell}^{\infty}\frac{1}{k!}\ga^{k}\left(Z^{(N)}(t)\right)^{k}\right]=0.
\]
When there is no risk for ambiguity, we will call this ``exponential
moment controlled'' and omit the ``as $t\to0$''. \end{definition}
\begin{lemma}
\label{lem:exp-mom-iii}If $\left\{ Z^{(N)}(t)\ :\ t\in\left[0,\tf\right]\right\} _{N\in\bN}$ is a collection of non-negative valued process which are exponential moment controlled, then for any exponent $m\in\bN$
we have that:
\[
\lim_{\ell\to\infty}\sup_{N\in\bN}\e\left[\left(\sum_{k=\ell}^{\infty}\frac{1}{k!}\ga^{k}\left(Z^{(N)}(t)\right)^{k}\right)^{m}\right]=0.
\]
\end{lemma}
\begin{proof}
Since each $Z^{(N)}(t)$ is non-negative, there is no harm in rearranging
the order of the terms in the infinite sum. For any $m\in\bN$, $t\in[0,\tf]$
and $\ga>0$ we have:
\begin{eqnarray*}
\left(\sum_{k=\ell}^{\infty}\frac{1}{k!}\ga^{k}\left(Z^{(N)}(t)\right)^{k}\right)^{m} & = & \sum_{k_{1},\ld,k_{m}=\ell}^{\infty}\frac{1}{k_{1}!\ld k_{m}!}\ga^{k_{1}+\cdots+k_{m}}\left(Z^{(N)}(t)\right)^{k_{1}+\cdots+k_{m}}\\
 & \leq & \sum_{k\geq m\ell}\left(\sum_{k_{1}+\ld+k_{m}=k}\binom{k}{k_{1},\ld,k_{m}}\frac{1}{k!}\ga^{k}Z^{(N)}(t)^{k}\right)\\
 & = & \sum_{k\geq m\ell}\left(m\ga\right)^{k}\frac{1}{k!}Z^{(N)}(t)^{k},
\end{eqnarray*}
so the desired result holds by property iii) from Definition \ref{def:exp_mom_control}
of exponential moment control with parameter chosen to be $m\ga$. \end{proof}
\begin{lemma}
\label{lemma:sum-of-exp} Suppose $\left\{ Z^{(N)}(t)\ :\ t\in\left[0,\tf\right]\right\} _{N\in\bN}$
and $\left\{ Y^{(N)}(t)\ :\ t\in\left[0,\tf\right]\right\} _{N\in\bN}$
are both exponential moment controlled as $t\to0$. If $\left\{ W^{(N)}(t)\ :\ t\in\left[0,\tf\right]\right\} _{N\in\bN}$
is a collection of non-negative valued processes so that for all $t\in[0,\tf]$
and all $N\in\bN$ we have
\[
W^{(N)}(t)\leq Z^{(N)}(t)+Y^{(N)}(t),
\]
then \textup{$\left\{ W^{(N)}(t)\ :\ t\in\left[0,\tf\right]\right\} $}
is exponential moment controlled as $t\to0$.\end{lemma}
\begin{proof}
We verify properties i), ii) and iii) from Definition \ref{def:exp_mom_control}.
For any $t\in[0,\tf]$ and $\ga>0$, we have by the Cauchy Shwarz
inequality:
\[
\e\left[\exp\left(\ga W^{(N)}(t)\right)\right]\leq \sqrt{\e\left[\exp\left(2\ga Z^{(N)}(t)\right)\right]\cdot\e\left[\exp\left(2\ga Z^{(N)}(t)\right)\right]}.
\]
From this inequality, properties i) and ii) for $W^{(N)}$ follow
by the hypothesis that $Z^{(N)}(t)$ and $Y^{(N)}(t)$ satisfy properties
i) and ii). To see property iii) for $W^{(N)}(t)$, consider that
for any $t\in[0,\tf]$ and $\ga>0$,
\begin{align}
\e\left[\sum_{k=2\ell}^{\infty}\frac{1}{k!}\ga^{k}\left(W^{(N)}(t)\right)^{k}\right] \leq & \e\left[\sum_{k=2\ell}^{\infty}\frac{1}{k!}\ga^{k}\left(Z^{(N)}(t)+Y^{(N)}(t)\right)^{k}\right] \nonumber \\
 = & \e\left[\sum_{k=2\ell}^{\infty}\sum_{\stackrel{a\geq0,b\geq0}{a+b=k}}\left(\frac{1}{a!}\ga^{a}\left(Z^{(N)}(t)\right)^{a}\right)\left(\frac{1}{b!}\ga^{b}\left(Y^{(N)}(t)\right)^{b}\right)\right]\nonumber \\
 \leq & \e\left[\left(\sum_{a=0}^{\infty}\left(\frac{1}{a!}\ga^{a}\left(Z^{(N)}(t)\right)^{a}\right)\right)\left(\sum_{b=\ell}^{\infty}\left(\frac{1}{b!}\ga^{b}\left(Y^{(N)}(t)\right)^{b}\right)\right)\right.\nonumber \\
 & \ \ \left.+\left(\sum_{a=\ell}^{\infty}\left(\frac{1}{a!}\ga^{a}\left(Z^{(N)}(t)\right)^{a}\right)\right)\left(\sum_{b=0}^{\infty}\left(\frac{1}{b!}\ga^{b}\left(Y^{(N)}(t)\right)^{b}\right)\right)\right]\nonumber \\
  \leq & \sqrt{\e\left[\exp\left(2\ga Z^{(N)}(t)\right)\right]}\sqrt{\e\left[\left(\sum_{b=\ell}^{\infty}\left(\frac{1}{b!}\ga^{b}\left(Y^{(N)}(t)\right)^{b}\right)\right)^{2}\right]}\nonumber \\
  & +\sqrt{\e\left[\exp\left(2\ga Y^{(N)}(t)\right)\right]}\sqrt{\e\left[\left(\sum_{a=\ell}^{\infty}\left(\frac{1}{a!}\ga^{a}\left(Z^{(N)}(t)\right)^{a}\right)\right)^{2}\right]}, \label{eq:moment_for_sum}
\end{align}
where we have applied the Cauchy-Schwarz inequality in the last line.
Since we have that $\e\left[\exp\left(2\ga Z^{(N)}(t)\right)\right]$ and $\e\left[\exp\left(2\ga Y^{(N)}(t)\right)\right]$
are bounded over all $N\in\bN$ by hypothesis i) of the exponential
moment control, the desired limit as $\ell\to\infty$ of equation
(\ref{eq:moment_for_sum}) follows by application of Lemma \ref{lem:exp-mom-iii}.
\end{proof}

\begin{lemma}
\label{lem:inf_radius_convergence} Suppose that $\left\{ a_{j}\right\} _{j=1}^{\infty}$,
$a_{j}\geq0$, are coefficients such that the power series $f(x)=\sum_{j=1}^{\infty}a_{j}x^{j}$
has an infinite radius of convergence. If $\left\{ Z^{(N)}(t)\ :\ t\in\left[0,\tf\right]\right\} _{N\in\bN}$
is a collection of non-negative valued processes and there is a constant
$c>0$ and an exponent $\al>0$ so that for all $k\in\bN$
\[
\sup_{N\in\bN}\frac{1}{k!}\e\left[Z^{(N)}(t)^{k}\right]\leq a_{k}(ct^{\al})^{k},
\]
then $\left\{ Z^{(N)}(t)\ :\ t\in\left[0,\tf\right]\right\} _{N\in\bN}$
is exponential moment controlled as $t\to0$.\end{lemma}
\begin{proof}
To verify property i) of Definition \ref{def:exp_mom_control}, since
all the terms are non-negative, we have by application of the monotone
convergence theorem that for any $\ga>0$, $t\in[0,\tf]$:
\begin{eqnarray*}
\sup_{N\in\bN}\e\left[\exp\left(\ga Z^{(N)}(t)\right)\right] & \leq & 1+\sum_{k=1}^{\infty}\frac{\ga^{k}}{k!}\sup_{N\in\bN}\e\left[Z^{(N)}(t)^{k}\right]\\ 
& \leq & 1+\sum_{k=1}^{\infty}a_{k}\left(ct^{\al}\ga\right)^{k}=1+f(ct^{\al}\ga).
\end{eqnarray*}
Since $f$ has an infinite radius of convergence, this is finite as
desired. Property ii) also follows from this display since we notice
that $f(ct^{\al}\ga)\to0$ as $t\to0$. Finally to see iii), notice
that for fixed $t\in[0,\tf]$ and $\ga>0$ we have in the same way
that for any $\ell\in\bN$, $\sup_{N\in\bN}\e\left[\sum_{k=\ell}^{\infty}\frac{1}{k!}\ga^{k}\left(Z^{(N)}(t)\right)^{k}\right]\leq\sum_{k=\ell}^{\infty}a_{k}\left(ct^{\al}\ga\right)^{k}$.
This tends to zero as $\ell\to\infty$ since $f(ct^{\al}\ga)=\sum_{k=1}^{\infty}a_{k}\left(ct^{\al}\ga\right)^{k}$
is a convergent series. 
\end{proof}

\begin{lemma}
\label{lem:Uniformly-sub-gaussian} If $\left\{ Z^{(N)}(t)\ :\ t\in\left[0,\tf\right]\right\} _{N\in\bN}$
is a collection non-negative valued processes, for which there exist
constants $C$ and c so that:
\[
\sup_{N\in\bN}\p\left(Z^{(N)}(t)>\al\right)\leq C\exp\left(-c\frac{\al^{2}}{t}\right)\ \forall\al>0,
\]
then $\left\{ Z^{(N)}(t)\ :\ t\in\left[0,\tf\right]\right\} _{N\in\bN}$
is exponential moment controlled as $t\to0$.\end{lemma}
\begin{proof}
The $k$-th moments are bounded as follows:
\begin{eqnarray*}
\e\left[\left(Z^{(N)}(t)\right)^{k}\right] & = & k\intop_{0}^{\infty}x^{k-1}\p\left(Z^{(N)}(t)>x\right)\d x\\
 & \leq & k\intop_{0}^{\infty}C\exp\left(-c\frac{x^{2}}{t}\right)x^{k-1}\d x=k\frac{C}{2}\left(\frac{t}{c}\right)^{\half k}\Ga\left(\frac{k}{2}\right),
\end{eqnarray*}
and the result follows by Lemma \ref{lem:inf_radius_convergence}
because the power series $f(x)=\sum_{k=1}^{\infty}k\frac{1}{k!}\Ga\left(\frac{k}{2}\right)x^{k}$
has infinite radius of convergence.
\end{proof}

\subsection{Positions of non-intersecting random walks}

In this subsection we prove exponential moment control for the rescaled
position of the random walks. This is used as an ingredient in Subsection
\ref{sub:exp_mom_overlap} to prove that the total overlap time is
exponential moment controlled.
\begin{lemma}
\label{lem:exp-mom-for-location} Recall from Definition \ref{def:NIW}
that $\X(n)$, \textup{$n\in\bN$} denotes an ensemble of $d$ non-intersecting
random walks started from $\X(0)=\vec{\de}_{d}(0)$. For any fixed
$\tf$, the absolute value of the rescaled top line process 
\[
\left\{ \frac{1}{\sqrt{N}}\big|{X_{d}\left(\floor{tN}\right)\big|},\ t\in[0,\tf]\right\} _{N\in\bN}
\]
is exponential moment controlled as $t\to0$.\end{lemma}
\begin{proof}
By Lemma \ref{lem:Uniformly-sub-gaussian}, it suffices to show that
there are constants $c,C$ so that for all $N$, $\p\left(\abs{X_{d}(\floor{tN})}>\sqrt{N}\al\right)\leq C\exp\left(-c\frac{\al^{2}}{t}\right)$.
We will prove the stronger statement that $\p\left(\sup_{0<n<tN}\abs{X_{d}\left(n\right)}>\sqrt{N}\al\right)\leq C\exp\left(-c\frac{\al^{2}}{t}\right)$
by induction on $d$, using the reflected construction of $d$ non-intersecting
random walks from Section 2 of \cite{Nordenstam-Domino-Shuffling} (this is the only part of the paper where $d$ is un-fixed).
The case $d=1$ is clear since in this case $\frac{1}{\sqrt{N}}X_{1}(n)$
is a rescaled simple symmetric random walk and the estimate above
is standard. Now suppose the result holds for $d-1$. The reflected
construction in \cite{Nordenstam-Domino-Shuffling} is a coupling
of the process $\X$ of $d$ non-intersecting walks started from
$\vec{\de}_{d}(0)$ and the process $\Y$ of $d-1$ non-intersecting
walks started from $\vec{\de}_{d-1}(0)$. In this coupling, the
process $\Y$ is first constructed, and then the top line $X_{d}$
is realized as a simple symmetric random walk which is reflected upward
upon collisions with the top line $Y_{d-1}$, namely:
\[
X_{d}(n+1)-X_{d}(n)=\be(n)+2\cdot\one\big\{ X_{d}(n)+\be(n)=Y_{d-1}(n+1)\big\} ,
\]
where $\be(t)$ are iid $\{-1,+1\}$ fair coinflips, independent of
the process $\Y$. Define the range of a simple random walk up to time $M$ by $R(\be)[0,M]\defequal\sup_{0<s<M}\sum_{i=1}^{s}\be(i)-\inf_{0<s<M}\sum_{i=1}^{s}\be(i)$. From this construction, we notice that for any
$\al$:
\[
\left\{ \sup_{0<n<tN}X_{d}(n)>\al\sqrt{N}\right\} \subset\left\{ \sup_{0<n<tN}Y_{d-1}(n)>\frac{\al}{2}\sqrt{N}\right\} \cup\left\{ R(\be)[0,\floor{tN}]>\frac{\al}{2}\sqrt{N}\right\} .
\]
This is because if $\sup_{0<n<tN}Y_{d-1}(n)\leq\frac{\al}{2}\sqrt{N}$,
then in order for $X_{d}$ to advance from position $\frac{\al}{2}\sqrt{N}$
to $\al\sqrt{N}$, the process $X_{d}$ will need a boost of at least
$\frac{\al}{2}\sqrt{N}$ from the coinflip sequence $\be$.

By the inductive hypothesis, $\p\left(\sup_{0<s<tN}Y_{d-1}(s)>\frac{\al}{2}\sqrt{N}\right)\leq C_{d-1}\exp\left(-c_{d-1}\frac{\al^{2}}{4t}\right)$
for some constants $C_{d-1}$, $c_{d-1}$ (which depend on $d-1)$.
On the other hand the range of the walk, $R(\be)[0,\floor{tN}]$ has been classically
studied see e.g. \cite{SimpleRandomWalkRange}, and is known to have subguassian
tails $\p\left(R(\be)[0,\floor{tN}]>\frac{\al}{2}\sqrt{N}\right)\leq C_{RW}\exp\left(-c_{RW}\frac{\al^{2}}{4t}\right)$.
A union bound then completes the bound on $\p\left(\sup_{0<s<t}X_{d}(s)>\sqrt{N}\al\right)$.
The bound on $\p\left(\sup_{0<s<tT}-X_{d}(t)>\al\sqrt{N}\right)$
is even easier since in this coupling above we have $X_{d}(n)\geq\sum_{i=1}^{n}\be(i)$
, and the result follows by a standard bound for the simple symmetric
random walk.\end{proof}
\begin{corollary}
\label{cor:k-th-line} For any $\tf>0$ and any $1\leq k\leq d$,
the rescaled $k$-th line process: 
\[
\left\{ \frac{1}{\sqrt{N}}\big|{X_{k}\left(\floor{tN}\right)}\big|,\ t\in[0,\tf]\right\} _{N\in\bN},
\]
is exponential moment controlled as $t\to0$.\end{corollary}
\begin{proof}
The case $k=d$ is exactly Lemma \ref{lem:exp-mom-for-location}.
The case $k=1$ (the bottom line) is immediate by the invariance of
the random walk under flipping the process vertically, namely: $X_{1}(t)\dequal2d-X_{d}(t)$.
Finally then notice that for $1<k<d$, because the walks are always
ordered so that $X_{1}(t)<X_{k}(t)<X_{d}(t)$, we have:
\[
\frac{1}{\sqrt{N}}\abs{X_{k}(\floor{tN})}\leq\frac{1}{\sqrt{N}}\abs{X_{d}(\floor{tN})}+\frac{1}{\sqrt{N}}\abs{X_{1}(\floor{tN})},
\]
 and the exponential moment control follows by application of Lemma
\ref{lemma:sum-of-exp} using the cases $k=1$, $k=d$ already proven.
\end{proof}

\subsection{Inverse gaps of non-intersecting random walks}

In this subsection we study some bounds involving the inverse gaps
between walks in the ensemble of $d$ non-intersecting random walks:
these are quantities involving $\abs{X_{b}(n)-X_{a}(n)}^{-1}$ for
$1\leq a,b\leq d$.
\begin{definition}
For fixed $n\in\bN$, $\ep>0$, define $\bS_{n,\ep}\subset\bZ^{d}$
by 
\[
\bS_{n,\ep}\defequal\left\{ x\in\bZ^{d}:\abs{x_{j}-x_{i}}>n^{\half-\ep}\ \forall1\leq i,j\leq d, i\neq j \right\} .
\]
\end{definition}
\begin{lemma}
\label{lem:Egaps_W_n_ep} Recall from Definition \ref{def:NIW} that
$\X(n)\in\bW^d_{2}$, \textup{$n\in\bN$} is an ensemble of $d$ non-intersecting
walks and $\e_{\vec{x}^{0}}\left[\cdot\right]$
denotes the expected value of this ensemble started from $\vec{X}(0)=\vec{x}^{0}\in\bW^d_{2}$.
Recall also from Definition \ref{def:NIBb} that $\vec{D}(t)\in\bW^d$, $t\in(0,\infty)$
denotes $d$ non-intersecting Brownian motions and $\e_{\vec{0}}\left[\cdot\right]$
is the expectation of these walks started from $\vec{D}(0)=(0,0,\ld,0)$. For any $\ep > 0$, there exists a constant $C_{\ep}$
so that for any indices $1\leq a < b \leq d$ and any $n\in\bN$ we have the bound
\begin{eqnarray}
\sup_{n\in\bN}\sup_{\vec{x}^{0}\in\bS_{n,\ep}\cap\bW^d_{2}}\e_{\x^{0}}\left[\frac{1}{\frac{1}{\sqrt{n}}\big(X_{b}(n)-X_{a}(n)\big)}\right] & \leq & 3^{\binom{d}{2}}\e_{\vec{0}}\left[\frac{1}{D_{b}(1)-D_{a}(1)}\right]+C_{\ep}.\label{eq:Egaps_W_n}
\end{eqnarray}
\end{lemma}
\begin{remark}
This argument is based on ideas from \cite{denisov2010conditional}
which goes by coupling random walks with Brownian motions
and using the Doob $h$-transform to evaluate the expectations. Using these ideas, it is possible
to show that for smooth functions $f$ and for fixed $\x^{0}$ that
$\e_{\x^{0}}\left[f\left(\frac{1}{\sqrt{n}}\vec{X}(n)\right)\right]=(1+o(1))\e_{\vec{0}}\left[f\left(\vec{D}(1)\right)\right]$
(see Lemma 18 in \cite{denisov2010conditional}). We need a bound
which holds uniformly over starting positions $\x^{0}$ which
is why our bound is relaxed by the factor $3^{\binom{d}{2}}$ and
constant $C_{\ep}$. Note also that the expectation on the RHS of equation
(\ref{eq:Egaps_W_n}) is finite since the possible singularity when $D_{b}(1)-D_{a}(1)=0$
is canceled away by the Vandermonde determinant in the density from equation (\ref{eq:law_of_D_from_0}).\end{remark}
\begin{proof}
Let $\vec{S}(n)=\left(S_{1}(n),\ld,S_{d}(n)\right)$
be $d$ iid simple symmetric walks started from $\vec{S}(0)=(0,0,\ld0)$, and denote their expectation simply by $\e$. By the Definition \ref{def:NIW} for $\X(n)$ as a Doob $h$-transform using the Vandermonde determinant $h_d$, the expectation on the LHS of equation (\ref{eq:Egaps_W_n}) can be written as
\begin{align}
 \e_{\x^{0}}\left[\frac{1}{\frac{1}{\sqrt{n}}\left(X_{b}(n)-X_{a}(n)\right)} \right]&=\frac{1}{h_{d}(\x^{0})}\e\left[\frac{h_{d}\left(\S(n)+\x^{0}\right)}{\frac{1}{\sqrt{n}}\left(S_{b}(n)+x_{b}^{0}-S_{a}(n)-x_{a}^{0}\right)}\one\left\{ \ta_{\vec{x}^{0}}^{S}>n\right\} \right] \nonumber \\
 &=\frac{\sqrt{n}}{h_{d}(\x^{0})}\e\left[\prod_{\stackrel{i<j}{(i,j)\neq(a,b)}} \left(S_{j}(n)+x_{j}^{0}-S_{i}(n)-x_{i}^{0}\right)\one\left\{ \ta_{\vec{x}^{0}}^{S}>n\right\} \right], \label{eq:E_oo_gap} \\
 \ta_{\x^{0}}^{S} &\defequal  \inf_{m \in \bN}\left\{ S_{i}(m)+x_{i}^{0}=S_{j}(m)+x_{j}^{0}\text{ for }1\leq i< j\leq d \right\}. \nonumber
\end{align}

By the KMT coupling \cite{KMTcoupling}, we couple the symmetric random walks $\vec{S}(i)$
with $d$ iid Brownian motions,
$\B(t)=\left(B_{1}(t),\ld,B_{d}(t)\right)$ started from $\vec{B}(0)=(0,0,\ld,0)$
so that for absolute constants $K_{1},K_{2},K_{3}>0$ we have 
\[
\p\left(\sup_{1\leq j\leq d}\sup_{1\leq m\leq n}\abs{S_{j}(m)-B_{j}(m)}>K_{3}\log n+x\right)\leq K_{1}\exp\left(-K_{2}x\right),
\]
for all $n\in\bN,x\in\bR$. For our purposes, we do not need the full
power of this $O\left(\log n\right)$ coupling, so we will put $x=\half n^{\half-2\ep}$ and enlarge the constants if necessary to get the weaker inequality
\begin{equation}
\p\left(\sup_{1\leq j\leq d}\sup_{1\leq m\leq n}\abs{S_{j}(m)-B_{j}(m)}\geq\half n^{\half-2\ep}\right)\leq K_{1}\exp\left(-\half K_{2}n^{\half-2\ep}\right).\label{eq:KMT}
\end{equation}
Now define the event $A_{\ep,n}=\left\{ \sup_{1\leq j\leq d}\sup_{t\in[0,n]}\abs{S_{j}(\floor t)-B_{j}(t)}<n^{\half-2\ep}\right\} $.
We have by a union bound that: 
\begin{align*}
A_{\ep,n}^{c} \subset \Big\{ {\displaystyle \sup_{\stackrel{1\leq j\leq d}{1\leq m\leq n}}\abs{S_{j}(m)-B_{j}(m)}\geq\frac{n^{\half-2\ep}}{2}}\Big\} {\displaystyle \bigcup_{\stackrel{1\leq j\leq d}{1\leq m\leq n}}}\Big\{ \sup_{0\leq t\leq1}\abs{B_{j}(m+t)-B_{j}(m)}\geq\frac{n^{\half-2\ep}}{2}\Big\}.
\end{align*}
Thus:
\begin{eqnarray}
\p\left(A_{\ep,n}^{c}\right) & \leq & K_{1}\exp\left(-K_{2}\left(\half n^{1-2\ep}\right)\right)+(2nd)\p\left(\sup_{0\leq t\leq1}B(t)\geq\half n^{\half-2\ep}\right)\nonumber \\
 & \leq & K_{1}\exp\left(-K_{2}\left(\half n^{1-2\ep}\right)\right)+\frac{4d}{\sqrt{2\pi}}n^{\half+2\ep}\exp\left(-\frac{1}{8}n^{1-4\ep}\right).\label{eq:A_bound} 
\end{eqnarray}
In the above, we have used equation (\ref{eq:KMT}) along with the
reflection principle and the Mill's ratio estimate $\p\left(\sup_{0\leq t\leq 1}B(t)\geq x\right)=2\p\left(B(1)\geq x\right)\leq\frac{2}{\sqrt{2\pi}}\frac{1}{x}\exp\left(-x^{2}/2\right)$.

We now analyze the expectation in equation (\ref{eq:E_oo_gap}) by separately examining the contribution on $A_{\ep,n}$ and $A_{\ep,n}^c$. On the event $A_{\ep,n}^c$ we use the bound $|S_i(n)|\leq n$ and then expand the Vandermonde determinant to see that:
\begin{eqnarray}
& & \frac{\sqrt{n}}{h_{d}(\x^{0})}\e\left[\prod_{{i<j},{(i,j)\neq(a,b)}}\left(S_{j}(n)+x_{j}^{0}-S_{i}(n)-x_{i}^{0}\right)\one\left\{ \ta_{\vec{x}^{0}}^{S}>n\right\} \one\left\{A_{\ep,n}^c\right\} \right] \label{eq:Ac_contribution}\\
& \leq & \frac{\sqrt{n}}{x_b^0 - x_a^0} \prod_{{i<j},{(i,j)\neq(a,b)}}\left(\frac{2n}{x_j^0 - x_i^0} + 1\right) \e\left[ \one\left\{ \ta_{\vec{x}^{0}}^{S}>n\right\} \one\left\{A_{\ep,n}^c\right\} \right] \nonumber \\
& \leq & \sqrt{n}(n+1)^{\binom{d}{2}-1} \p\left(A_{\ep,n}^c\right), \nonumber
\end{eqnarray}
where the last equality follows since $x_j^0 - x_i^0 \geq 2$ for $1\leq i<j \leq d$. By equation (\ref{eq:A_bound}), $\p\left(A_{\ep,n}^{c}\right)$ is \emph{exponentially} small as $n\to\infty$, and thus the LHS of equation (\ref{eq:Ac_contribution}) converges to $0$ as $n\to\infty$. In particular then, it is bounded for all $n$ by some constant $C_{\ep}$. 

To analyze the contribution to equation (\ref{eq:E_oo_gap}) on $A_{\ep,n}$, we first
define $\x^{+}$ to be a slightly dilated version of the initial position
$\x^{0}$ by setting $x_{i}^{+}\defequal x_{i}^{0}+2i\floor{n^{\half-2\ep}}$,
i.e. by expanding the initial gaps between adjacent walks by $2n^{\half-2\ep}$.
On event $A_{\ep,n}$ we take advantage of the following inequality
which holds for any $j>i$ and at all times $t\in[0,n]$:
\begin{eqnarray*}
S_{j}(\floor t)+x_{j}^{0}-S_{i}(\floor t)-x_{i}^{0} & \leq & B_{j}(t)+x_{j}^{0}-B_{i}(t)-x_{i}^{0}+2n^{\half-2\ep}\\
 & \leq & B_{j}(t)+x_{j}^{0}-B_{i}(t)-x_{i}^{0}+2(j-i)n^{\half-2\ep}\\
 & \leq & B_{j}(t)+x_{j}^{+}-B_{i}(t)-x_{i}^{+}.
\end{eqnarray*}
In other words, on the event $A_{\ep,n}$, we have for all times $t\in[0,n]$,
the gaps between the Brownian motions $\B(t)+\x^{+}$ are all strictly
greater than the the gaps between the walks $\S(t)+\x^{0}$. As
a consequence of this, the first intersection for the Brownian motions happens strictly after the first intersection time for the random walks and therefore $\one\left\{ \ta_{\x^{0}}^{S}>n\right\} \leq\one\left\{ \ta_{\x^{+}}^{B}>n\right\}$, where $\ta_{\x^{+}}^{B}\defequal\inf_{t\in[0,n]}\left\{ B_{i}(t)+x_{i}^{+}=B_{j}(t)+x_{j}^{+}\text{ for }i\neq j\right\}.$ Thus we have the following bound on the contribution on $A_{\ep,n}$ to the expectation in equation (\ref{eq:E_oo_gap}):

\begin{eqnarray}
 & & \frac{\sqrt{n}}{h_{d}(\x^{0})}\e\left[\prod_{i<j,(i,j)\neq(a,b)}\left(S_{j}(n)+x_{j}^{0}-S_{i}(n)-x_{i}^{0}\right)\one\left\{ \ta_{\vec{x}^{0}}^{S}>n\right\} \one\left\{ A_{\ep,n}\right\} \right]\label{eq:A_contribution}\\
 & \leq & \frac{\sqrt{n}}{h_{d}(\x^{0})}\e\left[\prod_{i<j,(i,j)\neq(a,b)}\left(B_{j}(n)+x_{j}^{+}-B_{i}(n)-x_{i}^{+}\right)\one\left\{ \ta_{\vec{x}^{+}}^{B}>n\right\} \one\left\{ A_{\ep,n}\right\} \right]\nonumber \\
 & \leq & \frac{h_{d}(\x^{+})}{h_{d}(\x^{0})}\frac{1}{h_{d}(\x^{+})}\e\left[\frac{h_{d}\left(\B(n)+\x^{+}\right)}{\frac{1}{\sqrt{n}}\left(B_{b}(n)+x_{b}^{+}-B_{a}(n)-x_{a}^{+}\right)}\one\left\{ \ta_{\vec{x}^{+}}^{B}>n\right\} \right]\nonumber \\
 & = & \frac{h_{d}(\x^{+})}{h_{d}(\x^{0})}\e_{\x^{+}}\left[\frac{1}{\frac{1}{\sqrt{n}}\left(D_{b}(n)-D_{a}(n)\right)}\right]. \nonumber
\end{eqnarray}
where we have recognized the Doob $h$-transform definition of the non-intersecting Brownian motions $\vec{D}$ from Definition \ref{def:NIBb}. We now use a coupling result for non-intersecting
Brownian motions that will allow us to compare this to a non-intersecting
Brownian motion started from $\D(0)=\vec{0}$. By Lemma 3.7. of \cite{LiChengBrownianMotionCoupling},
if two initial positions $\x^{(1)}$ and $\x^{(2)}$ have $x_{j}^{(1)}-x_{i}^{(1)}>x_{j}^{(2)}-x_{i}^{(2)}$
for all pairs $1\leq i<j\leq d$, then there exists a coupling of two
ensembles non-intersecting Brownian motions with $\D^{(1)}(0)=\x^{(1)}$ and
$\D^{(2)}(0)=\x^{(2)}$ so that the gaps are always ordered, $D_{j}^{(1)}(t)-D_{i}^{(1)}(t)>D_{j}^{(2)}(t)-D_{i}^{(2)}(t)$
for $1\leq i<j\leq d$, $t\in(0,\infty)$. By this coupling, we can make the comparison
\[
\frac{h_{d}(\x^{+})}{h_{d}(\x^{0})}\e_{\x^{+}}\left[\frac{1}{\frac{1}{\sqrt{n}}\left(D_{b}(n)-D_{a}(n)\right)}\right]\leq\frac{h_{d}(\x^{+})}{h_{d}(\x^{0})}\e_{\vec{0}}\left[\frac{1}{\frac{1}{\sqrt{n}}\left(D_{b}(n)-D_{a}(n)\right)}\right].
\]
Finally, since $\x^{0}\in\bS_{n,\ep}\cap\bW^d_2$ has gaps $x_{j}^{0}-x_{i}^{0}>(j-i)n^{\half-\ep}$
for $1\leq i<j\leq d$, we observe the inequality $0<x_{j}^{+}-x_{i}^{+}=x_{j}^{0}-x_{i}^{0}+2(j-i)\lfloor{n^{\half-2\ep}}\rfloor\leq 3(x_{j}^{0}-x_{i}^{0})$.
Since a Vandermonde determinant is the product of $\binom{d}{2}$
such gaps, we have the inequality:
\begin{eqnarray*}
\frac{h_{d}(\x^{+})}{h_{d}(\x^{0})}\e_{\vec{0}}\left[\frac{1}{\frac{1}{\sqrt{n}}\left(D_{b}(n)-D_{a}(n)\right)}\right] & \leq & 3^{\binom{d}{2}}\e_{\vec{0}}\left[\frac{1}{\frac{1}{\sqrt{n}}\left(D_{b}(n)-D_{a}(n)\right)}\right]\\
 & = & 3^{\binom{d}{2}}\e_{\vec{0}}\left[\frac{1}{D_{b}(1)-D_{a}(1)}\right],
\end{eqnarray*}
where the last equality follows by Brownian scaling. The conclusion of the lemma follows by combining the estimates for the contribution on $A_{\ep,n}^c$ from equation (\ref{eq:Ac_contribution}) and for the contribution on $A_{\ep,n}$ from equation (\ref{eq:A_contribution}). \end{proof}
\begin{lemma}
\label{lem:Expected_Inverse_Gap} Fix any indices $1\le a<b\leq d$
. There is a universal constant $C_{b,a}^{g}$ that bounds the expected
inverse gap size uniformly over all initial conditions $\vec{x}^{0}\in\bW^d_{2}$
and all times $n\in\bN$. Namely:
\[
\sup_{n\in\bN}\sup_{\vec{x}^{0}\in\bW^d_{2}}\e_{\x^{0}}\left[\frac{1}{\frac{1}{\sqrt{n}}\left(X_{b}(n)-X_{a}(n)\right)}\right]\leq C_{b,a}^{g}.
\]
\end{lemma}
\begin{proof}
Fix some $0<\ep<\frac{1}{4}$ ($\ep=\frac{1}{10}$ will do). Let $\vec{S}(n)=\left(S_{1}(n),\ld,S_{d}(n)\right)$
be $d$ iid simple symmetric walks started from $\vec{S}(0)=(0,0,\ld0)$, and denote their expectation simply by $\e$.
Let $\nu_{n,\ep}=\min\left\{ t\geq1:\x^{0}+\S(t)\in\bS_{n,\ep}\right\} $
be the first time these walk shifted by $\vec{x}^{0}$ enters $\bS_{n,\ep}$.
(Note that $\nu_{n,\ep}$ depends on $\vec{x}^{0}$ as well, but we
suppress this for notational convenience.) By Lemma 8 from \cite{denisov2010conditional},
we have some constants $c_{1},c_{2}$ so that the following bounds
holds:
\begin{equation}
\e\left[\big|{h_{d}\big(\x^{\nogt}+\S(n)\big)}\big|\cdot\one\left\{ \nu_{n,\ep}>n^{1-\ep}\right\} \right]\leq c_{1}h_{d}(\x^{\nogt})\exp\left(-c_{2}n^{\ep}\right).\label{eq:denisov_bound}
\end{equation}
(Actually, Lemma 8 in \cite{denisov2010conditional} includes a stronger
statement that

 \begin{equation*} \e\left[\abs{h_{d}\big(\x^{0}+\S(n)\big)}\one\left\{ \nu_{n,\ep}>n^{1-\ep}\right\} \right]\leq c_{1}\prod_{1\leq i<j\leq d}\left(1+\abs{x_{i}^{0}-x_{j}^{0}}\right)\exp\left(-c_{2}n^{\ep}\right),
 \end{equation*}
but since $1\leq\abs{x_{i}^{0}-x_{j}^{0}}$, we can easily reduce
to the above by enlarging the constant $c_{1}$ by a factor $3^{\binom{d}{2}}$.)

Define now the first intersection time $\ta_{\vec{x}^{0}}\defequal\inf\left\{ t:x_{i}^{0}+S_{i}^{0}(t)=x_{j}^{0}+S_{j}^{0}(t), i\ne j\right\} $
and using the Doob $h$-transform
definition of $\vec{X}$ as in Definition \ref{def:NIW}, we find that:
\begin{align}
&\e_{\x^{0}}\left[\frac{1}{\frac{1}{\sqrt{n}}\left(X_{b}(n)-X_{a}(n)\right)}\right]  = \frac{1}{h_{d}(\x^{0})}\e\left[\frac{h_{d}\left(\S(n)+\x^{0}\right)}{\frac{1}{\sqrt{n}}\left(S_{b}(n)+x_{b}^{0}-S_{a}(n)-x_{a}^{0}\right)}\one_{\left\{ \ta_{\x^{0}}>n\right\}} \right]\nonumber \\
 &= \frac{1}{h_{d}(\x^{0})}\e\left[\frac{h_{d}\left(\S(n)+\x^{0}\right)}{\frac{1}{\sqrt{n}}\left(S_{b}(n)+x_{b}^{0}-S_{a}(n)-x_{a}^{0}\right)}\one_{\left\{ \ta_{\x^{0}}>n\right\}} \one_{\left\{ \nu_{n,\ep}>n^{1-\ep}\right\}} \right]\nonumber \\
 & +\frac{1}{h_{d}(\x^{0})}\e\left[\frac{h_{d}\left(\S(n)+\x^{0}\right)}{\frac{1}{\sqrt{n}}\left(S_{b}(n)+x_{b}^{0}-S_{a}(n)-x_{a}^{0}\right)}\one_{\left\{ \ta_{\x^{0}}>n\right\}} \one_{\left\{ \nu_{n,\ep}\leq n^{1-\ep}\right\}} \right]\nonumber \\
 & \leq c_{1}\sqrt{n}\exp\left(-c_{2}n^{\ep}\right) +\frac{1}{h_{d}(\x^{0})}\e\left[\frac{h_{d}\left(\S(n)+\x^{0}\right)}{\frac{1}{\sqrt{n}}\left(S_{b}(n)+x_{b}^{0}-S_{a}(n)-x_{a}^{0}\right)}\one_{\left\{ \ta_{\x^{0}}>n\right\}} \one_{\left\{ \nu_{n,\ep}\leq n^{1-\ep}\right\}} \right], \label{eq:split_1}
\end{align}

where we have applied the bound from equation (\ref{eq:denisov_bound})
and also the simple bound $S_{b}(n)+x_{b}^{0}-S_{a}(n)-x_{a}^{0})\geq 1$
on the event $\left\{ \ta_{\x^{0}}>n\right\} $. In the second term
of the RHS of equation (\ref{eq:split_1}), since we are on the event $\left\{ \nu_{n,\ep}\leq n^{1-\ep}\right\}\cap\left\{ \ta_{\vec{x}^{0}}>n\right\}$, we know that there is no self intersection up to time $\nu_{n,\ep}$, i.e. $\S(\nu_{n,\ep})+\x^{0}\in\cap\bW^d_{2}$ is still in the Weyl chamber. We now use the strong Markov
property to think of $\S(\nu_{n,\ep})+\vec{x}^{0}\in\bS_{n,\ep}\cap\bW^d_{2}$
as the initial position of the walks, which we run for the remaining time $n-\nu_{n,\ep}$. Since $\S(\nu_{n,\ep})+\vec{x}^{0}\in\bS_{n,\ep}\cap\bW^d_{2}$ here,
we are in a position to apply the bound from Lemma \ref{lem:Egaps_W_n_ep}
for this initial position. Integrating over all possible times for
$\nu_{n,\ep}$ and all possible positions for $\vec{S}(\nu_{n,\ep})$
gives:
\begin{align}
  & \e\left[\frac{h_{d}\left(\S(n)+\x^{0}\right)}{\frac{1}{\sqrt{n}}\left(S_{b}(n)+x_{b}^{0}-S_{a}(n)-x_{a}^{0}\right)}\one\left\{ \ta_{\x^{0}}>n\right\} \one\left\{ \nu_{n,\ep}\leq n^{1-\ep}\right\} \right]\label{eq:strong_markov} \\
 = & \sum_{k=1}^{n^{1-\ep}}\sum_{\y\in\bS_{n,\ep}\cap\bW^d_{2}}\p\left(\nu_{n,\ep}=k,\vec{S}(k)+\x^{0}=\y,\ta_{\vec{x}^{0}}>k\right)\nonumber \\
  & \times\e\left[\frac{h_{d}\left(\S(n)+\x^{0}\right)}{\frac{1}{\sqrt{n}}\left(S_{b}(n)+x_{b}^{0}-S_{a}(n)-x_{a}^{0}\right)}\one\left\{ \ta_{\x^{0}}>n\right\} \Bigg|{\nu_{n,\ep}=k,\vec{S}(k)+\x^{0}=\vec{y},\ta_{\vec{x}^{0}}>k}\Bigg.\right]\nonumber \\
  = & \sum_{k=1}^{n^{1-\ep}}\frac{\sqrt{n}}{\sqrt{n-k}}\sum_{\y\in\bS_{n,\ep}\cap\bW^d_{2}}h_{d}(\y)\p\left(\nu_{n,\ep}=k,\vec{S}(k)+\x^{0}=\y,\ta_{\vec{x}^{0}}>k\right)\nonumber \\
  & \times\left(\frac{1}{h_{d}(\y)}\e\left[\frac{h_{d}\left(\S(n-k)+\vec{y}\right)}{\frac{1}{\sqrt{n-k}}\left(S_{b}(n-k)+y_{b}-S_{a}(n-k)-y_{a}\right)}\one\left\{ \ta_{\x^{0}}>n-k\right\} \Bigg|{\vec{S}(0)=\vec{0}}\Bigg.\right]\right)\nonumber \\
  = & \sum_{k=1}^{n^{1-\ep}}\frac{\sqrt{n}}{\sqrt{n-k}}\sum_{\y\in\bS_{n,\ep}\cap\bW^d_{2}}h_{d}(\y)\p\left(\nu_{n,\ep}=k,\vec{S}(k)+\x^{0}=\y,\ta_{\vec{x}^{0}}>k\right)\nonumber \\
   & \times\Bigg(\e_{\y}\Bigg[\frac{1}{\frac{1}{\sqrt{n-k}}\left(X_{b}(n-k)-X_{a}(n-k)\right)}\Bigg]\Bigg)\nonumber \\
 \leq& \frac{\e\left[h_{d}\big(\vec{S}(\nu_{n,\ep})+\x^{0}\big)\one{\left\{ \nu_{n,\ep}\leq n^{1-\ep}\right\}} \one{\left\{ \ta_{\x^{0}}>\nu_{n,\ep}\right\}} \right]}{\sqrt{1-n^{-\ep}}}{ \left(3^{\binom{d}{2}}\e_{\vec{0}}\left[\frac{1}{\scriptstyle D_{b}(1)-D_{a}(1)}\right]+ C_{\ep} \right)}. \nonumber
\end{align}
where we have applied Lemma \ref{lem:Egaps_W_n_ep} and recognized
the remaing sum as an expectation in the last line of equation (\ref{eq:strong_markov}).
We now claim that:
\begin{equation}
\e\left[h_{d}\left(\vec{S}(\nu_{n,\ep})+\x^{0}\right)\one\left\{ \nu_{n,\ep}\leq n^{1-\ep}\right\} \one\left\{ \ta_{\x^{0}}>\nu_{n,\ep}\right\} \right]\leq h_{d}(\x^{0})\left(1+c_{1}\exp\left(-c_{2}n^{\ep}\right)\right)\label{eq:bound_by_hd}
\end{equation}
Indeed, one verifies that (with the notation $x\wedge y = \min(x,y)$)
\[
\one\left\{ \nu_{n,\ep}\leq n^{1-\ep}\right\} \one\left\{ \ta_{\x^{0}}>\nu_{n,\ep}\right\} =1-\one\left\{ \ta_{\x^{0}}\leq \nu_{n,\ep}\wedge n^{1-\ep}\right\} -\one\left\{ \ta_{\x^{0}}>n^{1-\ep}\right\} \one\left\{ \nu_{n,\ep}>n^{1-\ep}\right\},
\]
and then we have:
\begin{align*}
\text{LHS (\ref{eq:bound_by_hd})} = & \e\left[h_{d}\left(\vec{S}(\nu_{n,\ep})+\x^{0}\right)\right]-\e\left[h_{d}\left(\vec{S}(\nu_{n,\ep})+\x^{0}\right)\one{\left\{ \ta_{\x^{0}}\leq \nu_{n,\ep}\wedge n^{1-\ep}\right\}} \right]\\
 & -\e\left[h_{d}\left(\vec{S}(\nu_{n,\ep})+\x^{0}\right)\one{\left\{ \ta_{\x^{0}}>n^{1-\ep}\right\}} \one{\left\{ \nu_{n,\ep}>n^{1-\ep}\right\}} \right]\\
 \leq & h_{d}(\vec{x}^{0})-0+c_{1}h_{d}(\vec{x}^{0})\exp\left(-c_{2}n^{\ep}\right),
\end{align*}
where we have used the bound from equation (\ref{eq:denisov_bound})
and the fact that $h_{d}\left(\vec{S}(\cdot)+\x^{0}\right)$ is a
martingale (see e.g. \cite{OConnell_Roch_Konig_NonCollidingRandomWalks}),
so $\e\left[h_{d}\left(\vec{S}(\nu_{n,\ep})+\x^{0}\right)\right]=h_{d}(\vec{x}^{0})$.
The middle term is zero since $h_{d}\left(\vec{S}(\cdot)+\x^{0}\right)$ is a martingale and
reaches zero at the earlier time $\ta_{\vec{x}^{0}}\leq \nu_{n,\ep}$.
Finally, combining equations (\ref{eq:split_1}), (\ref{eq:strong_markov})
and (\ref{eq:bound_by_hd}) we have 
\begin{equation}
\e_{\vec{x}^{0}}\left[\frac{1}{\frac{1}{\sqrt{n}}\left(X_{b}(n)-X_{a}(n)\right)}\right]  \leq \frac{3^{\binom{d}{2}}\e_{\vec{0}}\left[\frac{1}{D_{b}(1)-D_{a}(1)}\right]+C_{\ep}}{\sqrt{1-n^{-\ep}}}\left(1+(1+\sqrt{n})c_{1}\exp\left(-c_{2}n^{\ep}\right)\right). \nonumber
\end{equation}
This upper bound does not depend on $\x^{0}$ and has a finite limit
as $n\to\infty$, and is hence bounded above by some constant, as
desired.\end{proof}
\begin{lemma}
\label{lem:exp-mom-control-for-gaps} Let $\X(n)$, $n\in\bN$ denote $d$ non-intersecting random walks started from any $\vec{x}^{0}\in\bW_2^d$ as in Definition \ref{def:NIWb}. For any $\tf>0$ and any indices
$1\leq a<b\leq d$, the collection
\[
\left\{ \frac{1}{\sqrt{N}}\sum_{i=1}^{\floor{tN}}\frac{1}{X_{b}(i)-X_{a}(i)}:t\in[0,\tf]\right\} _{N\in\bN},
\]
is exponential moment controlled as $t\to0$.\end{lemma}
\begin{proof}
We assume without loss that the constant from Lemma \ref{lem:Expected_Inverse_Gap} has $C_{b,a}^g > 1$. We will show that the $k$-th moment obeys the bound 
\begin{equation}
\frac{1}{k!}\e_{\vec{x}^0} \left[\left(\frac{1}{\sqrt{N}}\sum_{i=1}^{\floor{tN}}\frac{1}{X_{b}(i)-X_{a}(i)}\right)^{k}\right]\leq\left(C_{b,a}^{g}\sqrt{t}\right)^{k}\frac{\Ga(\half)^{k}}{\Ga\left(\half k+1\right)},\label{eq:result}
\end{equation}
From here the exponential moment control follows from Lemma \ref{lem:inf_radius_convergence}
since the power series with coefficents given by the RHS of equation (\ref{eq:result}) is
easily verified to have infinite radius of convergence. To simplify
notation, we will use the shorthands $G(i)\defequal\big(X_{b}(i)-X_{a}(i)\big)^{-1}$
and $E_{k}(s)\defequal\left\{ \vec{t}\in\bR^{k}:0\leq t_{1}\leq\ld\leq t_{k}\leq s\right\} $
and $E_{k}^{\bN}(s)\defequal E_{k}(s)\cap\bN^{k}$. We will show the
slightly stronger statement that for any $k\in\bN$, $k\geq1$, and
for any non-negative non-increasing function $f:\bR^{\geq0}\to\bR^{\geq0}$
we have:
\begin{equation}
\e_{\x^{\nogt}} \left[\sum_{\vec{i}\in E_{k}^{\bN}(\floor{tN})}\left(\prod_{\ell=1}^{k}G(i_{\ell})\right) f(i_{k})\right]\leq \big(C_{b,a}^{g}\big)^{k}\intop_{\mathclap{\vec{t}\in E_{k}\left(\floor{tN}\right)}}\frac{1}{\sqrt{t_{1}}}\frac{1}{\sqrt{t_{2}-t_{1}}}\ld\frac{1}{\sqrt{t_{k}-t_{k-1}}}f(t_{k})\d \vec{t}.\label{eq:induction}
\end{equation}
Once this is established equation (\ref{eq:result}) follows by setting
$f(x)\equiv1$, using the inequality $\e_{\x^{\nogt}}\left[\left(\frac{1}{\sqrt{N}}\sum_{i=1}^{\floor{tN}}\frac{1}{X_{b}(i)-X_{a}(i)}\right)^{k}\right]\leq\frac{k!}{\sqrt{N}^{k}}\e_{\x^{\nogt}}\left[\sum_{\vec{i}\in E_{k}^{\bN}(\floor{tN})}\left(\prod_{\ell=1}^{k}G(i_{\ell})\right)\right]$
and evaluating the integral that appears on the RHS of equation (\ref{eq:induction}). (The integral appears when computing the normalizing constant for the Dirichlet distribution, see e.g. Section 3.4 in \cite{AKQ_2014})

The proof of equation (\ref{eq:induction}) goes by induction on $k$.
The base case $k=1$ follows by direct application of Proposition
\ref{lem:Expected_Inverse_Gap}, the bound $0<G(i)<1$ and the integral
comparison test for non-increasing functions: 
\begin{equation*}
\e_{\x^{\nogt}}\left[\sum_{i=1}^{\floor{tN}}G(i)f(i)\right] = \sum_{i=1}^{\floor{tN}}\e_{\x^{\nogt}}\left[\frac{1}{X_{b}(i)-X_{a}(i)}\right]f(i) \leq \sum_{i=1}^{\floor{tN}}\frac{C_{b,a}^{g}}{\sqrt{i}}f(i) \leq C_{b,a}^{g}\intop_{0}^{\floor{tN}}\frac{f(s)}{\sqrt{s}} \d s.
\end{equation*}

Now suppose that the statement holds for $k-1$. Let $\cF_{i}=\si\left(\X(1),\ld,\X(i)\right)$
be the filtration generated by the first $i$ steps of the ensemble,
and then consider by the law of total expectation that: 
\begin{eqnarray}
& & \e_{\x^{\nogt}}\left[\sum_{\vec{i}\in E_{k}^{\bN}(\floor{tN})}\left(\prod_{\ell=1}^{k}G(i_{\ell})\right) f(i_{k})\right] \label{eq:induction2} \\
 & \leq  & \e_{\x^{\nogt}}\left[\sum_{\vec{i}\in E_{k-1}^{\bN}(\floor{tN})}\left(\prod_{\ell=1}^{k-1}G(i_{\ell})\right)\left(\sum_{i_{k}=i_{k-1}+2}^{\floor{tN}}\e\left[G(i_{k})\given{\cF_{i_{k-1}}}\right]f(i_{k})+2f(i_{k-1})\nonumber \right)\right]\\
 & \leq & \e_{\x^{\nogt}}\left[\sum_{\vec{i}\in E_{k-1}^{\bN}(\floor{tN})}\left(\prod_{\ell=1}^{k-1}G(i_{\ell})\right)\left(\sum_{i_{k}=i_{k-1}+2}^{\floor{tN}}\frac{C_{b,a}^{g}}{\sqrt{i_{k}-i_{k-1}}}f(i_{k})+2f(i_{k-1})\right)\right]\nonumber,
\end{eqnarray}
where we have used $0\leq G(i)\leq1$ and the fact that $\X(\cdot)$ is a
Markov process, so $\e_{\x^{\nogt}}\left[G(i_{k})\given{\cF_{i_{k-1}}}\right]=\e_{\X(i_{k-1})}\left[G(i_{k}-i_{k-1})\right]\leq{C_{b,a}^{g}}\big(i_{k}-i_{k-1}\big)^{-\half}$
by an application of the inequality from Lemma \ref{lem:Expected_Inverse_Gap}. We now bound the sum over $i_k$ that appears
by the integral comparison test for non-increasing functions and use the fact that $\int_0^1 \frac{1}{\sqrt{x}} \d x = 2$ to see that: 
\[
C_{b,a}^{g}\sum_{i_{k}=i_{k-1}+2}^{\floor{tN}}\frac{f(i_{k})}{\sqrt{i_{k}-i_{k-1}}} + 2f(i_{k-1}) \leq  C_{b,a}^{g} \intop_{i_{k-1}}^{\floor{tN}}\frac{f(s)}{\sqrt{s-i_{k-1}}}\d s.
\]
The result then follows by applying the inductive hypothesis to the RHS of equation (\ref{eq:induction2}) applied with the
non-increasing function $\tilde{f}(x_{k-1})=C_{b,a}^{g}\intop_{x_{k-1}}^{\floor{tN}}\frac{f(s)}{\sqrt{s-x_{k-1}}}\d s$
($\tilde{f}$ can be verified to be non-increasing by doing a change
of variable $u=s-x_{k-1}$).
\end{proof}

\subsection{Conditional drift of non-intersecting random walks\label{sub:Conditional-Drift}}

In this subsection, we study the conditional drift of the $k$-th
walk at time $n$, namely $\e\left[X_{k}(n+1)-X_{k}(n)\given{\X(n)=\vec{x}}\right]$.
\begin{lemma}
\label{lem:prod-lemma} Fix any $m\in\bN$. Then for any collection
of real numbers $\al_{1},\ld,\al_{m}$ with $\abs{\al_{i}}\leq1$
$\forall i$, we have:
\[
\abs{\prod_{i=1}^{m}\left(1+\al_{i}\right)-\left(1+\sum_{i=1}^{m}\al_{i}\right)}\leq\left(2^{m}-1\right)\sum_{i=1}^{m}\abs{\al_{i}}.
\]
\end{lemma}
\begin{proof}
The proof is a straightforward proof by induction on $m$.
\end{proof}
\begin{lemma}
\label{lem:increments}For any $1\leq k\leq d$, denote by $\De X_{k}(n)\defequal X_{k}(n+1)-X_{k}(n)$
the $n$-th increment of the $k$-th non-intersecting random walk in the
ensemble. We have the following bound on the conditional expectation:
\[
\abs{\e\left[\De X_{k}(n)\Big|{\X(n)=\vec{x}}\Big.\right]}\leq2^{d}\sum_{\stackrel{i=1}{i\neq k}}^{d}\frac{1}{\abs{x_{k}-x_{i}}}.
\]
\end{lemma}
\begin{remark}
\label{rem:X_is_more_complicated_than_D}A more careful version of the proof of Lemma \ref{lem:increments}
shows actually that
\[
\e\left[\De X_{k}(n)\given{\X(n)=\vec{x}}\right]=\sum_{\stackrel{i=1}{i\neq k}}^{d}\frac{1}{x_{k}-x_{i}}+\ep(\vec{x}),
\]
where $\ep(\vec{x})$ involves only terms that have the product
of at least two different gaps in the denominator:
\[
\ep(\vec{x})\sim\sum\frac{1}{\left(x_{i_{1}}-x_{j_{1}}\right)\left(x_{i_{2}}-x_{j_{2}}\right)\cdots}.
\]
In the limit $N\to\infty$, these terms $\ep(\vec{x})$
are expected to be typically negligible compared to the first term $\sum_{i=1,i\neq k}^{d}({x_{k}-x_{i}})^{-1}$.
This would recover the generator for the non-intersecting Brownian
motion $\vec{D}$, whose $k$-th walker has drift exactly equal to $\sum_{i=1,i\neq k}^{d}({z_{k}-z_{i}})^{-1}$. 

Part of the additional difficulty in studying the
non-intersecting random walks is these additional terms $\ep(\vec{x})$
give much more complicated interactions between the walks. In $\vec{D}$, the $i$-th
and $j$-th Brownian motion interact in a symmetric way by both pushing on each other with the same ``force'' equal to $({D_{i}-D_{j}})^{-1}$; in $\vec{X}$
the exact interaction between the $i$-th and $j$-th walk will depend on the positions of \emph{all} the other elements of $\vec{X}$ too. This symmetry for $\vec{D}$
was used in \cite{OConnellWarren2015} as part of the analysis to
control singularities of the $k$-point correlation function $\ps_{k}^{(\tf,\zf)}$,
so this complication can be thought of as one of the reasons why our
analysis of $\ps_{k}^{(N),(\tf,\zf)}$ is more involved.
\end{remark}

\begin{proof}
(Of Lemma \ref{lem:increments} ) The proof goes by expanding the
Vandermonde determinant $h_{d}$ that appears in the generator for
$\X$ in terms of a smaller Vandermonde determinant $h_{d-1}$ by factoring out the $k$-th term. We will use the following notation with
the absentee hat ``$\;\hat{\cdot}\;$'', to denote the vector $\vec{x}$
with the $k$-th component removed
\[
\vec{x}^{\hat{k}}\defequal\left(x_{1},\ld,\hat{x}_{k},\ld x_{d}\right)\in\bZ^{d-1}.
\]
Notice with this notation we can decompose the Vandermonde determinant
as
\[
h_{d}(\x)=\prod_{i=1}^{k-1}\left(x_{k}-x_{i}\right)\cdot\prod_{i=k+1}^{d}\left(x_{i}-x_{k}\right)\cdot h_{d-1}\left(\x^{\hat{k}}\right).
\]
We will also write $\vec{\de}^{\hat{k}}\in\left\{ -1,+1\right\} ^{d-1}$
to mean the set of $\vec{\de}^{\hat{k}}=\left(\de_{1},\ld,\hat{\de}_{k},\ld\de_{d}\right)$
that have $\de_{i}\in\left\{ -1,+1\right\} $ for all $1\leq i\leq k-1$
and $k+1\leq i\leq d$ (i.e the labelling is shifted to not include
the index $k$). 

From the definition of the non-intersecting random walks in Definition
\ref{def:NIW} as a Doob $h$-transform of a simple symmetric random
walk, we have:

\begin{align*}
 &\p\left[\De X_{k}(n)=+1\given{\X(n)=\vec{x}}\right] \\
 = & \sum_{\vec{\de}\in\left\{ -1,+1\right\} ^{d},\de_{k}=+1}\frac{1}{2^{d}}\frac{h_{d}(\x+\vec{\de})}{h_{d}(\x)}\\
 = & \sum_{\vec{\de}^{\hat{k}}\in\left\{ -1,+1\right\} ^{d-1}}\frac{1}{2}\prod_{i=1}^{k-1}\left(1+\frac{1-\de_{i}}{x_{k}-x_{i}}\right)\cdot\prod_{i=k+1}^{d}\left(1+\frac{\de_{i}-1}{x_{i}-x_{k}}\right)\cdot\frac{1}{2^{d-1}}\frac{h_{d-1}\left(\x^{\hat{k}}+\vec{\de}^{\hat{k}}\right)}{h_{d-1}\left(\x^{\hat{k}}\right)}\\
 = & \sum_{\vec{\de}^{\hat{k}}\in\left\{ -1,+1\right\} ^{d-1}}\frac{1}{2}\prod_{\stackrel{i\neq k}{i:\de_{i}=-1}}\left(1+\frac{2}{x_{k}-x_{i}}\right)\frac{1}{2^{d-1}}\frac{h_{d-1}\left(\x^{\hat{k}}+\vec{\de}^{\hat{k}}\right)}{h_{d-1}\left(\x^{\hat{k}}\right)},
\end{align*}
and similarly, 
\begin{equation*}
\p\left[\De X_{k}(n)=-1\given{\X(n)=\vec{x}}\right] = \sum_{\vec{\de}^{\hat{k}}\in\left\{ -1,+1\right\} ^{d-1}}\frac{1}{2}\prod_{\stackrel{i\neq k}{i:\de_{i}=+1}}\left(1-\frac{2}{x_{k}-x_{i}}\right)\frac{1}{2^{d-1}}\frac{h_{d-1}(\x^{\hat{k}}+\de^{\hat{k}})}{h_{d-1}(\x^{\hat{k}})}.
\end{equation*}
Subtracting these from each other, we have that $\e\left[\De X_{k}(n)\given{X(n)=\x}\right]$ is given by:
\begin{align}
  & \p\left[\De X_{k}(n)=+1\given{X(n)=\x}\right]-\p\left[\De X_{k}(n)=-1\given{X(n)=\x}\right]\label{eq:E_delta_X}\\
 = & \half\sum_{\vec{\de}^{\hat{k}}\in\left\{ -1,+1\right\} ^{d-1}}\left(\prod_{\stackrel{i\neq k}{i:\de_{i}=-1}}\left(1+\frac{2}{x_{k}-x_{i}}\right)-\prod_{\stackrel{i\neq k}{i:\de_{i}=+1}}\left(1-\frac{2}{x_{k}-x_{i}}\right)\right)\frac{1}{2^{d-1}}\frac{h_{d-1}(\x^{\hat{k}}+\de^{\hat{k}})}{h_{d-1}(\x^{\hat{k}})}\nonumber 
\end{align}
To approximate these products by sums, we now define error terms $\ep_{\de,k}^{-}$, $\ep_{\de,k}^{+}$ by:\begin{eqnarray}
 \ep_{\de,k}^{-}(\x) & \defequal & \prod_{\stackrel{i\neq k}{i:\de_{i}=-1}}\left(1+\frac{2}{x_{k}-x_{i}}\right)- \Bigg(1+\sum_{\stackrel{i\neq k}{i:\de_{i}=-1}}\frac{2}{x_{k}-x_{i}}\Bigg), \label{eq:error_1}\\
\ep_{\de,k}^{+}(\x) & \defequal & \prod_{\stackrel{i\neq k}{i:\de_{i}=+1}}\left(1-\frac{2}{x_{k}-x_{i}}\right)-\Bigg(1-\sum_{\stackrel{i\neq k}{i:\de_{i}=+1}}\frac{2}{x_{k}-x_{i}}\Bigg), \nonumber
\end{eqnarray}
 where, since $\abs{\frac{2}{x_{k}-x_{i}}}\leq1$ for all $i\neq k$,
the errors are bounded by application of Lemma \ref{lem:prod-lemma},
\begin{equation}
\abs{\ep_{\de,k}^{-}(\x)} \leq \left(2^{d}-1\right)\sum_{\stackrel{i\neq k}{i:\de_{i}=-1}}\frac{2}{\abs{x_{k}-x_{i}}}, \;\;\; \abs{\ep_{\de,k}^{+}(\x)}\leq\left(2^{d}-1\right)\sum_{\stackrel{i\neq k}{i:\de_{i}=+1}}\frac{2}{\abs{x_{k}-x_{i}}}.\label{eq:error_boudns_pm}
\end{equation}
Plugging equation (\ref{eq:error_1}) into equation (\ref{eq:E_delta_X}) now
yields
\begin{align*}
 &   \e\left[\De X_{k}(n)\given{X(n)=\x}\right]\\
  = & \half\sum_{\vec{\de}^{\hat{k}}\in\left\{ -1,+1\right\} ^{d-1}}\Big(\sum_{\stackrel{i\neq k}{i:\de_{i}=-1}}\frac{2}{x_{k}-x_{i}}+\ep_{\de,k}^{-}(\x)+\sum_{\stackrel{i\neq k}{i:\de_{i}=+1}}\frac{2}{x_{k}-x_{i}}-\ep_{\de,k}^{+}(\x)\Big)\frac{1}{2^{d-1}}\frac{h_{d-1}(\x^{\hat{k}}+\de^{\hat{k}})}{h_{d-1}(\x^{\hat{k}})} \\
  = & \sum_{\vec{\de}^{\hat{k}}\in\left\{ -1,+1\right\} ^{d-1}}\Big(\sum_{i\neq k}\frac{1}{x_{k}-x_{i}}+\frac{\ep_{\de,k}^{-}(\x)-\ep_{\de,k}^{+}(\x)}{2}\Big)\frac{1}{2^{d-1}}\frac{h_{d-1}(\x^{\hat{k}}+\de^{\hat{k}})}{h_{d-1}(\x^{\hat{k}})} \\
  = & \sum_{i\neq k}^{d}\frac{1}{x_{k}-x_{i}}+\sum_{\de^{\hat{k}}\in\left\{ -1,+1\right\} ^{d-1}}\frac{\ep_{\de,k}^{-}(\x)-\ep_{\de,k}^{+}(\x)}{2}\frac{1}{2^{d-1}}\frac{h_{d-1}(\x^{\hat{k}}+\de^{\hat{k}})}{h_{d-1}(\x^{\hat{k}})},
\end{align*}
where we have pulled out the term that does not depend on $\vec{\de}^{\hat{k}}$ and used the fact that $h_{d-1}$ is harmonic for
the simple symmetric random walk in dimension $d-1$ (see e.g. \cite{OConnell_Roch_Konig_NonCollidingRandomWalks}) to evaluate the sum over $\vec{\de}^{\hat{k}}$. In the error term, we use the triangle inequality to bound
$\abs{\ep_{\de,k}^{-}(\x)-\ep_{\de,k}^{+}(\x)}\leq\abs{\ep_{\de,k}^{-}(\x)}+\abs{\ep_{\de,k}^{+}(\x)}\leq\left(2^{d}-1\right)\sum_{i\neq k}\frac{2}{\abs{x_{k}-x_{i}}}$
by equation (\ref{eq:error_boudns_pm}) and use the fact that $h_{d-1}$
is harmonic again to finally arrive at
\begin{equation}
\abs{\e\left[\De X_{k}(n)\given{\vec{X}(n)=\vec{x}}\right]}\leq \abs{\sum_{i=1,i\neq k}^{d}\frac{1}{x_{k}-x_{i}}}+\left(2^{d}-1\right)\sum_{i=1,i\neq k}^{d}\frac{1}{\abs{x_{k}-x_{i}}}\leq \sum_{i=1,i\neq k}^{d}\frac{2^{d}}{\abs{x_{k}-x_{i}}}, \nonumber
\end{equation}
as desired.\end{proof}
\begin{corollary}
\label{cor:exp-mom-incs} Fix $\tf>0$. For any $1\leq k\leq d$,
the collection
\[
\left\{ \frac{1}{\sqrt{N}}\sum_{n=1}^{\floor{tN}}\abs{\e\left[X_{k}(n+1)-X_{k}(n)\given{\X(n)}\right]},t\in[0,\tf]\right\} _{N\in\bN},
\]
is exponential moment controlled as $t\to0$.\end{corollary}
\begin{proof}
By Lemma \ref{lem:increments},
we have for any $N\in\bN$ and $t\in[0,\tf]$:
\begin{eqnarray*}
\frac{1}{\sqrt{N}}\sum_{n=1}^{\floor{tN}}\abs{\e\left[X_{k}(n+1)-X_{k}(n)\Big|{\X(i)}\big.\right]} & \leq & 2^{d}\sum_{i=1,i\neq k}^{d}\left(\frac{1}{\sqrt{N}}\sum_{n=1}^{\floor{tN}}\frac{1}{\abs{X_{k}(n)-X_{i}(n)}}\right).
\end{eqnarray*}
Each term $\frac{1}{\sqrt{N}}\sum_{n=1}^{\floor{tN}}\big|{X_{k}(n)-X_{i}(n)}\big|^{-1}$
is exponential moment controlled by application of Lemma \ref{lem:exp-mom-control-for-gaps}.
Thus, this whole quantity is a sum of $d-1$ random variables each
of which are exponential moment controlled. Since finite sums of exponentially
moment controlled variables are still exponential moment controlled
by Lemma \ref{lemma:sum-of-exp}, the result follows.
\end{proof}

\subsection{\textmd{\normalsize \label{sub:exp_mom_overlap}}Overlap times of
non-intersecting random walks} In this subsection we establish the exponential moment control for overlap times 
by using a discrete version of Tanaka's formula to write the overlap time
as a finite sum of quantities we can control.
\begin{lemma}[Discrete Tanaka formula]
\label{lem:DiscreteTanaka} Let $\al(i),\be(i)$
be any sequences with $\al(i),\be(i)\in\left\{ -1,+1\right\} $. For
any $A(0),B(0)\in\bZ$ with $B(0)+A(0)\zmtwo$, let $A(n)\defequal\sum_{i=1}^{n}\al(i)+A(0)$
and $B(n)\defequal\sum_{i=1}^{n}\be(i)+B(0)$. Then:
\begin{align*}
\sum_{i=0}^{n}\one\left\{ A(i)=B(i)\right\}  = & \abs{A(n+1)-B(n+1)}-\abs{A(0)-B(0)}\\
  & -\sum_{i=0}^{n}\mathrm{sgn}\big(A(i)-B(i)\big)\al(i+1)+\sum_{i=0}^{n}\mathrm{sgn}\big(A(i+1)-B(i)\big)\be(i+1),
\end{align*}
where we use the convention on the sign function that $\mathrm{sgn}(0)=0$.\end{lemma}
\begin{proof}
Consider:
\begin{align*}
& \abs{A(i)-B(i)}-\abs{A(i-1)-B(i-1)} \\
 = & \abs{A(i)-B(i)}-\abs{A(i)-B(i-1)} +\abs{A(i)-B(i-1)}-\abs{A(i-1)-B(i-1)}\\
 = & \big(\abs{B(i)-A(i)}-\abs{B(i-1)-A(i)}\big) +\big(\abs{A(i)-B(i-1)}-\abs{A(i-1)-B(i-1)}\big)\\
  = & \sgn\big(B(i-1)-A(i)\big)\be(i)+\one{\left\{ B(i-1)=A(i)\right\}} \\
   & +\sgn\big(A(i-1)-B(i-1)\big)\al(i)+\one{\left\{ A(i-1)=B(i-1)\right\}} 
\end{align*}
Summing from $i=1$ to $n+1$ and rearranging then gives the result,
taking into account that $\one\left\{ B(i-1)=A(i)\right\} $ is always
zero since $A(i)$ and $B(i-1)$ always have opposite parity.\end{proof}
\begin{lemma}
\label{lem:O_kl_exp_mom}For any indices $1\leq k,\ell\leq d$, recall the definition of the overlap time
$O_{k,\ell}[a,b]$ from Definition \ref{def:overlap_times_NIW}. For
any fixed $\tf>0$, the collection of processes
\[
\left\{ \frac{1}{\sqrt{N}}O_{k,\ell}[0,\floor{tN}]:t\in[0,\tf]\right\} _{N\in\bN}
\]
is exponential moment controlled as $t\to0$.\end{lemma}
\begin{proof}
As in Definition \ref{def:overlap_times_NIW}, let $\left\{ \X(n):n\in\bN\right\} $
and $\left\{ \iX(n):n\in\bN\right\} $ denote two independent
copies of the non-intersecting walks started from $\vec{\de}(0).$
For notational convenience, we use the shorthand $\De X_{k}(i)\defequal X_{k}(i+1)-X_{k}(i)$
and $\De{X^{\prime}}_{\ell}(i)\defequal X^{\prime}_{\ell}(i+1)-X^\prime_{\ell}(i)$.
By the discrete version of Tanaka's formula, Lemma \ref{lem:DiscreteTanaka},
applied to the definition of $O_{k,\ell}[0,\floor{tN}]$ in equation
(\ref{eq:O_kl}) we have:
\begin{eqnarray}
O_{k,\ell}[0,\floor{tN}] & = & \big|{X_{k}(\floor{tN})-{X^\prime}_{\ell}(\floor{tN})}\big|-\big|{2k-2\ell}\big|-S(\floor{tN})+{S^\prime}\left(\floor{tN}\right) \label{eq:O_kl_discrete_Tanaka} \\
 & \leq & \big|{X_{k}(\floor{tN})}\big|+\big|{{X^\prime}_{\ell}(\floor{tN})}\big|+\big|S(\floor{tN})\big|+\big|{S^\prime}\left(\floor{tN}\right)\big|,\nonumber
\end{eqnarray}
where we define
\begin{equation}
S\left(n\right) \defequal \sum_{i=0}^{n}\sgn\left(X_{k}(i)-{X^\prime}_{\ell}(i)\right)\De X_{k}(i),\ S^\prime \left(n\right)\defequal\sum_{i=0}^{n}\sgn\left(X_{k}(i+1)-{X^\prime}_{\ell}(i)\right)\De{X^\prime}_{k}(i) \nonumber
\end{equation}
By Lemma \ref{lemma:sum-of-exp}, to see the exponential moment control
for ${N^{-\half}}O_{k,\ell}[0,\floor{tN}], $ it suffices to
show that the four terms that appear on the RHS of equation (\ref{eq:O_kl_discrete_Tanaka})
are each exponential moment controlled when scaled by $N^{-\half}$. The first two terms on the
RHS of equation (\ref{eq:O_kl_discrete_Tanaka}) are exponential moment
controlled by Corollary \ref{cor:k-th-line}. We show the remaining two terms are exponential moment controlled below.

To see that $\left\{ \frac{1}{\sqrt{N}}\abs{S\left(\floor{tN}\right)}:t\in[0,\tf]\right\} _{N\in\bN}$ is exponential moment controlled as $t\to0$, notice by triangle inequality we have that
\begin{eqnarray}
\frac{1}{\sqrt{N}}\abs{S(\floor{tN})} & \leq & \frac{1}{\sqrt{N}}\abs{M(\floor{tN})}+\frac{1}{\sqrt{N}}\sum_{i=1}^{\floor{tN}}\abs{\e\big[\De X_{k}(i)\given{\X(i)}\big]},\label{eq:S}
\end{eqnarray}
where we define
\[
M(n)\defequal\sum_{i=0}^{n}\sgn\left(X_{k}(i)-{X^\prime}_{\ell}(i)\right)\left(\De X_{k}(i)-\e\big[\De X_{k}(i)\given{\X(i)}\big]\right).
\]
By Lemma \ref{lemma:sum-of-exp}, it suffices to check that both terms
that appear on the RHS of equation (\ref{eq:S}) are exponential moment
controlled. The second term in equation (\ref{eq:S}) is exponential
moment controlled by Corollary \ref{cor:exp-mom-incs}. To see the
exponential moment control for the first term, we observe that $\left\{ M(n)\right\} _{n\in\bN}$
is a martingale with respect to the the filtration $\cF_{n}\defequal\si\left(\X(1),\iX(1),\ld,\X(n+1),\iX(n+1)\right)$.
Indeed, its increments are
\begin{equation}
M(n)-M(n-1)=\sgn\left(X_{k}(n)-{X^\prime}_{\ell}(n)\right)\left(\De X_{k}(n)-\e\left[\De X_{k}(n)\given{\vec{X}(n)}\right]\right),\label{eq:M_inc}
\end{equation}
which have $\e\left[M(n)-M(n-1)\given{\cF_{n-1}}\right]=0$ since
$\sgn\left(X_{k}(n)-{X^\prime}_{\ell}(n)\right)$ is $\cF_{n-1}$ measurable
and since $\vec{X}(\cdot)$ is a Markov process. Moreover, since $\De X_{k}(n)\in\{-1,+1\}$,
we also notice from equation (\ref{eq:M_inc}) that $\abs{M(n)-M(n-1)}\leq2$.
We are thus in a position to apply Azuma's inequality for martingales
with bounded differences (see e.g. Lemma 4.1 of \cite{McDiarmid}).
This gives for any $N\in\bN$ that 
\[
\p\left(\frac{1}{\sqrt{N}}\abs{M(\floor{tN})}>\al\right)\leq2\exp\left(-\frac{\al^{2}}{8t}\right).
\]
By Lemma \ref{lem:Uniformly-sub-gaussian}, this bound shows that $\left\{ \frac{1}{\sqrt{N}}\abs{M\left(\floor{tN}\right)}:t\in[0,\tf]\right\} _{N\in\bN}$
is exponential moment controlled as desired. 

The proof that  $\left\{ \frac{1}{\sqrt{N}}\abs{{S^\prime}\left(\floor{tN}\right)}:t\in[0,\tf]\right\} _{N\in\bN}$ is exponential moment controlled is similar to the above argument, this time using
\[
{M^\prime}(n)\defequal\sum_{i=0}^{n}\sgn\left(X_{k}(i+1)-{X^\prime}_{\ell}(i)\right)\left(\De{X^\prime}_{k}(i)-\e\big[\De{X^\prime}_{k}(i)\given{\iX(i)}\big]\right),
\]
which is a  martingale on ${\cF^\prime}_{n}\defequal\si\left(\X(1),\iX(1),\ld,\X(n+1),\iX(n+1),\vec{X}(n+2)\right)$.\end{proof}
\begin{corollary}
\label{cor:exp-moment-control-NIW} For any fixed $\tf$, the collection of total overlap time 
\[
\left\{ \frac{1}{\sqrt{N}}O[0,\floor{tN}]:t\in[0,\tf]\right\} _{N\in\bN},
\]
is exponential moment controlled as $t\to0$.\end{corollary}
\begin{proof}
This follows from Lemma \ref{lem:O_kl_exp_mom} since $O[0,\floor{tN}]=\sum_{1\leq k,\ell\leq d}O_{k,\ell}[0,\floor{tN}]$
and because exponential moment control is maintained under finite
sums by Lemma \ref{lemma:sum-of-exp}.
\end{proof}

\subsection{Overlap times of non-intersecting random walk bridges}
In this subsection we will use the exponential moment control for
overlap times of non-intersecting random walks proved in Corollary
\ref{cor:exp-moment-control-NIW} to obtain the exponential moment
control for non-intersecting random walk bridges in Proposition
\ref{prop:ON_exp_mom}.
\begin{lemma}
\label{lem:RN-deriv-bounded}Fix any $\tf>0$ and $\zf\in\bR$. Recall
from Definitions \ref{def:NIW}, \ref{def:NIWb}, and \ref{def:scaled-NIWb}
the definition of the non-intersecting random walks, the non-intersecting
random walk bridges and its rescaled version. There is a constant $C_{R}^{(\tf,\zf)}<\infty$
so that the Radon-Nikodym derivative of the rescaled non-intersecting random
walk bridges $\X^{(N),(\tf,\zf)}(t)$ with respect to the
rescaled non-intersecting walks $\frac{1}{\sqrt{N}}\X(\floor{tN})$
is uniformly bounded by $C_{R}^{(\tf,\zf)}$ over all possible positions and
at all times $t$ with $t<\frac{2}{3}\tf$:
\[
\sup_{N\in\bN}\sup_{t<\frac{2}{3}\tf}\sup_{\vec{z}\in\left(\frac{\bZ}{\sqrt{N}}\right)^{d}}\frac{\p\Big(\X^{(N),(\tf,\zf)}(t)=\vec{z}\Big)}{\p\Big(\frac{1}{\sqrt{N}}\X(\floor{tN})=\vec{z}\Big)}\leq C_{R}^{(\tf,\zf)}.
\]
\end{lemma}
\begin{proof}
Define the variables (which depend on $N$), $\left(\nf,\xf\right)\defequal\big(N\tf,\sqrt{N}\zf\big)_{2}$.
From Definition \ref{def:NIWb}, $\X^{(\nf,\xf)}$ is absolutely continuous
with respect to the non-intersecting walks $\X$ with Radon-Nikodym
derivative explicitly given in equation (\ref{eq:radon_nik_deriv}) in terms of the non-intersection probability $q_{n}(\x,\y)$ given in Definition \ref{def:NIW}. 
There is an exact formula for for $q_{n}\big(\vde(0),\x\big)$ from
Theorem 1 in \cite{ExactFormulaForP}: 
\begin{equation}
q_{n}(\vde(0),\x) = 2^{-nd}2^{-\binom{d}{2}}\prod_{i=1}^{d}\binom{n+d-1}{\frac{n+x_{i}}{2}}\frac{1}{(n+d-i+1)_{i-1}}h_{d}(\x),
\end{equation}
where $(x)_{n}=x(x+1)\cdots(x+n-1)$. By shifting and time reversal,
this gives an explicit formula for $q_{n}\big(\vec{x},\vec{\de}_{d}(\xf)\big)$,
namely $q_{n}\big(\x,\vde(\xf)\big)=q_{n}\big(\x-\xf\vec{1},\vde(0)\big)=q_{n}\big(\vde(0),\x-\xf\vec{1}\big)$
where $\vec{1}$ denotes $\vec{1}\defequal(1,1,\ld1)$. Using this
explicit formula twice in equation (\ref{eq:radon_nik_deriv}) gives: 
\begin{equation}
\text{LHS (\ref{eq:radon_nik_deriv})}=\prod_{i=1}^{d}2^{n}\binom{\nf-n+d-1}{\frac{\nf-n}{2}+\frac{x_{i}-\xf}{2}}\binom{\nf+d-1}{\frac{\nf+\xf}{2}+i-1}^{-1}\frac{\left(\nf+d-i+1\right)_{i-1}}{\left(\nf-n+d-i+1\right)_{i-1}}.\label{eq:rn_explicit}
\end{equation}
We now use the local limit theorem \[\lim\limits_{M\to\infty}\sup\limits_{\ell\in\bZ}\abs{\sqrt{M}2^{-M}\binom{M}{\ell}-\sqrt{\frac{2}{\pi}}\exp\left(-\frac{\left(2\ell-M\right)^{2}}{2M}\right)}=0.\]
Using this, since $\nf-n>\frac{1}{3}\tf N\to\infty$ for $n = Nt$ with $t < \frac{2}{3} \tf$ as $N\to\infty$,
we have:
\begin{align*}
\limsup_{N\to\infty}\sqrt{N}2^{-\left(\nf-n\right)}\binom{\nf-n+d-1}{\frac{\nf-n}{2}+\frac{x_{i}-\xf}{2}} & \leq & \limsup_{N\to\infty}\sqrt{N}2^{-\left(\nf-n\right)}\binom{\nf-n+d-1}{\frac{\nf-n}{2}}\\
 & = & \frac{1}{\sqrt{\tf-t}}\sqrt{\frac{2}{\pi}} 2^{d-1}.
\end{align*}
Similarly, since $\nf+d-1>\tf N\to\infty$ as $N\to\infty$, we have:
\[
\limsup_{N\to\infty}\frac{1}{\sqrt{N}}2^{\nf}\binom{\nf+d-1}{\half\nf+\half\xf+i-1}^{-1}\leq\sqrt{\frac{\pi}{2}}\sqrt{\tf}\exp\left(\frac{\zf^{2}}{2\tf}\right)2^{-(d-1)}.
\]
Combining these and also using $\frac{\left(\nf+d-i+1\right)_{i-1}}{\left(\nf-n+d-i+1\right)_{i-1}}\leq\left(\frac{\tf}{\tf-t}\right)^{i-1}$
gives:
\begin{eqnarray*}
\limsup_{N\to\infty}\left(\text{LHS of equation }(\ref{eq:rn_explicit})\right) & \leq & \prod_{i=1}^{d}\left(\frac{\tf}{\tf-t}\right)^{i-\half}\exp\left(\frac{\zf^{2}}{2\tf}\right)\\
 & \leq & \prod_{i=1}^{d}3^{i-\half}\exp\left(\frac{\zf^{2}}{2\tf}\right),
\end{eqnarray*}
where we have used finally that $\tf-t\geq\frac{1}{3}\tf$. Since
this is finite, there is some constant $C_{R}^{(\tf,\zf)}$ that acts
as an upper bound for all $N\in\bN$ as desired.\end{proof}
\begin{proposition}
\label{prop:ON_exp_mom} Recall the definition of the rescaled overlap
time $O^{(N),(\tf,\zf)}[0,t]$ from Definition \ref{def:overlap_times_NIW}.
For any $\tf>0$ and $\zf\in\bR$, the collection of rescaled overlap
times
\[
\left\{ O^{(N),(\tf,\zf)}[0,t],t\in[0,\tf]\right\} _{N\in\bN}
\]
 is exponential moment controlled as $t\to0$.\end{proposition}
\begin{proof}
We first show the result holds for $t \in[0,\half \tf]$, i.e. $\left\{ O^{(N),(\tf,\zf)}[0,t],t\in[0,\half\tf]\right\} _{N\in\bN}$ is exponential moment controlled. Indeed, for any $t<\half\tf$, the
Radon-Nikodym bound from Lemma \ref{lem:RN-deriv-bounded} shows that
for any $k\in\bN$ we have the bound: 
\[
\e\left[\left(O^{(N),(\tf,\zf)}[0,t]\right)^{k}\right]\leq C_{R}^{(\tf,\zf)}\e\left[\left(\frac{1}{\sqrt{N}}O[0,\floor{tN}]\right)^{k}\right].
\]
The exponential moment control follows from this bound by inspecting
the conditions for exponential moment control in Definition \ref{def:exp_mom_control}
and using the exponential moment of $\left\{ N^{-\half}O[0,tN],t\in[0,\half\tf]\right\}$ from Corollary \ref{cor:exp-moment-control-NIW}.

To extend from exponential moment control for $t\in[0,\half\tf]$
to exponential moment control on all $t\in[0,\tf]$, since $O^{(N),\left(\tf,\zf\right)}[0,t]$
is monotone non-deceasing in $t$, it suffices to check only the point $t=\tf$. This is verified by doing the following subdivision:
\begin{equation}
O^{(N),(\tf,\zf)}[0,\tf]\leq O^{(N),(\tf,\zf)}\left[0,\half\tf\right]+O^{(N),(\tf,\zf)}\left[\half\tf,\tf\right].\label{eq:overlap_split}
\end{equation}
By the symmetry $t\leftrightarrow\tf-t$ of the non-intersecting random
walk bridges, we have equality in distribution $O^{(N),(\tf,\zf)}\left[\half\tf,\tf\right]\dequal O^{(N),(\tf,\zf)}\left[0,\half\tf\right]$. By the exponential moment control of $\left\{ O^{(N),(\tf,\zf)}[0,t],t\in[0,\half\tf]\right\} _{n\in\bN}$, we see then that both terms on the RHS of equation (\ref{eq:overlap_split}) are exponential moment controlled and the desired result holds by Lemma \ref{lemma:sum-of-exp}.
\end{proof}

\section{$L^{2}$ Bounds From Overlap Times\label{sec:L2_bounds}}

The purpose of this section is to make the connection between $\norm{\ps_{k}^{(N),(\tf,\zf)}}_{L^{2}(\cS_{k}(0,\tf))}^{2}$
and the moments of the overlap times $O^{(N),(\tf,\zf)}$ introduced
and studied in Section \ref{sec:overlap_times}. This connection,
along with Lemma \ref{lem:exp-mom-moments} which relates exponential
moment control to bounds on moments, will allow us to prove the bounds
in Propositions \ref{prop:D3}, \ref{prop:D4} and \ref{prop:uniform_exp_moms}
as consequences of the exponential moment control of $O^{(N),(\tf,\zf)}$
proven earlier in Proposition \ref{prop:ON_exp_mom}.
\begin{lemma}
\label{lem:overlap_to_sum_bound} For $k \in\bN$, recall the definition of the simplex $\De_k(s,s^\prime)\subset(0,\tf)^k$ from Definition \ref{def:Z_partition}. Define also $\De^{(N)}_k(s,s^\prime) \defequal \De_k(s,s^\prime) \cap \left(\frac{\bN}{N}\right)^k$. We have the inequality: 
\begin{align}
&  \frac{1}{k!}\left(O^{\left(N\right),(\tf,\zf)}\left[s,s^{\prime}\right]\right)^{k} \label{eq:overlap_to_l2} \\
\geq & \frac{1}{N^\frac{k}{2}}\sum_{\vec{t}\in \De_{k}^{(N)}(s,s^{\prime})}\sum_{\vec{z}\in\frac{\bZ^{k}}{\sqrt{N}}}\one\left\{ \bigcap_{i=1}^{k}\left\{  z_{i}\in\X^{(N),(\tf,\zf)}(t_{i})\right\} \right\} \one\left\{ \bigcap_{i=1}^{k}\left\{  z_{i}\in\X^{\prime (N),(\tf,\zf)}(t_{i})\right\} \right\}.\nonumber 
\end{align}
\end{lemma}
\begin{proof} Recall from Definition \ref{def:overlap_times_NIW} that $O^{(N),(\tf,\zf)}$ is defined in terms of two independent copies of the bridges $\X^{(N),(\tf,\zf)}$, $\X^{\prime(N),(\tf,\zf)}$. Expanding this definition gives:
\begin{eqnarray}
 && \frac{1}{k!}\left(O^{\left(N\right)(\tf,\zf)}\left[s,s^{\prime}\right]\right)^{k}= \frac{1}{k!{N^\frac{k}{2}}}\left(\sum_{t\in\left(s,s^{\prime}\right)\cap\frac{\bN}{N}}\abs{\left\{ \X^{(N),(\tf,\zf)}(t)\cap\X^{\prime(N),(\tf,\zf)}(t)\right\} }\right)^{k}\nonumber \\
 && =\frac{1}{k!{N^\frac{k}{2}}}\left(\sum_{t\in\left(s,s^{\prime}\right)\cap\frac{\bN}{N}}\ \sum_{z\in\frac{\bZ}{\sqrt{N}}}\one\left\{ z\in\X^{(N),(\tf,\zf)}(t)\right\} \one\left\{ z\in\X^{\prime(N),(\tf,\zf)}(t)\right\} \right)^{k}\label{eq:overlap_bound}
\end{eqnarray}
The desired inequality follows by expanding the RHS of equation (\ref{eq:overlap_bound})
as a $k$-fold sum, and discarding the contribution from the indices
in the sum $\vec{t}=(t_{1},\ld,t_{k})$ that have $t_{i}=t_{\ell}$
for some $i\neq\ell$. In the remaining sum, we can switch from an
un-ordered $k$-fold sum to an ordered sum $\vec{t}\in \De^{(N)}_{k}(s,s^\prime)$ at
the cost of the factor $k!$, which gives the desired result.\end{proof}
\begin{corollary}
\label{cor:ON_moments}Have for $0<s<s^{\prime}<\tf$ that:
\[
\intop_{\bR^{k}} \intop_{\De_{k}(s,s^{\prime})} \abs{\ps_{k}^{(N),(\tf,\zf)}\big((t_{1},z_{1}),\ld,(t_{k},z_{k})\big)}^{2}\d\vec{t}\d\vec{z}\leq\e\left[\frac{1}{2^{k}k!}\left(O^{\left(N\right)(\tf,\zf)}\left[s,s^{\prime}\right]\right)^{k}\right].
\]
\end{corollary}
\begin{proof}
Taking $\e$ of the RHS of equation (\ref{eq:overlap_to_l2}) and
keeping in mind that $\vec{X}^{(N),(\tf,\zf)}$ and ${\vec{X}}^{\prime(N),(\tf,\zf)}$
are independent copies, we have:
\begin{align*}
\e\left[\text{RHS of (\ref{eq:overlap_to_l2})}\right] & =  2^k \left(\frac{2}{N\sqrt{N}}\right)^{k}\sum_{\vec{t}\in \De_{k}^{(N)}(s,s^{\prime})}\sum_{\vec{z}\in\frac{\bZ^{k}}{\sqrt{N}}}\left(\frac{N^\frac{k}{2}}{2^k}\p\left(\bigcap_{i=1}^{k}\left\{ z_{i}\in\X^{(N),(\tf,\zf)}(t_{i})\right\} \right)\right)^{2}\\
 & =  2^k \intop_{\bR^{k}} \intop_{\De_{k}(s,s^{\prime})} \abs{\ps_{k}^{(N),(\tf,\zf)}\left((t_{1},z_{1}),\ld,(t_{k},z_{k})\right)}^{2}\d\vec{t}\d\vec{z}.
\end{align*}

The last line follows since we recall from Definition \ref{def:scaled-NIWb}
that $\ps_{k}^{(N),(\tf,\zf)}$ is constant on the cells $\prod_{i=1}^{k}\big[\frac{n_i}{N},\frac{n_i}{N}+\frac{1}{N}\big)\times\big[\frac{x_i}{\sqrt{N}},\frac{x_i}{\sqrt{N}}+\frac{2}{\sqrt{N}}\big)$ of volume $\left(\frac{2}{N\sqrt{N}}\right)^{k}$,
so the integral is a sum over discrete cells in the same manner as explained in equation (\ref{eq:integral_is_sum}). Dividing by $2^k$ on both sides gives the desired result.\end{proof}
\begin{lemma}
\label{lem:exp-mom-moments} If $\left\{ Z^{(N)}(t)\ :\ t\in\left[0,\tf\right]\right\} _{N\in\bN}$ is a collection of non-negative valued processes that is exponential moment controlled as $t\to0$, then for each $t\in[0,\tf]$:
\[
\forall k\in\bN,\forall t>0\ \sup_{N\in\bN}\e\left[\left(Z^{(N)}(t)\right)^{k}\right]<\infty.
\]
Moreover, for any fixed $k\in\bN$, the $k$-th moment can be made arbitrarily
small by taking $t$ small enough:
\[
\forall k\in\bN,\ \lim_{t\to0}\sup_{N\in\bN}\e\left[\left(Z^{(N)}(t)\right)^{k}\right]=0.
\]
\end{lemma}
\begin{proof}
For any $t\in[0,\tf]$ and any $\ga>0$, we use the bound $x^{k}\leq\frac{k!}{\ga^{k}}e^{\ga x}$ for $x\geq0$ to see
\begin{equation}
\sup_{N\in\bN}\e\left[\left(Z^{(N)}(t)\right)^{k}\right] \leq \frac{k!}{\ga^{k}}\sup_{N\in\bN}\e\left[e^{\ga Z^{(N)}(t)}\right],
\end{equation}
which is finite by property i) of the exponential moment control from
Definition \ref{def:exp_mom_control}. Moreover, we can take the limit ${t\to 0}$ and apply property ii) of the same definition to see that:
\[
0\leq\lim_{t\to0}\sup_{N\in\bN}\e\left[\left(Z^{(N)}(t)\right)^{k}\right]\leq\frac{k!}{\ga^{k}}\lim_{t\to0}\sup_{N\in\bN}\e\left[e^{\ga Z^{(N)}(t)}\right]=\frac{k!}{\ga^{k}}.
\]
By taking $\ga$ arbitrarily large, we conclude that the
limit as $t\to0$ is actually $0$ as desired.
\end{proof}

\subsection{Bounds on  \texorpdfstring{$D_{3}(\de)$}{D3(de)}{\:\textendash\:}proof of Proposition \ref{prop:D3}}
\begin{proof}
(Of Proposition \ref{prop:D3}) Let $D_3^0(\de) = \cS_{k}(0,\tf)\cap \left\{ t_{1}\leq\de\right\}$ and let $D_3^\tf(\de) = \cS_{k}(0,\tf)\cap \left\{ t_{k}\geq\tf - \de\right\}$ so that $D_3(\de) = D_3^0(\de) \cup D_3^\tf(\de)$. It suffices to show that for that
for all $\ep>0$, there exists $\de>0$ so that:
\begin{equation}
\sup_{N\in\bN}\intop_{D_{3}^{0}(\de)}\abs{\ps^{(N),(\tf,\zf)}(\vec{w})}^{2}\d\vec{w} <  \ep, \label{eq:D3_0j_target} \quad \sup_{N\in\bN}\intop_{D_{3}^{\tf}(\de)}\abs{\ps^{(N),(\tf,\zf)}(\vec{w})}^{2}\d\vec{w}  < \ep,
\end{equation}
since once this is proven we can use a union bound to complete the result. It suffices to prove the inequality in equation (\ref{eq:D3_0j_target}) for $D_{3}^{0}$ only, as the result for
$D_{3}^{\tf}$ follows by the symmetry $t\leftrightarrow\tf-t$.
We first claim:
\begin{equation}
\intop_{D_{3}^{0}(\de)}\abs{\ps_{k}^{(N),(\tf,\zf)}(\vec{w})}^{2}\d\vec{w}\leq\e\left[\frac{1}{2}\left(O^{(N),(\tf,\zf)}[0,\de]\right) \frac{1}{2^{k-1}(k-1)!}\left(O^{(N),(\tf,\zf)}[\de,\tf]\right)^{k-1}\right].\label{eq:bound_on_D3_0j}
\end{equation}
Equation (\ref{eq:bound_on_D3_0j}) is verified in the same way
as the proof of Corollary \ref{cor:ON_moments}: first using Lemma \ref{lem:overlap_to_sum_bound}
to bound the RHS of equation (\ref{eq:bound_on_D3_0j}) as a sum over
indicator function, and then recognizing by the definition of $D_{3}^{0}(\de)$
that this sum coincides with $\intop_{D_{3}^{0}(\de)}\abs{\ps_{k}^{(N),(\tf,\zf)}(\vec{w})}^{2}\d\vec{w}$.
With equation (\ref{eq:bound_on_D3_0j}) established, we use the fact
that $O^{(N),(\tf,\zf)}[0,\tf]\geq O^{(N),(\tf,\zf)}[\de,\tf]$ and
the Cauchy-Schwarz inequality to see that
\begin{equation}
\intop_{\mathclap{D_{3}^{0}(\de)}}\abs{\ps_{k}^{(N),(\tf,\zf)}(\vec{w})}^{2}\d\vec{w} \leq \frac{2^{-k}}{(k-1)!}\sqrt{\e\left[\left(O^{(N),(\tf,\zf)}[0,\de]\right)^{2}\right]\e\left[\left(O^{(N),(\tf,\zf)}[0,\tf]\right)^{2(k-1)}\right]}. \nonumber
\end{equation}
We now use the exponential moment control of
$\left\{ O^{(N),(\tf,\zf)}[0,t]:t\in(0,\tf)\right\} _{N\in\bN}$ from
Proposition \ref{prop:ON_exp_mom}. By Lemma \ref{lem:exp-mom-moments},
this implies that $\sup_{N\in\bN}\e\left[\left(O^{(N),(\tf,\zf)}[0,\tf]\right)^{2(k-1)}\right]<\infty$
is some finite constant. Also by Lemma \ref{lem:exp-mom-moments} we have that $\sup_{N}\e\left[\left(O^{(N),(\tf,\zf)}[0,\de]\right)^{2}\right]\to 0$ as $\de \to 0$. Thus the limit as $\de \to 0$ of the LHS of
equation (\ref{eq:D3_0j_target}) is 0, and we can thus choose $\de>0$ small enough to bound above by $\ep$, as desired.
\end{proof}

\subsection{Bounds on  \texorpdfstring{$D_{4}(M)$}{D4(M)}{\:\textendash\:}proof of Proposition \ref{prop:D4}}
\begin{proof}
(Of Proposition \ref{prop:D4}) Define $W^{(N),(\tf,\zf)}\defequal\max_{i\in\left\{ 1,\ld,d\right\} }\sup_{t\in[0,\tf]}\abs{X_{i}^{(N),(\tf,\zf)}(t)}$
to be largest absolute value achieved by the ensemble at any time
$t\in[0,\tf]$, and let ${W}^{\prime(N),(\tf,\zf)}$ be the same for
the independent copy ${\X}^{\prime(N),(\tf,\zf)}$. Writing the integral
over $D_{4}(M)$ in equation (\ref{eq:D4_target}) as the expectation
of a sum of indicator functions in the same way as was done in  Corollary \ref{cor:ON_moments},
by the definition of $D_{4}(M)$ we have:
\begin{align*}
 & \text{LHS equation (\ref{eq:D4_target}}) \\
 =& \frac{1}{2^k N^\frac{k}{2}}\e\bigg[\sum_{\stackrel{\left(\vec{z},\vec{t}\right)\in\left(\bT^{(N)}\right)^{k},}{\bigcup_{i=1}^{k}\left\{ \abs{z_{i}}>M\right\}}}  \one{\bigg\{ \bigcap_{i=1}^{k}\left\{  z_{i}\in\X^{(N),(\tf,\zf)}(t_{i})\right\}\bigg\}} \one{\bigg\{  \bigcap_{i=1}^{k}\Big\{ z_{i}\in\X^{\prime(N),(\tf,\zf)}(t_{i})\Big\} \bigg\}} \bigg]\\
  \leq & \frac{1}{2^k N^\frac{k}{2}} \e\Bigg[\one\left\{ W^{(N),(\tf,\zf)}>M\right\} \one\left\{ {W}^{\prime(N),(\tf,\zf)}>M\right\} \Bigg.\\
   & \ \ \ \ \ \ \ \ \Bigg.\times \sum_{\left(\vec{z},\vec{t}\right)\in\left(\bT^{(N)}\right)^{k}}\one\bigg\{ \bigcap_{i=1}^{k}\left\{ z_{i}\in\X^{(N),(\tf,\zf)}(t_{i})\right\} \bigg\} \one\bigg\{ \bigcap_{i=1}^{k}\left\{ z_{i}\in\X^{\prime (N),(\tf,\zf)}(t_{i})\right\} \bigg\} \Bigg].
\end{align*}
This inequality follows by inclusion since if $\abs{z_{i}}>M$ and
$z_{i}\in\vec{X}^{(N),(\tf,\zf)}(t_{i})$, then certainly the maximum has $W^{(N),(\tf,\zf)}>M$. By application of Lemma \ref{lem:overlap_to_sum_bound}
and Cauchy-Schwarz we have:
\begin{align}
\text{LHS (\ref{eq:D4_target}}) & \leq \e\left[\one\left\{ W^{(N),(\tf,\zf)}>M\right\} \one\left\{ {W}^{\prime(N),(\tf,\zf)}>M\right\} \frac{1}{2^{k}k!}\left(O^{\left(N\right)(\tf,\zf)}\left[0,\tf\right]\right)^{k}\right]\nonumber \\
 & \leq \frac{1}{2^{k}k!}\p\left(W^{(N),(\tf,\zf)}>M\right)\sqrt{\e\left[\left(O^{\left(N\right)(\tf,\zf)}\left[0,\tf\right]\right)^{2k}\right]}.\label{eq:D4_post_CS}
\end{align}
Finally, by the exponential moment control from Proposition \ref{prop:ON_exp_mom}
and Lemma \ref{lem:exp-mom-moments} we have $\sup_{N\in\bN}\e\left[\left(O^{\left(N\right)(\tf,\zf)}\left[0,\tf\right]\right)^{2k}\right]<\infty$.
By this bound and $\sup_{N\in\bN}\p\left(W^{(N),(\tf,\zf)}>M\right)$ goes to $0$ as $M \to \infty$ (this is justified for example by the argument in Lemma \ref{lem:exp-mom-for-location})
it is possible to choose an $M$ so large so that the RHS of equation
(\ref{eq:D4_post_CS}) is less than $\ep$, as desired. 
\end{proof}

\subsection{Proof of Proposition \ref{prop:uniform_exp_moms}}
\begin{proof}
(Of Proposition \ref{prop:uniform_exp_moms}) As explained in equation (\ref{eq:S_k_int}), the integral over $\cS_k(0,\tf)$ and the integral over $((0,\tf)\times\bR)^k$ differ by a factor of $k!$; we will find it more convenient to work with $\cS_k(0,\tf)$ for this proof. Noticing that $\De_k(0,\tf)\times\bR^k$ is in natural bijection with $\cS_k(0,\tf)$, we have by Corollary \ref{cor:ON_moments}
applied to each term of the sum on the LHS of equation (\ref{eq:uniform_exp_moms}) that
\begin{align}
\sup_{N\in\bN}\sum_{k=\ell}^{\infty}\ga^{k}\intop_{\cS_{k}(0,\tf)}\abs{\ps_{k}^{(N),(\tf,\zf)}(\vec{w})}^{2}\d\vec{w} & \leq \sup_{N\in\bN}\sum_{k=\ell}^{\infty}\ga^{k}\frac{1}{2^{k}k!}\e\left[\left(O^{\left(N\right)(\tf,\zf)}\left[0,\tf\right]\right)^{k}\right]\nonumber \\
 & = \sup_{N\in\bN}\e\left[\sum_{k=\ell}^{\infty}\frac{1}{k!}\left(\frac{\ga}{2}\right)^{k}\left(O^{\left(N\right)(\tf,\zf)}\left[0,\tf\right]\right)^{k}\right].\label{eq:L2_unif}
\end{align}
The interchange of expectation with the infinite sum is justified
by the monotone convergence theorem since $O^{(N),(\tf,\zf)}[0,\tf]$
is non-negative. Finally, if we take the limit  ${\ell\to\infty}$, then the RHS of equation (\ref{eq:L2_unif})
goes to $0$ by property iii) of exponential moment control from Definition
\ref{def:exp_mom_control} since $\left\{ O^{\left(N\right)(\tf,\zf)}\left[0,t\right]:t\in[0,\tf]\right\} _{N\in\bN}$
is exponential moment controlled by Proposition \ref{prop:ON_exp_mom}. 
\end{proof}




\providecommand{\bysame}{\leavevmode\hbox to3em{\hrulefill}\thinspace}
\providecommand{\MR}{\relax\ifhmode\unskip\space\fi MR }
\providecommand{\MRhref}[2]{%
  \href{http://www.ams.org/mathscinet-getitem?mr=#1}{#2}
}
\providecommand{\href}[2]{#2}

%
%
%
%
%
%
%


\ACKNO{IC thanks Jeremy Quastel for helpful discussions. MN thanks his advisor G{\'e}rard Ben Arous for his constant support and also Vadim Gorin, Wolfgang K{\"o}nig, Neil O'Connell and Mykhaylo Shkolnikov for their friendly email responses to queries related to this project. Thanks also to an anonymous reviewer for a very careful reading of the paper which led to many improvements throughout. IC was partially supported by the NSF through grant DMS-1208998 and grant PHY11-25915 as well as by the Clay Mathematics Institute through the Clay Research Fellowship, by the Institute Henri Poincare through the Poincare Chair, and by the Packard Foundation through a Packard Fellowship in Science and Engineering. MN was partially supported by the NSF through grant DMS-1209165 as well as by the MacCracken Fellowship from New York University.}


\end{document}